\documentclass[11pt,a4paper,reqno]{amsart}
\usepackage{amsfonts}\usepackage{amsmath,amssymb,amsthm,amsxtra}
\allowdisplaybreaks[2]
\usepackage{aliascnt}
\usepackage[dvipsnames]{xcolor}
\usepackage[colorlinks, linkcolor=RoyalBlue,anchorcolor=Periwinkle,
citecolor=Orange,urlcolor=Emerald]{hyperref}
\usepackage{enumitem}
\usepackage{geometry,array}
\usepackage{graphicx}
\usepackage{bookmark}
\usepackage{tikz}\usetikzlibrary{matrix}
\usepackage{longtable}
\usepackage{cleveref}
\usepackage{tikz-cd}
\usetikzlibrary{calc}
\tikzset{curve/.style={settings={#1},to path={(\tikztostart)
			.. controls ($(\tikztostart)!\pv{pos}!(\tikztotarget)!\pv{height}!270:(\tikztotarget)$)
			and ($(\tikztostart)!1-\pv{pos}!(\tikztotarget)!\pv{height}!270:(\tikztotarget)$)
			.. (\tikztotarget)\tikztonodes}},
	settings/.code={\tikzset{quiver/.cd,#1}
		\def\pv##1{\pgfkeysvalueof{/tikz/quiver/##1}}}, quiver/.cd,pos/.initial=0.35,height/.initial=0}
\newcommand{\mK}{\mathbb K}
\newcommand{\mS}{\mathbb S}
\newcommand{\mB}{\mathcal B}
\newcommand{\mC}{\mathcal C}
\newcommand{\mD}{\mathcal D}
\newcommand{\mI}{\mathcal I}
\newcommand{\mL}{\mathcal L}
\newcommand{\mR}{\mathcal R}

\newcommand{\Hom}{\operatorname{Hom}}
\newcommand{\HH}{\operatorname{HH}}
\newcommand{\id}{\operatorname{id}}
\newcommand{\Sub}{\operatorname{Sub}}
\newcommand{\Sp}{\operatorname{Sp}}
\newcommand{\Int}{\operatorname{Int}}
\newcommand{\Occ}{\operatorname{Occ}}
\newcommand{\para}[2]{(#1,#2)}
\newcommand{\Para}[2]{(#1\,||\,#2)}
\newcommand{\num}[2]{\operatorname{Num}_{#1}(#2)}
\newcommand{\ep}{\varepsilon}
\newcommand{\sgn}{\operatorname{sgn}}
\newcommand{\one}{\mathbf 1}
\newcommand{\taue}{\tau_{\ell(\tau)}}
\newcommand{\bi}{\beta_{i_0}}
\newcommand{\ei}{e_{i_0}}
\newcommand{\epi}{\varepsilon_{i_0}}
\newcommand{\bet}{\beta_{\ell(b)}}
\newcommand{\gj}{\gamma_{j_0}}
\theoremstyle{plain}
\newtheorem{theorem}{Theorem}[section]

\newaliascnt{lemma}{theorem}
\newtheorem{lemma}[lemma]{Lemma}
\aliascntresetthe{lemma}

\newaliascnt{corollary}{theorem}
\newtheorem{corollary}[corollary]{Corollary}
\aliascntresetthe{corollary}

\newaliascnt{proposition}{theorem}
\newtheorem{proposition}[proposition]{Proposition}
\aliascntresetthe{proposition}

\theoremstyle{definition}
\newaliascnt{definition}{theorem}
\newtheorem{definition}[definition]{Definition}
\aliascntresetthe{definition}

\newaliascnt{example}{theorem}
\newtheorem{example}[example]{Example}
\aliascntresetthe{example}

\newaliascnt{remark}{theorem}
\newtheorem{remark}[remark]{Remark}
\aliascntresetthe{remark}

\numberwithin{equation}{section}

\newtheorem*{maintheorema}{Theorem A}
\newtheorem*{maintheoremb}{Theorem B}
\newtheorem*{maintheoremc}{Theorem C}
\crefname{theorem}{theorem}{theorems}
\Crefname{theorem}{Theorem}{Theorems}
\crefname{lemma}{lemma}{lemmas}
\Crefname{lemma}{Lemma}{Lemmas}
\crefname{corollary}{corollary}{corollaries}
\Crefname{corollary}{Corollary}{Corollaries}
\crefname{proposition}{proposition}{propositions}
\Crefname{proposition}{Proposition}{Propositions}
\crefname{definition}{definition}{definitions}
\Crefname{definition}{Definition}{Definitions}
\crefname{remark}{remark}{remarks}
\Crefname{remark}{Remark}{Remarks}
\crefname{example}{example}{examples}
\Crefname{example}{Example}{Examples}
\crefname{equation}{equation}{equations}
\Crefname{equation}{Equation}{Equations}
\crefname{appendix}{appendix}{appendices}
\Crefname{appendix}{Appendix}{Appendices}
\title[Transferred $L_\infty$-structures and binary brackets]
{Transferred $L_\infty$-structures and a classification of binary brackets for finite-dimensional skew-gentle algebras}

\author{Zixu Li}
\address{Yancheng Institute of Technology, 224007 Yancheng, Jiangsu, P. R. China}
\email{lizx19@tsinghua.org.cn}

\author{Liyu Liu}
\address{School of Mathematics, Yangzhou University, 225009 Yangzhou, Jiangsu, P. R. China}
\email{lyliu@yzu.edu.cn}

\author{Lingchao Meng}
\address{School of Mathematics, Yangzhou University, 225009 Yangzhou, Jiangsu, P. R. China}
\email{bcmenglc@outlook.com}

\date{}


\begin{document}
	
	\begin{abstract}
		We construct a transferred $L_\infty$-structure on the shifted Chouhy--Solotar cochain complex of a finite-dimensional skew-gentle algebra and give explicit basis-level formulas for its binary operation $l_2$ in all shifted degrees. We further prove that the Hochschild dg Lie algebra of an $A_n$-type skew-gentle algebra is homotopy abelian, and hence $L_\infty$-formal, and contrast this with Wong's non-formality result for the dual numbers $\mK[\theta]/\langle \theta^2 \rangle$, leading to a general formality classification problem.
	\end{abstract}
	
	\maketitle
	\tableofcontents
	\bigskip
	
	%===========================
	\section{Introduction}
	
	Hochschild cohomology is one of the fundamental homological invariants of an
	associative algebra.  Beyond its structure as a graded vector space, it
	carries the cup product and the Gerstenhaber bracket, which together form a
	Gerstenhaber algebra \cite{Ger63}.  At the cochain level, the shifted Hochschild cochain complex
	$C^*(A)[1]$, equipped with the Hochschild differential and the shifted
	Gerstenhaber bracket, is a differential graded (dg) Lie algebra.  We refer to it as the Hochschild dg Lie algebra.  This algebra controls the deformation
	theory of $A$: infinitesimal deformations, the Maurer--Cartan equation, and
	obstruction theory are naturally expressed in its terms; see, for example,
	\cite{Ger63,Wit19}.  The question underlying the present work is therefore
	not only what $\HH^*(A)$ and its Gerstenhaber bracket look like, but how much
	of the chain-level Lie-theoretic structure of $C^*(A)[1]$ survives, up to
	higher homotopy, on Hochschild cohomology itself.
	
	The natural language for this question is that of $L_\infty$-algebras.  The
	origins of higher homotopy algebra go back to Stasheff's work on homotopy
	associativity and $A_\infty$-structures \cite{Sta63}.  Higher-homotopical Lie
	structures subsequently emerged in deformation theory, notably in the work
	of Schlessinger--Stasheff \cite{SS85}, and were developed in the language of
	strongly homotopy Lie, or $L_\infty$, algebras by Lada--Stasheff and
	Lada--Markl \cite{LS93,LM95}.  An $L_\infty$-algebra
	$(L,\{l_n\}_{n\geq1})$ replaces the strict identities of a dg Lie algebra
	by a coherent hierarchy of higher homotopies.  In particular, the
	arity-three $L_\infty$-identity relates the Jacobiator of $l_2$ to $l_1$ and
	$l_3$, while the subsequent identities govern the compatibility of the
	higher operations.
	
	For explicit calculations on quiver algebras, this higher-homotopical
	viewpoint is closely tied to the existence of small projective bimodule
	resolutions.  Bardzell constructed such a resolution for monomial algebras
	\cite{Bar97}.  Comparison morphisms between Bardzell's resolution and the bar
	resolution were made explicit by Redondo--Rom\'an \cite{RR18a}, who used them
	to compute the cup product and the Gerstenhaber bracket for quadratic string
	algebras \cite{RR18b}; see also Volkov's construction of the Gerstenhaber
	bracket via arbitrary projective resolutions \cite{Vol19}.  A complementary
	higher-homotopical approach to monomial algebras was developed by Tamaroff,
	who constructed explicit minimal models and the associated canonical
	$A_\infty$-structures on Ext-algebras \cite{Tam21}.
	
	The Hochschild dg Lie structure was transferred to a small cochain complex by
	Redondo--Rossi Bertone \cite{RR22}, who constructed an explicit
	$L_\infty$-structure on the shifted Bardzell complex of a monomial algebra
	together with a weak $L_\infty$ quasi-isomorphism to the Hochschild complex.
	Building on this construction, M\"uller--Redondo--Rossi Bertone--Suarez
	studied finite-dimensional gentle algebras and, under additional hypotheses
	on the quiver, obtained nilpotency results for the transferred
	$L_\infty$-structure and described Maurer--Cartan elements in the nilpotent
	cases \cite{MRRS25}.  Thus, already for gentle algebras, the small transferred
	model retains deformation-theoretic information that is invisible from the
	cohomology groups alone.
	
	Within this framework, Hochschild $L_\infty$-formality asks whether the full
	chain-level dg Lie algebra can be replaced, up to $L_\infty$
	quasi-isomorphism, by the much smaller graded Lie algebra carried by Hochschild
	cohomology.  More precisely, after transferring to a minimal model on
	$\HH^*(A)[1]$ and writing $l_r^{\min}$ for its operations, one has
	$l_1^{\min}=0$ and $l_2^{\min}=[-,-]_{\HH}$, while the operations
	$l_r^{\min}$ for $r\geq3$ record the genuinely higher homotopical information
	that survives on cohomology.  Hochschild $L_\infty$-formality means that the
	minimal model is $L_\infty$-isomorphic to one with $l_r^{\min}=0$ for all
	$r\geq3$, whereas homotopy abelianity means that it can be chosen abelian; in
	particular, the induced Gerstenhaber bracket then vanishes.  Precise
	definitions are given in \Cref{def:formality}.
	
	Formality has deep roots in rational homotopy theory.  A classical landmark
	is the theorem of Deligne--Griffiths--Morgan--Sullivan for compact K\"ahler
	manifolds \cite{DGMS75}.  A prominent formality result in a Hochschild-type
	setting is Kontsevich's formality theorem \cite{Kon03}.  Hochschild
	$L_\infty$-formality is not automatic for associative algebras: non-formality
	phenomena are known, for instance, for certain universal enveloping algebras
	\cite{BEGM23}.  Closer to the setting of the present paper,
	Bocklandt--van de Kreeke proved formality of the Hochschild dg Lie algebra for
	a class of gentle $A_\infty$-algebras associated with arc collections having
	no loops or two-cycles \cite{BK24}.
	On the other hand, Wong proved that the Hochschild dg Lie algebra
	of the exterior algebra on a one-dimensional vector space is not
	formal \cite[Theorem~6.3]{Won17}. As an ungraded associative algebra,
	this exterior algebra is isomorphic to
	$\mK[\theta]/\langle\theta^2\rangle$, giving a finite-dimensional
	gentle example with non-formal Hochschild dg Lie algebra.  Thus Hochschild
	$L_\infty$-non-formality does occur in the finite-dimensional gentle setting.
	
	We now turn to the classes of algebras considered here.  Gentle algebras were
	introduced by Assem--Skowro\'nski \cite{AS87} and have since become a central
	class of tame algebras with rich representation-theoretic, derived, and
	geometric structures; see, for example, \cite{AG08}.  Their Hochschild theory
	has been developed progressively at several levels.  Ladkani computed the
	dimensions of the Hochschild cohomology groups and showed that they are
	determined by the Avella-Alaminos--Gei\ss{} derived invariant \cite{Lad12}.
	For quadratic string algebras, and hence in particular for gentle algebras,
	Redondo--Rom\'an described Hochschild cohomology in terms of Bardzell's
	resolution and obtained explicit cup-product and Gerstenhaber-bracket formulas
	\cite{RR18b}.  In a surface-theoretic direction, Valdivieso-D\'iaz gave
	geometric computations of Hochschild cohomology for Jacobian gentle algebras
	arising from unpunctured surfaces \cite{VD15}.  Chaparro--Schroll--Solotar
	described the Lie algebra structure of the first Hochschild cohomology of
	gentle algebras and related it to the associated ribbon graph \cite{CSS20}.
	More recently, Chaparro--Schroll--Solotar--Su\'arez-\'Alvarez computed the
	complete Tamarkin--Tsygan calculus of gentle algebras, including the
	Hochschild cohomology ring, its Gerstenhaber structure, Hochschild and cyclic
	homology, and their geometric interpretation \cite{CSSSA26}.  For graded
	gentle algebras, Opper described the bigraded Hochschild cohomology together
	with the cup product and Gerstenhaber bracket and obtained a criterion for
	intrinsic formality \cite{Opp26}.  We stress that intrinsic formality in that
	work concerns $A_\infty$-enhancements of a graded algebra and is distinct from
	the Hochschild $L_\infty$-formality considered here.
	
	Skew-gentle algebras were introduced by Gei\ss--de la Pe\~na \cite{GP99}.
	Their derived categories admit an orbifold-surface interpretation
	\cite{LFSV22}, making them a natural extension of the gentle setting in which
	Hochschild structures can also be compared with geometry.  In the
	skew-gentle setting, Bian--Schroll--Solotar--Wang--Wen computed the Hochschild
	cohomology and Gerstenhaber algebra structure of graded skew-gentle algebras
	and related these structures to the underlying graded orbifold surface
	\cite{BSSWW26}.  Crucially for the present paper, they presented a graded
	version of the Chouhy--Solotar (CS) projective resolution \cite{CS15},
	constructed comparison morphisms, transported the Gerstenhaber bracket from
	the bar complex to the CS cochain complex, and derived explicit insertion
	formulas on basis cochains.  Their Hochschild basis theorem is also the key
	input in our $A_n$-type formality result below.
	
	The skew-gentle case is not a direct repetition of the monomial gentle case.
	The essential new feature is the double role played by a special loop.  If
	$\ep_i$ is a special loop, then $\ep_i^2\in I^S$ in the auxiliary monomial
	relation set used in the CS resolution, whereas multiplication in the actual
	algebra satisfies $\ep_i^2=\ep_i$.  Thus $\ep_i^2$ behaves as a
	monomial relation in the combinatorics of the CS resolution, while $\ep_i$ is
	an idempotent in the algebra itself.  Relation combinatorics is therefore
	monomial-like, but cochain values are multiplied in an algebra in which
	special-loop powers are reduced using the relation $\ep_i^2=\ep_i$.  This distinction must be
	tracked throughout the binary classification, particularly in path
	replacements and endpoint terms involving special loops, where idempotent
	absorption can occur.
	
	The first objective is to realize the Hochschild higher Lie structure
	on the CS cochain complex. The CS-projective resolution and comparison
	morphisms of \cite{BSSWW26} are combined with a relative version of
	the recursive homotopy construction in \cite{RR22}. The contraction
	identities and side conditions are verified in the skew-gentle
	setting, yielding the transferred $L_\infty$-structure in
	Theorem~A below. After translating between the grading, shift, and path-order conventions, the
	transferred binary operation is precisely the Gerstenhaber bracket transported
	to the CS complex in \cite[Section~3.6]{BSSWW26}.  The point is therefore not
	to introduce another binary bracket, but to place the small CS complex inside
	the full $L_\infty$ framework and thereby make the higher transfer recursion
	available on a path-combinatorial model.  Throughout the paper, the term
	binary bracket refers to the graded skew-symmetric binary operation
	$l_2$; it is not assumed to satisfy the strict graded Jacobi identity unless
	this is stated explicitly.
	
	The second, and computationally central, objective is to pass from the local
	insertion formulas to explicit basis-level formulas for the final binary
	bracket.  Proposition~3.25 of \cite{BSSWW26} determines the individual
	insertions, but the nonvanishing of $l_2$ is not decided by either insertion
	direction separately.  One must combine the two circle products, determine
	their signs and cancellations, evaluate the resulting path products in $A$
	via $\pi$, and account for repeated occurrences of the same support.  We
	organize this problem by the ordered pair of shifted input degrees.  Up to
	graded skew-symmetry, the resulting four types are $(-1,n)$ with $n\geq-1$,
	$(0,0)$, $(0,+)$, and $(+,+)$, where $+$ denotes an arbitrary positive
	shifted degree.  The four types exhaust all shifted degree pairs, and in each
	type we obtain explicit basis-level formulas for the bracket.
	
	The computation combines local path replacement and substitution, a global
	classification of repeated supports, and the analysis of signed occurrences
	in the periodic cases.  Locally, one must also account for the endpoint and
	interior terms produced by the comparison maps and reduce special-loop powers
	using their idempotent relations in $A$.  Globally, finite-dimensionality and
	the gentle uniqueness conditions force every relation concatenation with
	repeated arrows into one of three periodic forms: $\omega^N$, a power
	$\tau^w$ of a primitive relation cycle, or $\tau^w\tau_1\cdots\tau_r$ with a
	proper suffix; all remaining supports contain no repeated arrows.  In the
	periodic cases we determine every relevant occurrence position and its sign,
	so that multiple local contributions can be assembled with their correct
	multiplicities and cancellations.  This converts the transported bracket from an insertion formula into a finite explicit procedure for arbitrary basis	cochains and provides the concrete low-order input needed to analyse the first transferred higher operations and the nontrivial homotopy governing the graded Jacobi relation for $l_2$.
	
	The final part of the paper turns from the binary classification to the
	formality problem and reveals two contrasting phenomena. On the positive
	side, every $A_n$-type skew-gentle algebra, with arbitrary orientation of the underlying $A_n$-graph, has a homotopy abelian Hochschild dg Lie algebra and is therefore Hochschild $L_\infty$-formal.
	At the same time, if a special vertex is not an endpoint, the transferred CS model carries a nonzero chain-level ternary operation, which provides a nontrivial homotopy governing the graded Jacobi relation for $l_2$.	Thus, already within the $A_n$ family, nontrivial higher operations on a transferred model coexist with a trivial minimal	$L_\infty$-structure.
	
	By contrast, the dual numbers
	$D=\mK[\theta]/\langle\theta^2\rangle$, recalled above, provide a
	finite-dimensional gentle algebra with non-formal Hochschild dg Lie
	algebra. The coexistence of formal and non-formal examples motivates
	a combinatorial classification of Hochschild $L_\infty$-formality
	in the finite-dimensional gentle/skew-gentle setting. The aim is to
	find necessary and sufficient conditions on the quiver, the quadratic
	relations, and the special vertices, distinguishing among the
	homotopy abelian, formal but not homotopy abelian, and non-formal cases.
	
	We now state the main results more formally.  In the statements below, $\mK$
	is a field of characteristic zero, and each skew-gentle algebra is considered
	together with a skew-gentle triple $(Q,I,S)$.
	
	\begin{maintheorema}[Theorem~\ref{thm:transferred-Linfty}]
		Let $A$ be a finite-dimensional skew-gentle algebra.  The shifted CS cochain
		complex $B_{\mathrm{CS}}^*(A)[1]$ carries a transferred $L_\infty$-structure.
		Moreover, $G^*$ extends to a weak $L_\infty$
		quasi-isomorphism
		\[
		\phi:B_{\mathrm{CS}}^*(A)[1]\longrightarrow C^*(A)[1].
		\]
		In particular,
		\[
		l_1=\delta,\qquad
		l_2(f\otimes_{\mK}g)=F^*[G^*f,G^*g].
		\]
	\end{maintheorema}
	
	Theorem~A establishes the small-model framework in which the Hochschild
	higher Lie structure can be studied explicitly.  The main computational
	result is an explicit basis-level treatment of the binary stage.
	
	\begin{maintheoremb}
		Let $A$ be a finite-dimensional skew-gentle algebra and let
		$B=B_{\mathrm{CS}}^*(A)[1]$.  For arbitrary basis cochains $f\in B^m$ and
		$g\in B^n$, where $m,n\geq-1$, the computation of the transferred binary
		bracket $l_2(f\otimes_{\mK}g)$ falls, after interchanging the two inputs by
		graded skew-symmetry if necessary, into exactly one of the following four
		types.
		
		\begin{enumerate}[label=\textnormal{(\arabic*)}]
			\item
			For type $(-1,n)$, where $n\geq-1$, the bracket is given by
			\Cref{prop:degree-minus-one-brackets}.
			
			\item
			For type $(0,0)$, the bracket is given by
			\Cref{prop:degree-zero-classification}.
			
			\item
			For type $(0,+)$, the bracket is given by
			\Cref{prop:mixed-classification}.
			
			\item
			For type $(+,+)$, the bracket is evaluated by
			\Cref{prop:positive-positive-classification}.
		\end{enumerate}
	\end{maintheoremb}
	Thus Theorem~B gives a complete basis-level classification of the
	transferred binary bracket and a finite procedure for computing
	$l_2(f\otimes_{\mK}g)$ in every shifted degree.
	
	\begin{maintheoremc}[Theorems~\ref{thm:An-formality} and~\ref{thm:An-interior-special}]
		Let $A$ be an $A_n$-type skew-gentle algebra. Then the Hochschild dg Lie
		algebra $C^*(A)[1]$ is homotopy abelian; in particular, $A$ is Hochschild $L_\infty$-formal. If $A$ has a special vertex that is not an endpoint,
		then the transferred ternary operation on $B_{\mathrm{CS}}^*(A)[1]$ is
		nonzero, and the graded Jacobi relation for $l_2$ is governed by a
		nontrivial ternary homotopy.
	\end{maintheoremc}
	
	Theorem~C exhibits, already within the $A_n$ family, the separation between
	chain-level higher operations and minimal-model formality.  Together with
	Wong's non-formal gentle example recalled above, it shows that both formal and
	non-formal behaviour occur within the finite-dimensional gentle/skew-gentle
	framework.
	
	The paper is organized as follows. In \Cref{sec:preliminaries} we recall
	skew-gentle algebras, the normalized relative Hochschild complex, and the
	$L_\infty$ conventions and homotopy transfer theorem used below. In
	\Cref{sec:main-formulas} the transferred $L_\infty$-structure on
	the shifted CS cochain complex is constructed and the general formula for its binary
	operation is derived. In
	\Cref{sec:degree-minus-one-classification,sec:degree-zero-classification,sec:mixed-classification,sec:positive-positive-classification}
	we give the basis-level classification of the transferred binary
	bracket in the four degree types $(-1,n)$, $(0,0)$, $(0,+)$, and $(+,+)$.  The
	corresponding classifications are summarized in tabular form in
	\Cref{app:degree-minus-one-lookup,app:degree-zero-lookup,app:mixed-lookup,app:positive-positive-lookup}. In \Cref{sec:higher-operations-formality} we turn from the binary
	classification to higher operations and the formality problem.  The homotopy abelian property for the $A_n$-type is proved, and nontrivial
	transferred ternary operations are exhibited when a special vertex is not
	an endpoint. Finally, we employ Wong's work to give a non-formal gentle example \cite{Won17}.
	%===========================

	\section*{Declaration of generative AI and AI-assisted technologies in the manuscript preparation process}
	
	All computations and proofs in this work were carried out
	and completed by the human authors. During the preparation of the
	manuscript, the authors used ChatGPT (OpenAI) as an auxiliary tool for
	additional checks of mathematical consistency and correctness, as well as
	for language editing and improving readability. All AI-assisted suggestions
	were subsequently reviewed, verified, and, where necessary, revised by the
	human authors. The three authors take full responsibility for the
    correctness, originality, and content of the manuscript.
    %============================
    	\subsection*{Acknowledgments}
    This work is supported by National Natural Science Foundation of China (Grant No. 12601043).
    
	%===========================
	\section{Preliminaries}\label{sec:preliminaries}
	
	Throughout the paper, $\mK$ is a field of characteristic zero. All algebras are finite-dimensional over $\mK$ unless explicitly stated otherwise. All graded objects are cohomologically graded, and $|x|$ denotes the degree of a homogeneous element $x$ of a graded vector space. 
	%===========================
	\subsection{Quivers and skew-gentle algebras}
	
	We use standard terminology for quivers and path algebras. A quiver is a quadruple
	$Q=(Q_0,Q_1,s,t)$, where $Q_0$ is the set of vertices, $Q_1$ is the set of
	arrows, and $s,t:Q_1\to Q_0$ are the source and target maps. All quivers in
	this paper are finite and connected. For $n\geq0$, let $Q_n$ denote the set
	of paths of length $n$. We compose paths from left to right: a path
	$a\in Q_n$ of positive length is written
	$a=\alpha_1\cdots\alpha_n$, with $t(\alpha_i)=s(\alpha_{i+1})$ for
	$1\leq i<n$. Its source, target, and length are denoted by
	$s(a)=s(\alpha_1)$, $t(a)=t(\alpha_n)$, and $\ell(a)=n$, respectively. For
	each $i\in Q_0$, the trivial path at $i$ is denoted by $e_i$ and has length
	zero. An arrow $\alpha$ is a \emph{loop} if $s(\alpha)=t(\alpha)$. Two paths
	$p$ and $q$ are \emph{parallel}, written $p\parallel q$, if
	$s(p)=s(q)$ and $t(p)=t(q)$.
	
	A path $c=c_1\cdots c_d$ of positive length is a \emph{cycle} if
	$t(c_d)=s(c_1)$. A \emph{relation set} of $Q$ is a finite set of
	$\mK$-linear combinations of paths in $Q$. If $R$ is a relation set whose
	elements are paths of length two, then a cycle $c=c_1\cdots c_d$ is called a
	\emph{relation cycle} (with respect to $R$) if
	$c_ic_{i+1}\in R$ for $1\leq i<d$ and $c_dc_1\in R$. A cycle is
	\emph{primitive} if it is not a proper power of a shorter cycle.
	
	The path algebra of $Q$ over $\mK$ is denoted by $\mK Q$. If $R$ is a set of
	relations, then $\langle R\rangle$ denotes the two-sided ideal generated by
	$R$. For an arrow $\alpha$ and a path $c$, we write $\num{\alpha}{c}$ for
	the number of occurrences of $\alpha$ in $c$. Whenever a displayed sum
	involves products of paths, only composable products are included.
	
	\begin{definition}\label{def:gentle}
		A pair $(Q,I)$ consisting of a quiver $Q$ and a relation set $I$ is called a
		\emph{gentle pair} if the following conditions hold:
		\begin{enumerate}[label=\textup{(G\arabic*)}]
			\item Every element of $I$ is a path of length two.
			\item At each vertex, at most two arrows start and at most two arrows end.
			\item For every arrow $\alpha\in Q_1$, there is at most one arrow $\beta$
			with $\alpha\beta\in I$, and at most one arrow $\gamma$ with
			$\gamma\alpha\in I$.
			\item For every arrow $\alpha\in Q_1$, there is at most one arrow $\beta$
			with $\alpha\beta\notin I$, and at most one arrow $\gamma$ with
			$\gamma\alpha\notin I$.
		\end{enumerate}
		A finite-dimensional $\mK$-algebra is called \emph{gentle} if it is Morita
		equivalent to $\mK Q/\langle I\rangle$ for a gentle pair $(Q,I)$.
	\end{definition}
	
	\begin{definition}\label{def:finite-skew-gentle} 
		Let $Q$ be a quiver, let $I$ be a relation set of $Q$, and let $S\subseteq Q_0$.
		For each $i\in S$, adjoin a loop $\ep_i:i\to i$, called a
		\emph{special loop}, and denote the resulting quiver by $Q^S$.
		Set
		\[
		\Sp=\{\ep_i \mid i\in S\},
		\qquad
		I^S=I\cup\{\ep_i^2 \mid i\in S\}.
		\]
		The triple $(Q,I,S)$ is called a \emph{skew-gentle triple} if
		$(Q^S,I^S)$ is a gentle pair. A finite-dimensional $\mK$-algebra is called \emph{skew-gentle} if it is Morita equivalent to
		\[
		\mK Q^S/
		\left\langle
		I\cup\{\ep_i^2-\ep_i \mid i\in S\}
		\right\rangle.
		\]
		When $S=\varnothing$, this is precisely the gentle case. 
	\end{definition}
	
	A vertex $i\in S$ is called a \emph{special vertex}. 
	For $n\geq0$, denote by $Q_n^S$ the set of paths of length $n$ in $Q^S$.
	Now fix a skew-gentle triple $(Q,I,S)$ and consider the skew-gentle algebra 
	\[
	A=\mK Q^S/
	\left\langle
	I\cup\{\ep_i^2-\ep_i \mid i\in S\}
	\right\rangle.
	\]
	Let $\pi:\mK Q^S\longrightarrow A$ be the canonical projection. Following
	\cite[Section~2.1]{BSSWW26}, let $\mB$ be the set of paths in $Q^S$ that
	contain no element of $I^S$ as a subpath. The images under $\pi$ of the
	paths in $\mB$ form a $\mK$-basis of $A$. Whenever no confusion can arise,
	a path in $\mB$ is identified with its image in $A$.
	%===========================
	\subsection{The normalized relative Hochschild complex}
	\label{subsec:normalized-relative-Hochschild}
	
	Let
	$E=\mK Q_0=\bigoplus_{i\in Q_0}\mK e_i\subseteq A$.
	Then $E$ is a semisimple $\mK$-algebra. We write
	$E^e=E\otimes_{\mK}E^{\mathrm{op}}$ and $A^e=A\otimes_{\mK}A^{\mathrm{op}}$.
	Let $\mB^+=\mB\setminus Q_0$ and set $\bar A=\operatorname{span}_{\mK}\mB^+$.
	Then $A=E\oplus_{\mK}\bar A$ as $E$-bimodules, and the quotient map induces	an $E$-bimodule isomorphism $\bar A\cong A/E$. We denote by $\varpi:A\longrightarrow\bar A$	the projection associated with this decomposition. In particular,
	$\varpi(e_i)=0$ and $\varpi(\ep_i)=\ep_i$. For a path $p$ in $Q^S$, its
	value in $A$ is always understood through the canonical projection
	$\pi:\mK Q^S\to A$ fixed above. Thus, for example,
	$\pi(\ep_i^r)=\ep_i$ for every $r\geq1$.
	
	Ordinary Hochschild cohomology can be computed from the normalized
	$E$-relative bar resolution; see \cite{GS86,Wit19}.
	Using the $E^e$-module identification $A/E\cong\bar A$, the normalized
	$E$-relative bar resolution of $A$ has, for $n\geq0$, the terms
	\[
	\overline C_n(A)
	=
	A\otimes_E\bar A^{\otimes_E n}\otimes_E A,
	\]
	where $\bar A^{\otimes_E0}=E$. Together with the multiplication augmentation
	$d_0:\overline C_0(A)=A\otimes_E A\to A$, these terms form the normalized
	$E$-relative projective $A^e$-resolution of $A$. For $n\geq1$, its
	differential is induced by the ordinary bar differential and is given, for
	$r_i\in\bar A$, by
	\[
	\begin{aligned}
		&d_n(
		a_0\otimes_E r_1\otimes_E\cdots\otimes_E r_n\otimes_E a_{n+1}
		)
		\\
		&\quad={}
		a_0r_1\otimes_E r_2\otimes_E\cdots\otimes_E r_n\otimes_E a_{n+1}
		\\
		&\qquad+
		\sum_{i=1}^{n-1}(-1)^i
		a_0\otimes_E r_1\otimes_E\cdots\otimes_E
		\varpi(r_ir_{i+1})\otimes_E\cdots\otimes_E r_n\otimes_E a_{n+1}
		\\
		&\qquad+(-1)^n
		a_0\otimes_E r_1\otimes_E\cdots\otimes_E r_{n-1}\otimes_E r_na_{n+1}.
	\end{aligned}
	\]
	In an internal term, $\varpi(r_ir_{i+1})$ is the projection of
	$r_ir_{i+1}$ onto $\bar A$.
	
	Applying $\Hom_{A^e}(-,A)$ gives, for $n\geq0$, the normalized
	$E$-relative Hochschild cochain spaces
	\[
	\Hom_{A^e}(\overline C_n(A),A)
	\cong
	\Hom_{E^e}(\bar A^{\otimes_E n},A):=C^n(A).
	\]
	For $\varphi\in C^n(A)$ and $r_i\in\bar A$, the cochain differential is
	given by  
	\begin{equation}\label{eq:signed-Hochschild-differential}
		\begin{aligned}
			&(d_C\varphi)(r_1\otimes_E\cdots\otimes_E r_{n+1})
			\\
			&\quad={}
			(-1)^n\Bigg[
			r_1\varphi(r_2\otimes_E\cdots\otimes_E r_{n+1})
			\\
			&\qquad\qquad+
			\sum_{i=1}^{n}(-1)^i
			\varphi(
			r_1\otimes_E\cdots\otimes_E\varpi(r_ir_{i+1})
			\otimes_E\cdots\otimes_E r_{n+1}
			)
			\\
			&\qquad\qquad+
			(-1)^{n+1}
			\varphi(r_1\otimes_E\cdots\otimes_E r_n)r_{n+1}
			\Bigg].
		\end{aligned}
	\end{equation}	 
	The factor $(-1)^n$ does not change the Hochschild cohomology groups,
	since it leaves the kernels and images of the differentials unchanged; it
	is included only to match the dg Lie sign convention on $C^*(A)[1]$ used
	below. 
	
	Let $\varphi\in C^p(A)$ and $\psi\in C^q(A)$ be homogeneous. For $0\leq i\leq p-1$, define the normalized relative insertion by
	\[
	\varphi\circ_i\psi
	=
	\varphi(
	\id^{\otimes_E i}\otimes_E(\varpi\circ\psi)
	\otimes_E\id^{\otimes_E(p-i-1)}
	),
	\]
	and set
	\[
	\varphi\circ\psi
	=
	\sum_{i=0}^{p-1}(-1)^{i(q-1)}\varphi\circ_i\psi.
	\]
	The sum is zero when $p=0$. The Gerstenhaber bracket is
	\[
	[\varphi,\psi]
	=
	\varphi\circ\psi-(-1)^{(p-1)(q-1)}\psi\circ\varphi.
	\]
	These are the standard Gerstenhaber operations on the normalized relative
	Hochschild complex over $E$; equivalently, they are the relative bar formulas
	of \cite[Section~2.2]{BSSWW26}, written in the present grading convention.
	The Gerstenhaber bracket has cohomological degree $-1$ on $C^*(A)$. Under
	the shift convention $C^*(A)[1]^m=C^{m+1}(A)$,
	a $p$-cochain has degree $p-1$, and the Gerstenhaber bracket has degree
	zero. Together with the signed differential
	\eqref{eq:signed-Hochschild-differential}, it makes $(C^*(A)[1],d_C,[-,-])$
	a dg Lie algebra.
	
	%===========================
	\subsection{\texorpdfstring{$L_\infty$-algebras}{L-infinity algebras}}
	\label{subsec:Linfty-preliminaries}
	
	We follow the unsuspended cohomological convention and notation of
	\cite[Section~2]{RR22}. Let $V$ be a
	graded $\mK$-vector space and let $v_1,\dots,v_n\in V$ be homogeneous. We
	write $\bigwedge V$ for the free graded-commutative algebra on $V$, thus
	$\wedge$ denotes its product. Let $\mS_n$ denote the symmetric
	group on $n$ letters. For $\sigma\in\mS_n$, the Koszul sign
	$\epsilon(\sigma;v_1,\dots,v_n)$ is determined by
	\[
	v_1\wedge\cdots\wedge v_n
	=
	\epsilon(\sigma;v_1,\dots,v_n)\,
	v_{\sigma(1)}\wedge\cdots\wedge v_{\sigma(n)}.
	\]
	Let $\sgn(\sigma)\in\{\pm1\}$ denote the signature of $\sigma$, and set
	\[
	\chi(\sigma;v_1,\dots,v_n)
	=
	\sgn(\sigma)\epsilon(\sigma;v_1,\dots,v_n).
	\]
	When the homogeneous inputs are clear from the context, we write simply
	$\epsilon(\sigma)$ and $\chi(\sigma)$.
	
	For integers $r,s\geq0$ with $r+s=n$, an $(r,s)$-unshuffle is a permutation
	$\sigma\in\mS_n$ satisfying $\sigma(1)<\cdots<\sigma(r)$ and
	$\sigma(r+1)<\cdots<\sigma(n)$; the set of all such unshuffles is denoted by
	$\mS_{r,s}$. For $1\leq t<n$, put
	\[
	\mS_{t,n-t}^{-}=
	\{\sigma\in\mS_{t,n-t} \mid \sigma(1)<\sigma(t+1)\}.
	\]
	For homogeneous inputs $x_1,\dots,x_n$, define
	\[
	\kappa(\sigma)_t
	=(-1)^{(t-1)+(n-t-1)(\sum_{p=1}^{t}|x_{\sigma(p)}|)}.
	\]
	For $\sigma\in\mS_n$, let $\widehat\sigma$ denote the \emph{unsigned}
	permutation operator
	\[
	\widehat\sigma
	(v_1\otimes_{\mK}\cdots\otimes_{\mK}v_n)
	=
	v_{\sigma(1)}\otimes_{\mK}\cdots\otimes_{\mK}v_{\sigma(n)}.
	\]
	Thus the Koszul signs in the formulas below are carried by $\chi(\sigma)$,
	not by $\widehat\sigma$.
	
	\begin{definition}\label{def:Linfty}
		An $L_\infty$-algebra is a graded $\mK$-vector space $L$ together with maps
		$l_n:L^{\otimes_{\mK}n}\longrightarrow L$, $n\geq1$, of degree $2-n$, such
		that, for every $\sigma\in\mS_n$,
		\[
		l_n\widehat\sigma=\chi(\sigma)l_n,
		\]
		and, for every $N\geq1$,
		\[
		\sum_{\substack{i+j=N+1\\ i,j\geq1}}
		\ \sum_{\sigma\in\mS_{i,N-i}}
		(-1)^{i(j-1)}\chi(\sigma)
		l_j(l_i\otimes_{\mK}\id^{\otimes_{\mK}(N-i)})\widehat\sigma
		=0
		\]
		on homogeneous tensors in $L^{\otimes_{\mK}N}$.
	\end{definition}
	
	The first three identities recover the familiar dg Lie relations up to higher
	homotopy. For $N=1$, one has $l_1^2=0$; for $N=2$, $l_1$ is a derivation of
	$l_2$; and for $N=3$, the graded Jacobi identity for $l_2$ holds up to the
	homotopy encoded by $l_3$. In particular, a dg Lie algebra $(L,d,[-,-])$ is
	regarded as an $L_\infty$-algebra by taking $l_1=d$, $l_2=[-,-]$, and
	$l_n=0$ for $n\geq3$.
	
	Following the terminology and conventions of \cite[Definition~2.4]{RR22}, a weak $L_\infty$-morphism
	$\phi:L\to M$ is a collection of graded skew-symmetric maps
	$\phi_n:L^{\otimes_{\mK}n}\to M$, $n\geq1$, of degree $1-n$ satisfying the
	$L_\infty$-morphism identities. It is \emph{strict} if $\phi_n=0$ for all
	$n\geq2$, and it is an $L_\infty$ \emph{quasi-isomorphism} if
	$\phi_1:(L,l_1)\to(M,l_1)$ is a quasi-isomorphism of cochain complexes.
	
	The recursive homotopy transfer construction of
	\cite[Section~2]{RR22} is used in the following form. Let $(\mathfrak g,d,[-,-])$ be a dg Lie algebra and
	let $(V,\delta)$ be a cochain complex. Suppose that there are cochain maps
	fitting into the diagram
	\[\begin{tikzcd}
		(V,\delta) && (\mathfrak g,d)
		\arrow["\iota", shift left, from=1-1, to=1-3]
		\arrow["p", shift left, from=1-3, to=1-1]
	\end{tikzcd}\]
	and a degree $-1$ map $h:\mathfrak g\to\mathfrak g$ such that
	\[
	p\iota=\id_V,
	\qquad
	\iota p-\id_{\mathfrak g}=hd+dh,
	\qquad
	h\iota=0,
	\qquad
	ph=0,
	\qquad
	h^2=0.
	\]
	We refer to such data $(p,\iota,h)$ as contraction data.
	If $\phi_t:V^{\otimes_{\mK}t}\to\mathfrak g$ and
	$\phi_{n-t}:V^{\otimes_{\mK}(n-t)}\to\mathfrak g$, then
	$[\phi_t,\phi_{n-t}]$ denotes the $n$-linear map obtained by applying the dg
	Lie bracket of $\mathfrak g$ to the two output values:
	\[
	\begin{aligned}
		&[\phi_t,\phi_{n-t}]
		(x_1\otimes_{\mK}\cdots\otimes_{\mK}x_n)
		\\
		&\qquad=
		\bigl[
		\phi_t(x_1\otimes_{\mK}\cdots\otimes_{\mK}x_t),
		\phi_{n-t}(x_{t+1}\otimes_{\mK}\cdots\otimes_{\mK}x_n)
		\bigr].
	\end{aligned}
	\]
	Set $l_1=\delta$, $\phi_1=\iota$, and $u_1=v_1=0$. For $n>1$, define
	recursively
	\begin{align}
		v_n
		&=
		\sum_{t=1}^{n-1}
		\sum_{\tau\in\mS_{t,n-t}^{-}}
		\chi(\tau)\kappa(\tau)_t
		[\phi_t,\phi_{n-t}]\widehat\tau,
		\label{eq:vn}\\
		l_n&=p\,v_n,
		\label{eq:ln}\\
		u_n
		&=
		\sum_{k=2}^{n}
		\sum_{\tau\in\mS_{k,n-k}}
		(-1)^{k(n-k)+1}\chi(\tau)
		\phi_{n-k+1}(l_k\otimes_{\mK}\id^{\otimes_{\mK}(n-k)})\widehat\tau,
		\label{eq:un}\\
		\phi_n&=h(u_n+v_n).
		\label{eq:phin}
	\end{align}
	
	\begin{theorem}[{\cite[Theorem~2.7]{RR22}}]\label{thm:transfer-theorem}
		With the preceding notation, the maps $l_n$ define an $L_\infty$-structure
		on $V$, and the maps $\phi_n$ define a weak $L_\infty$ quasi-isomorphism
		$V\to\mathfrak g$.
	\end{theorem}
	
	When $V=H^*(\mathfrak g,d)$ is equipped with the zero differential, we write
	the transferred structure as
	\[
	\bigl(
	H^*(\mathfrak g,d),
	0,
	l_2^{\min},
	l_3^{\min},
	l_4^{\min},
	\ldots
	\bigr)
	\]
	and call it the \emph{transferred minimal model} of $\mathfrak g$ associated
	with the chosen contraction. Thus $l_1^{\min}=0$, while
	$l_2^{\min}$ is the bracket induced by the dg Lie bracket on cohomology.
	More generally, we use \emph{minimal model} for an $L_\infty$-structure on
	$H^*(L,l_1)$ with vanishing unary operation that is $L_\infty$
	quasi-isomorphic to $L$.
	
	Formality and homotopy abelianity are understood in the sense of
	\cite[Definitions~6.2.1--6.2.2 and Section~12.7]{Man22}.
	Write $l_2^H$ for the binary operation induced by $l_2$ on $H^*(L,l_1)$.
	
	\begin{definition}\label{def:formality}
		An $L_\infty$-algebra $L$ is \emph{formal} if it is $L_\infty$
		quasi-isomorphic to
		\[
		\bigl(H^*(L,l_1),0,l_2^H,0,0,\ldots\bigr).
		\]
		It is \emph{homotopy abelian} if it is $L_\infty$ quasi-isomorphic to
		\[
		\bigl(H^*(L,l_1),0,0,0,\ldots\bigr).
		\]
		A finite-dimensional algebra $A$ is \emph{Hochschild $L_\infty$-formal} if
		its Hochschild dg Lie algebra $C^*(A)[1]$ is formal.
	\end{definition}
	
	Equivalently, in terms of minimal models,
	\[
	\begin{aligned}
		L\text{ is formal}
		&\quad\Longleftrightarrow\quad
		\text{a minimal model of }L\text{ can be chosen with }
		l_r^{\min}=0\text{ for all }r\geq3,
		\\
		L\text{ is homotopy abelian}
		&\quad\Longleftrightarrow\quad
		\text{a minimal model of }L\text{ can be chosen with }
		l_r^{\min}=0\text{ for all }r\geq2.
	\end{aligned}
	\]
	The individual higher brackets in a transferred minimal model depend on the
	chosen contraction and may change under an $L_\infty$-isomorphism of minimal
	models. Consequently, the nonvanishing of some $l_r^{\min}$ in one
	transferred minimal model does not by itself imply non-formality.

	%===========================
	\section{The transferred \texorpdfstring{$L_\infty$}{L-infinity}-structure}\label{sec:main-formulas}
	
	With the Hochschild dg Lie algebra and the homotopy transfer theorem
	fixed in the previous section, the transfer to the CS cochain complex
	requires the CS-projective resolution, comparison morphisms with the
	normalized relative bar resolution, and a compatible homotopy.
	The resolution and comparison morphisms are recalled below, and the
	required homotopy is constructed. The resulting contraction is used
	to transfer the dg Lie structure from $C^*(A)[1]$ to
	$B_{\mathrm{CS}}^*(A)[1]$.
	%===========================
	\subsection{The CS-projective resolution}
	
 Let
	$\Gamma=\cup_{n\geq0}\Gamma_n$,
	where $\Gamma_0=Q_0$, $\Gamma_1=Q_1^S$, and, for $n\geq2$,
	\[
	\Gamma_n=
	\{\gamma_1\cdots\gamma_n\in Q_n^S \mid 
	\gamma_i\gamma_{i+1}\in I^S\text{ for }1\leq i<n\}.
	\]
	Elements of $\Gamma_n$ are called relation concatenations. We identify a
	vertex $i\in Q_0$ with the corresponding trivial path $e_i$. For $n\geq1$,
	set $\Sp^n=\{\ep_i^n \mid i\in S\}\subseteq\Gamma_n$.
	
	For $n\geq0$, set
	$B_n(A)=A\otimes_E\mK\Gamma_n\otimes_E A$. The following is the
	ungraded, left-to-right form of the CS-projective resolution in
	\cite[Proposition~2.8]{BSSWW26}. When $S=\varnothing$, it reduces to
	Bardzell's resolution for quadratic monomial algebras \cite{Bar97}.
	
	\begin{lemma}[{\cite[Proposition~2.8]{BSSWW26}}]\label{thm:CS-resolution}
		Let $A$ be a finite-dimensional skew-gentle algebra.  There is a projective
		$A^e$-resolution
		\[
		\cdots\longrightarrow
		B_n(A)
		\xrightarrow{\delta_n}
		B_{n-1}(A)
		\longrightarrow\cdots\longrightarrow
		B_0(A)
		\xrightarrow{\delta_0}A.
		\]
		The augmentation $\delta_0$ is multiplication.  For
		$\gamma=\gamma_1\cdots\gamma_n\in\Gamma_n\setminus\Sp^n$, $n\geq1$, one has
		\[
		\delta_n(1\otimes_E\gamma_1\cdots\gamma_n\otimes_E1)
		=
		\gamma_1\otimes_E\gamma_2\cdots\gamma_n\otimes_E1
		+(-1)^n1\otimes_E\gamma_1\cdots\gamma_{n-1}\otimes_E\gamma_n.
		\]
		Here an empty middle concatenation is interpreted as the corresponding trivial
		path.  For a special-loop concatenation $\ep_i^n\in\Sp^n$,
		\[
		\delta_n(1\otimes_E\ep_i^n\otimes_E1)=
	\begin{cases}
		\ep_i\otimes_E\ep_i^{n-1}\otimes_E1
		-1\otimes_E\ep_i^{n-1}\otimes_E\ep_i,
		& \text{if } n\text{ is odd},\\
		\ep_i\otimes_E\ep_i^{n-1}\otimes_E1
		+1\otimes_E\ep_i^{n-1}\otimes_E\ep_i
		-1\otimes_E\ep_i^{n-1}\otimes_E1,
		& \text{if } n\text{ is even}.
	\end{cases}
		\]
	\end{lemma}

	We reserve $B_n(A)$ for the $n$th term of the CS-projective resolution and
	write $B_{\mathrm{CS}}^n(A)$ for the corresponding dual cochain space
	$\Hom_{A^e}(B_n(A),A)$.  Thus applying
	$\Hom_{A^e}(-,A)$ to $B_\bullet(A)$ gives the CS cochain complex
	$B_{\mathrm{CS}}^\bullet(A)$.
	
	For finite sets of paths $X$ and $Y$, define
	\[
	\Para{X}{Y}
	=
	\{\para{x}{y}\mid x\in X,\ y\in Y,\ x\parallel y\},
	\]
	and let $\mK\Para{X}{Y}$ be the $\mK$-vector space with basis
	$\Para{X}{Y}$.
	
	\begin{lemma}[{\cite{Cib90}}]
		\label{lem:CS-cochain-identification}
		For each $n\geq0$, there is an isomorphism of $\mK$-vector spaces
		\[
		\Hom_{A^e}(A\otimes_E\mK\Gamma_n\otimes_E A,A)
		\cong
		\Hom_{E^e}(\mK\Gamma_n,A)
		\cong
		\mK\Para{\Gamma_n}{\mB}.
		\]
	\end{lemma}
	
	Under this identification, 
	$\para{\gamma}{a}\in\Para{\Gamma_n}{\mB}$ denotes the elementary
	$E^e$-linear map $\mK\Gamma_n\to A$ sending $\gamma$ to $a$ and every other element of
	$\Gamma_n$ to zero. The notation is extended linearly in the
	second entry. If $u=\sum_j c_j a_j$ where $a_j\in\mB$ and $a_j\parallel\gamma$, then 
	we set
	$\para{\gamma}{u}
	=
	\sum_j c_j\para{\gamma}{a_j}$.
	
	The CS cochain spaces are equipped with the signed differential
	\begin{equation}\label{eq:signed-CS-differential}
		\delta^m:B_{\mathrm{CS}}^m(A)\longrightarrow B_{\mathrm{CS}}^{m+1}(A),
		\qquad
		\delta^m(f)=(-1)^m f\circ\delta_{m+1}.
	\end{equation}
	
	\begin{lemma}[{\cite[Proposition~2.10]{BSSWW26}}]
		\label{lem:dual-differential}
		Under the identification of
		\Cref{lem:CS-cochain-identification}, the signed CS differential is
		given as follows. If $m=0$, or if $m\geq1$ and
		$\gamma\in\Gamma_m\setminus\Sp^m$, then
		\[\delta^m\para{\gamma}{a}
			=
			(-1)^m
			\sum_{\substack{\beta\in Q_1^S\\
					\para{\beta\gamma}{\pi(\beta a)}
					\in\Para{\Gamma_{m+1}}{\mB}}}
			\para{\beta\gamma}{\pi(\beta a)}
			-
			\sum_{\substack{\alpha\in Q_1^S\\
					\para{\gamma\alpha}{\pi(a\alpha)}
					\in\Para{\Gamma_{m+1}}{\mB}}}
			\para{\gamma\alpha}{\pi(a\alpha)}.\]
		If $m\geq1$ and $\gamma=\ep_i^m$, then
		\[
		\delta^m\para{\ep_i^m}{a}
		=
		\begin{cases}
			\para{\ep_i^{m+1}}{\pi(\ep_i a-a\ep_i)},
			& \text{if } m\text{ is even},\\
			\para{\ep_i^{m+1}}{\pi(a-\ep_i a-a\ep_i)},
			& \text{if } m\text{ is odd}.
		\end{cases}
		\]
	\end{lemma}
	%===========================
	\subsection{Comparison morphisms and contraction}
	
	Write
	\[
	F_n:B_n(A)\longrightarrow\overline C_n(A),\qquad
	G_n:\overline C_n(A)\longrightarrow B_n(A).
	\]
	These are the comparison morphisms of
	\cite[Propositions~2.12 and~2.13]{BSSWW26} in the ungraded left-to-right
	convention. On standard $A$-bimodule generators, they are given as follows
	and extended $A$-bilinearly.
	For $n\geq1$,
	\[
	F_n(1\otimes_E\gamma_1\cdots\gamma_n\otimes_E1)
	=1\otimes_E\gamma_1\otimes_E\cdots\otimes_E\gamma_n\otimes_E1.
	\]
	For a path $x\in\mB$, written $x=x_1\cdots x_{\ell(x)}$,
	\[
	G_1(1\otimes_E x\otimes_E1)=
	\sum_{i=1}^{\ell(x)}x_1\cdots x_{i-1}\otimes_E x_i\otimes_E x_{i+1}\cdots x_{\ell(x)}.
	\]
	For $n\geq2$ and composable paths
	$x_1,\dots,x_n\in\mB^+$,
	\[
	G_n(1\otimes_E x_1\otimes_E\cdots\otimes_E x_n\otimes_E1)
	=
	\begin{cases}
		L\otimes_E\alpha x_2\cdots x_{n-1}\beta\otimes_E R,
		& \begin{array}{l}
			\text{if } x_1=L\alpha,\ x_n=\beta R\\
			\text{and }\alpha x_2\cdots x_{n-1}\beta\in\Gamma_n,
		\end{array}\\[5mm]
		0, & \text{otherwise.}
	\end{cases}
	\]
	For $n=0$, under the canonical identification
	$B_0(A) \cong \overline C_0(A) \cong A\otimes_E A$,
	we take $F_0=G_0=\id$.
	%================================
	\begin{lemma}[{\cite[Propositions~2.12 and~2.13]{BSSWW26}}]\label{lem:comparison-maps}
		The maps $F_\bullet$ and $G_\bullet$ are chain maps and satisfy $G_nF_n=\id_{B_n(A)}$ for every $n\geq0$.
	\end{lemma}

	The basis decomposition
	$A=E\oplus_{\mK}\bar A$ gives the standard contracting homotopy of the reduced
	$E$-relative bar resolution,
	\[
	\begin{aligned}
		s_n:\overline C_n(A)&\longrightarrow\overline C_{n+1}(A),\\
		s_n(a\otimes_E r_1\otimes_E\cdots\otimes_E r_n\otimes_E b)
		&=
		1\otimes_E\varpi(a)\otimes_E r_1\otimes_E\cdots\otimes_E r_n\otimes_E b.
	\end{aligned}
	\]
	The map $s_n$ is left $E$-linear and right $A$-linear, but in general it is
	not left $A$-linear. At the augmentation level, we use the standard convention
	$s_{-1}(a)=1\otimes_E a$ for $a\in A$. On the augmented reduced relative bar resolution these
	maps satisfy the usual contracting-homotopy identity
	\[
	d_{n+1}s_n+s_{n-1}d_n=\id.
	\]
	
	The following adapts \cite[Lemma~3.1]{RR22} to the $E$-relative setting.
	
	\begin{lemma}
		\label{lem:relative-recursive-homotopy}
		Set $H_{-1}=H_0=0$.  For $n\geq1$, suppose that $H_{n-1}$ has already
		been defined as an $A^e$-linear map.  For
		\[
		\begin{aligned}
			\mathbf r&=r_1\otimes_E\cdots\otimes_E r_n
			\in\bar A^{\otimes_E n},\\
			\widetilde{\mathbf r}
			&=1\otimes_Er_1\otimes_E\cdots\otimes_Er_n\otimes_E1,
		\end{aligned}
		\]
		define $H_n$ on this standard generator by
		\begin{equation}\label{eq:recursive-H}
			H_n(\widetilde{\mathbf r})
			=
			(s_nF_nG_n-s_nH_{n-1}d_n)
			(\widetilde{\mathbf r}),
		\end{equation}
		and extend by
		\begin{equation}\label{eq:H-Ae-extension}
			H_n(a\otimes_E\mathbf r\otimes_Eb)
			=
			a H_n(\widetilde{\mathbf r}) b.
		\end{equation}
		Then \eqref{eq:H-Ae-extension} is well defined and gives an $A^e$-linear
		map
		$H_n:\overline C_n(A)\longrightarrow\overline C_{n+1}(A)$.
		Moreover,
		\begin{equation}\label{eq:FG-homotopy}
			F_nG_n-\id_{\overline C_n(A)}
			=
			d_{n+1}H_n+H_{n-1}d_n
		\end{equation}
		for every $n\geq0$.
	\end{lemma}
	
	\begin{proof}
		Because $\varpi:A\to\bar A$ is an $E^e$-linear projection, $s_n$ is left
		$E$-linear and right $A$-linear.  By induction, the maps
		$F_n,G_n,d_n$, and $H_{n-1}$ occurring before $s_n$ are $A^e$-linear.
		Consequently,
		\[T_n:=s_nF_nG_n-s_nH_{n-1}d_n\]
		is left $E$-linear and right $A$-linear.  If $e,f\in E$, then the balancing
		relations in the relative tensor product give
		$\widetilde{e\mathbf r f} = e \widetilde{\mathbf r} f$,
		and hence
		$T_n(\widetilde{e\mathbf r f}) = e T_n(\widetilde{\mathbf r}) f$.
		Thus the assignment in \eqref{eq:recursive-H} is $E^e$-linear in
		$\mathbf r$, and therefore it has the unique $A^e$-linear extension
		\eqref{eq:H-Ae-extension}.
		
		It remains to prove the homotopy identity.  The case $n=0$ follows from
		$F_0G_0=\id$ and $H_{-1}=H_0=0$.  Assume inductively that
		\eqref{eq:FG-homotopy} holds in degree $n-1$.  Composing that identity on
		the right with $d_n$, and using that $F_\bullet$ and $G_\bullet$ are chain
		maps, gives
		$d_nH_{n-1}d_n = d_nF_nG_n-d_n$.
		Let $z=\widetilde{\mathbf r}$ be a standard generator.  Then
		$d_n(F_nG_n-H_{n-1}d_n)(z)=d_n(z)$.
		Moreover, $s_n(z)=0$ because $\varpi(1)=0$, and the contracting-homotopy
		identity therefore gives
		$s_{n-1}d_n(z)=z$.
		Using \eqref{eq:recursive-H}, we obtain
		\begin{align*}
			d_{n+1}H_n(z)
			&=
			d_{n+1}s_n(F_nG_n-H_{n-1}d_n)(z)\\
			&=
			(\id-s_{n-1}d_n)
			(F_nG_n-H_{n-1}d_n)(z)\\
			&=
			F_nG_n(z)-H_{n-1}d_n(z)-s_{n-1}d_n(z)\\
			&=
			F_nG_n(z)-z-H_{n-1}d_n(z).
		\end{align*}
		This proves \eqref{eq:FG-homotopy} on the standard generators.  Both sides
		are $A^e$-linear, so the identity holds on all of $\overline C_n(A)$.
	\end{proof}
	
	The side conditions required for the transfer are established below by
	adapting the monomial arguments of \cite[Lemmas~3.5 and~3.6]{RR22}
	to the skew-gentle setting. The required path combinatorics is governed
	by the quadratic monomial set $I^S$, and every path in $\mB$
	avoids the relations in $I^S$.
	
	\begin{lemma}\label{lem:GH-zero}
		For every $n\geq0$, $G_{n+1}H_n=0$.
	\end{lemma}
	
	\begin{proof}
		The case $n=0$ is immediate from $H_0=0$.  We prove the assertion for
		$n\geq1$ by induction.
		
		We first record a factorization of $G_{n+1}s_n$ through $G_n$.  For
		$\gamma=\gamma_1\cdots\gamma_n\in\Gamma_n$, define a $\mK$-linear map
		$K_n:B_n(A)\to B_{n+1}(A)$ on the standard basis determined by $\mB$ by
		\[
		K_n(c\otimes_E\gamma\otimes_E b)
		=
		\begin{cases}
			L\otimes_E\lambda\gamma\otimes_E b,
			& \text{if } c=L\lambda\text{ and }\lambda\gamma_1\in I^S,\\
			0, & \text{otherwise},
		\end{cases}
		\]
		where in the first case $c$ is a nontrivial path in $\mB$ and $\lambda$
		is its last arrow.  Since $\gamma\in\Gamma_n$, the condition
		$\lambda\gamma_1\in I^S$ implies $\lambda\gamma\in\Gamma_{n+1}$, so
		$K_n$ is well defined.
		
		We claim that $G_{n+1}s_n=K_nG_n$ on $\overline C_n(A)$.  By linearity,
		it is enough to consider a basis tensor
		$u=a\otimes_E x_1\otimes_E\cdots\otimes_E x_n\otimes_E b$, where
		$a$ and the $x_i$ lie in $\mB$.
		
		Assume first that $n\geq2$.  We first note that
		$G_{n+1}s_n(u)\neq0$ implies $G_n(u)\neq0$.  Indeed, in this case
		$\varpi(a)\neq0$, so we may write $a=L\lambda$, and writing
		$x_n=\beta R$, the defining formula for $G_{n+1}$ requires
		$\lambda x_1\cdots x_{n-1}\beta\in\Gamma_{n+1}$.
		Since this relation concatenation has length $n+1$, each of
		$x_1,\ldots,x_{n-1}$ is an arrow.  Deleting its first arrow therefore gives
		$x_1\cdots x_{n-1}\beta\in\Gamma_n$, which is precisely a nonzero term of
		$G_n(u)$.  Hence, if $G_n(u)=0$, both sides of the claimed factorization
		vanish.
		
		We may therefore assume that $G_n(u)\neq0$.  Write
		$x_1=P\alpha$, $x_n=\beta R$, with
		$\alpha x_2\cdots x_{n-1}\beta\in\Gamma_n$.  Then
		$G_n(u)=aP\otimes_E\alpha x_2\cdots x_{n-1}\beta\otimes_E Rb$.
		If $P$ is nontrivial and this term is nonzero after applying $\pi$, the last arrow of the left coefficient $aP$ is the last
		arrow $\mu$ of $P$: the only possible nonzero reduction at the joining
		boundary between $a$ and $P$ is a special-loop idempotent reduction, which
		does not change that last arrow. Since $P\alpha=x_1$ lies in $\mB$,
		$\mu\alpha\notin I^S$, and hence $K_nG_n(u)=0$. The preceding observation
		also gives $G_{n+1}s_n(u)=0$ in this case, since a nonzero value would force
		$x_1$ to be an arrow, whereas $P$ is nontrivial.
		
		Thus either side can be nonzero only when $P$ is trivial.  In that
		case $x_1=\alpha$.  Writing $a=L\lambda$, the value $K_nG_n(u)$ is
		nonzero precisely when $\lambda\alpha\in I^S$, and then it equals
		$L\otimes_E\lambda\alpha x_2\cdots x_{n-1}\beta\otimes_E Rb$.
		This is exactly the value obtained by first applying $s_n$, which
		inserts $\varpi(a)$ as the first bar factor, and then applying
		$G_{n+1}$.  If $a$ is trivial, both sides vanish because $\varpi(a)=0$
		and $K_n$ vanishes on trivial left coefficients.  Hence
		$G_{n+1}s_n(u)=K_nG_n(u)$ for $n\geq2$.
		
		For $n=1$, write $u=a\otimes_E x\otimes_E b$. The formula for $G_1$
		gives a sum over decompositions $x=P\alpha Q$. Whenever $P$ is
		nontrivial, a nonzero value $\pi(aP)$ still ends
		in the last arrow $\mu$ of $P$, since the only possible nonzero reduction
		at the joining boundary is a special-loop idempotent reduction. Moreover,
		$\mu\alpha\notin I^S$, since $P\alpha$ is a subpath of the path in $\mB$
		$x$; hence the corresponding summand is killed by $K_1$.  Only the
		decomposition with $P$ trivial can contribute, and on that summand
		the condition imposed by $K_1$ is precisely the condition imposed by
		$G_2$ after $s_1$ has inserted $\varpi(a)$.  Thus
		$G_2s_1=K_1G_1$ as well.  This proves the claimed factorization for
		every $n\geq1$.
		
		Now let
		$z=1\otimes_E x_1\otimes_E\cdots\otimes_E x_n\otimes_E1$
		be a standard normalized basis generator.  By
		\eqref{eq:recursive-H} and the factorization above,
		\[
		\begin{aligned}
			G_{n+1}H_n(z)
			&=
			K_nG_n(F_nG_n-H_{n-1}d_n)(z)\\
			&=
			K_n((G_nF_n)G_n-(G_nH_{n-1})d_n)(z).
		\end{aligned}
		\]
		By \Cref{lem:comparison-maps}, $G_nF_n=\id$.  For $n=1$ we have
		$G_1H_0=0$, while for $n>1$ the induction hypothesis gives
		$G_nH_{n-1}=0$.  Consequently,
		$G_{n+1}H_n(z)=K_nG_n(z)$.
		
		It remains to show that $K_nG_n(z)=0$.  Suppose first that $n\geq2$
		and $G_n(z)\neq0$.  Then, by the explicit formula for $G_n$, there
		are paths $L,R\in\mB$ and arrows $\alpha,\beta$ such that
		$x_1=L\alpha$, $x_n=\beta R$, and
		$\alpha x_2\cdots x_{n-1}\beta\in\Gamma_n$.  If $L$ is trivial,
		$K_nG_n(z)=0$ by definition.  If $L=L'\lambda$ is nontrivial, then
		$\lambda\alpha$ is an adjacent pair inside the path in $\mB$
		$x_1=L\alpha$, and therefore $\lambda\alpha\notin I^S$.  Again
		$K_nG_n(z)=0$.
		
		For $n=1$, every summand of $G_1(z)$ has the form
		$L\otimes_E\alpha\otimes_E R$ with $x_1=L\alpha R$.  If $L$ is
		trivial, it is killed by $K_1$; if $L=L'\lambda$ is nontrivial,
		then $\lambda\alpha\notin I^S$ because $x_1\in\mB$, so it is
		again killed by $K_1$.  Hence $K_1G_1(z)=0$.
		
		Thus $G_{n+1}H_n(z)=0$ for every standard normalized generator $z$.
		Since $G_{n+1}H_n$ is $A^e$-linear, the identity holds on all of
		$\overline C_n(A)$.
	\end{proof}
	
	\begin{lemma}\label{lem:HF-zero}
		For every $n\geq0$, $H_nF_n=0$.
	\end{lemma}
	
	\begin{proof}
		We proceed by induction on $n$.  The case $n=0$ follows from $H_0=0$.
		Assume that
		$H_{n-1}F_{n-1}=0$.
		Since $H_nF_n$ is $A^e$-linear, it is enough to evaluate it on a standard
		generator
		$1\otimes_E\gamma\otimes_E1\in B_n(A)$, where $\gamma=\gamma_1\cdots\gamma_n\in\Gamma_n$.
		The image of this generator under $F_n$ is again a standard normalized
		bar generator.  Since $G_nF_n=\id$ by \Cref{lem:comparison-maps} and
		$F_\bullet$ is a chain map, the generator-level formula
		\eqref{eq:recursive-H} gives
		\begin{align*}
			&H_nF_n(1\otimes_E\gamma\otimes_E1)
			\\
			&\quad=
			s_nF_nG_nF_n(1\otimes_E\gamma\otimes_E1)
			-
			s_nH_{n-1}d_nF_n(1\otimes_E\gamma\otimes_E1)
			\\
			&\quad=
			s_nF_n(1\otimes_E\gamma\otimes_E1)
			-
			s_nH_{n-1}F_{n-1}
			\delta_n(1\otimes_E\gamma\otimes_E1)
			\\
			&\quad=
			s_nF_n(1\otimes_E\gamma\otimes_E1).
		\end{align*}
		By the definition of $F_n$,
		$F_n(1\otimes_E\gamma\otimes_E1) = 1\otimes_E\gamma_1\otimes_E\cdots \otimes_E\gamma_n\otimes_E1$.
		Therefore
		$s_nF_n(1\otimes_E\gamma\otimes_E1)=0$,
		because $\varpi(1)=0$.  Thus $H_nF_n$ vanishes on standard generators,
		and $A^e$-linearity gives $H_nF_n=0$ for every $n\geq0$.
	\end{proof}
	
	\begin{lemma}\label{lem:HH-zero}
		For every $n\geq0$, $H_{n+1}H_n=0$.
	\end{lemma}
	
	\begin{proof}
		We argue by induction on $n$.  The case $n=0$ is immediate from $H_0=0$.
		Assume that
		$H_nH_{n-1}=0$.
		Since $H_{n+1}H_n$ is $A^e$-linear, it is enough to evaluate it on a
		standard normalized generator
		$z= 1\otimes_E x_1\otimes_E\cdots\otimes_E x_n\otimes_E1$.
		By \eqref{eq:recursive-H}, every summand of $H_n(z)$ lies in
		$\operatorname{Im}s_n$ and therefore has trivial left outer coefficient.
		After expanding the middle factors in the basis $\mB$, each summand
		has the form $z'b$, where $z'$ is a standard normalized generator and
		$b\in A$. Since both $H_{n+1}$ and the generator-level recursive expression
		are right $A$-linear, the recursion may be applied termwise to $H_n(z)$.
		Using \Cref{lem:GH-zero}, we obtain
		\begin{align*}
			H_{n+1}H_n(z)
			&=
			s_{n+1}F_{n+1}G_{n+1}H_n(z)
			-
			s_{n+1}H_nd_{n+1}H_n(z)
			\\
			&=
			-s_{n+1}H_nd_{n+1}H_n(z).
		\end{align*}
		By \eqref{eq:FG-homotopy},
		$d_{n+1}H_n = F_nG_n-\id-H_{n-1}d_n$.
		Hence
		\begin{align*}
			H_{n+1}H_n(z)
			&=
			-s_{n+1}H_n
			(F_nG_n-\id-H_{n-1}d_n)(z)
			\\
			&=
			s_{n+1}H_n(z),
		\end{align*}
		where $H_nF_n=0$ by \Cref{lem:HF-zero} and
		$H_nH_{n-1}=0$ by the induction hypothesis.  Finally, the
		generator-level formula \eqref{eq:recursive-H} gives
		$H_n(z) = s_n(F_nG_n-H_{n-1}d_n)(z)$.
		Since $s_{n+1}s_n=0$ by $\varpi(1)=0$, it follows that
		$H_{n+1}H_n(z)=s_{n+1}H_n(z)=0$.
		The result on all of $\overline C_n(A)$ follows from $A^e$-linearity.
	\end{proof}
	
	Consequently, \eqref{eq:FG-homotopy} and
	\Cref{lem:comparison-maps,lem:GH-zero,lem:HF-zero,lem:HH-zero}
	give the contraction data $(F_\bullet,G_\bullet,H_\bullet)$.
	%====================
	Since $F_\bullet$, $G_\bullet$, and $H_\bullet$ are $A^e$-linear,
	precomposition gives maps on $A^e$-linear cochains. Dualizing by
	$\Hom_{A^e}(-,A)$ gives the cochain maps $F^*:C^*(A)\to B_{\mathrm{CS}}^*(A)$ and $G^*:B_{\mathrm{CS}}^*(A)\to C^*(A)$, where $F^*(\varphi)=\varphi F$ and $G^*(f)=fG$.
	In accordance with the signed differentials
	\eqref{eq:signed-Hochschild-differential} and
	\eqref{eq:signed-CS-differential}, define
	\begin{equation}\label{eq:signed-dual-homotopy}
		\begin{aligned}
			H^*:C^n(A)\longrightarrow C^{n-1}(A),\qquad
			H^*(\varphi)=
			\begin{cases}
				0, & \text{if } n=0,\\
				(-1)^{n-1}\varphi\circ H_{n-1}, & \text{if } n\geq1.
			\end{cases}
		\end{aligned}
	\end{equation}
	
	\begin{corollary}\label{cor:dual-contraction}
		The maps $F^*,G^*,H^*$ satisfy
		\[
		F^*G^*=\id,\qquad
		G^*F^*-\id=H^*d_C+d_CH^*,
		\]
		and
		\[
		H^*G^*=0,\qquad
		F^*H^*=0,\qquad
		H^*H^*=0.
		\]
	\end{corollary}
	
	\begin{proof}
		For $f\in B_{\mathrm{CS}}^n(A)$ and $\varphi\in C^n(A)$,
		$F^*G^*(f)=fG_nF_n=f$.
		The maps $F^*$ and $G^*$ are cochain maps because $F_\bullet$ and
		$G_\bullet$ are chain maps and both cochain complexes use the same factor
		$(-1)^n$ in degree $n$.
		
		For $n=0$, the homotopy identity is immediate from $F_0G_0=\id$ and
		$H_0=0$.  Let $n\geq1$ and take $\varphi\in C^n(A)$.  By
		\eqref{eq:signed-Hochschild-differential} and
		\eqref{eq:signed-dual-homotopy}, the two degree-dependent signs cancel in
		each composite:
		\begin{align*}
			H^*d_C(\varphi)
			&=
			(-1)^n(-1)^n\varphi d_{n+1}H_n
			=
			\varphi d_{n+1}H_n,
			\\
			d_CH^*(\varphi)
			&=
			(-1)^{n-1}(-1)^{n-1}\varphi H_{n-1}d_n
			=
			\varphi H_{n-1}d_n.
		\end{align*}
		Hence \eqref{eq:FG-homotopy} gives
		\[
		(H^*d_C+d_CH^*)(\varphi)
		=
		\varphi(F_nG_n-\id)
		=
		(G^*F^*-\id)(\varphi).
		\]
		The three side conditions follow from
		\Cref{lem:GH-zero,lem:HF-zero,lem:HH-zero}; the scalar signs in
		\eqref{eq:signed-dual-homotopy} do not affect their vanishing.
	\end{proof}
	
	%===========================
	\subsection{Transfer of the Hochschild dg Lie structure}
	\label{subsec:transfer-Hochschild}
	
	The preceding construction provides the contraction data required by
	\Cref{thm:transfer-theorem}.  We now apply that theorem to the shifted CS and
	Hochschild cochain complexes.  Set $B=B_{\mathrm{CS}}^*(A)[1]$ and
	$C=C^*(A)[1]$.
	A basis element of shifted degree $m$ in $B$ is represented by a cochain
	$\para{x}{a}\in\mK\Para{\Gamma_{m+1}}{\mB}$.
	Apply \Cref{thm:transfer-theorem} with
	\[
	V=B,\qquad \mathfrak g=C,\qquad \iota=G^*,\qquad p=F^*,\qquad h=H^*.
	\]
	
	\begin{theorem}\label{thm:transferred-Linfty}
		Let $A$ be a finite-dimensional skew-gentle algebra.  The shifted CS cochain
		complex $B=B_{\mathrm{CS}}^*(A)[1]$ carries an $L_\infty$-structure
		$\{l_n\}_{n\geq1}$ defined by
		\Cref{eq:vn,eq:ln,eq:un,eq:phin}.  Moreover, $G^*$
		extends to a weak $L_\infty$ quasi-isomorphism
		\[
		\phi:B\longrightarrow C^*(A)[1].
		\]
		In particular,
		\[
		l_1=\delta,\qquad
		l_2(f\otimes_{\mK}g)=F^*[G^*f,G^*g].
		\]
	\end{theorem}
	
	\begin{proof}
		By \Cref{cor:dual-contraction}, the complexes $B$ and $C$, together with
		$F^*,G^*,H^*$, satisfy the contraction identities and side conditions required
		in \Cref{thm:transfer-theorem}.  Since $C=C^*(A)[1]$ is a dg Lie algebra under
		the shifted Gerstenhaber bracket, the transfer theorem applies and gives an
		$L_\infty$-structure on $B$.  The formula for $l_1$ is the transferred unary
		operation, and the recursion for $n=2$ gives
		$l_2(f\otimes_{\mK}g)=F^*[G^*f,G^*g]$.  The same theorem also gives the weak
		$L_\infty$ morphism whose first component is $\phi_1=G^*$.
	\end{proof}

	After the degree shift and our path-order conventions, the operation $l_2$
	agrees with the Gerstenhaber bracket transported to the CS cochain complex in
	\cite[Section~3.6]{BSSWW26}; see also the insertion formulas in
	\cite[Proposition~3.25]{BSSWW26}.  These formulas are reorganized below into explicit basis-level formulas for
	the bracket, including cancellations between the two insertion directions
	and, in the periodic cases, occurrence multiplicities and their signs.
	%===========================
	
	Let $m,n\geq0$, let
	$f=\para{x}{a}\in B^m=\mK\Para{\Gamma_{m+1}}{\mB}$, and let
	$g=\para{y}{b}\in B^n=\mK\Para{\Gamma_{n+1}}{\mB}$.
	For
	$\gamma=\gamma_1\cdots\gamma_{m+n+1}\in\Gamma_{m+n+1}$,
	all products occurring as cochain values or outer coefficients are evaluated in $A$ via the canonical projection $\pi:\mK Q^S\to A$. Whenever such a value is nonzero, we identify it with its expansion in the basis $\mB$. Empty strings of tensor factors are omitted. Then
	\begin{align}\label{eq:general-l2}
		&l_2(f\otimes_{\mK} g)(\gamma)
		\notag\\
		={}&
		\sum_{i=0}^{m}(-1)^{i(n+2)}
		\para{x}{a}G_{m+1}
		(
		\gamma_1\otimes_E\cdots\otimes_E\gamma_i
		\otimes_E
		\varpi(\para{y}{b}(\gamma_{i+1}\cdots\gamma_{i+n+1}))
		\notag\\[-1mm]
		&\hspace{44mm}
		\otimes_E\gamma_{i+n+2}\otimes_E\cdots\otimes_E\gamma_{m+n+1}
		)
		\notag\\
		&-
		(-1)^{mn}
		\sum_{j=0}^{n}(-1)^{j(m+2)}
		\para{y}{b}G_{n+1}
		(
		\gamma_1\otimes_E\cdots\otimes_E\gamma_j
		\otimes_E
		\varpi(\para{x}{a}(\gamma_{j+1}\cdots\gamma_{j+m+1}))
		\notag\\[-1mm]
		&\hspace{44mm}
		\otimes_E\gamma_{j+m+2}\otimes_E\cdots\otimes_E\gamma_{m+n+1}
		).
	\end{align}
	Here an empty initial or terminal tensor string is simply omitted. The two
	sums correspond, respectively, to inserting $g$ into $f$ and inserting $f$
	into $g$.
	
	For arbitrary homogeneous $f,g\in B$, the transferred binary bracket is
	graded skew-symmetric:
	\begin{align}\label{sym}
		l_2(f\otimes_{\mK} g)
		=
		-(-1)^{|f||g|}l_2(g\otimes_{\mK} f),
	\end{align}
	where $|-|$ denotes the shifted degree on $B$.
	
	For homogeneous inputs $f\in B^m$ and $g\in B^n$, we call $(m,n)$ the
	\emph{degree pair} of $f\otimes_{\mK}g$, and we say that
	$l_2(f\otimes_{\mK}g)$ is a bracket of type $(m,n)$.  This terminology
	records only the two shifted input degrees and does not introduce an additional
	$\mathbb Z^2$-grading on $B$.  In this degree-type notation, the symbol $+$
	denotes an arbitrary positive shifted degree.  Thus, up to graded
	skew-symmetry, the four possible types are $(-1,n)$ with
	$n\geq-1$, $(0,0)$, $(0,+)$, and $(+,+)$.
	
	%===========================
	\subsection{Combinatorial preparation for the classification}
	\label{subsec:path-replacement-notation}
	
	\begin{lemma}\label{lem:special-loop-dichotomy}
		Let $n\geq1$ and let $\gamma\in\Gamma_n$. If $\gamma$ contains a special loop $\ep_i$, then $\gamma=\ep_i^n$. Consequently, every element of $\Gamma_n$ either consists entirely of arrows of $Q$ or is a pure power of a special loop.
	\end{lemma}
	
	\begin{proof}
		Since $\ep_i^2\in I^S$, condition \textup{(G3)} in
		\Cref{def:gentle} implies that $\ep_i$ is the unique arrow that
		can follow $\ep_i$ in a relation and also the unique arrow that can precede
		it in a relation. Hence every arrow adjacent to $\ep_i$ in a relation
		concatenation must again be $\ep_i$. Propagating this observation in both
		directions proves that $\gamma=\ep_i^n$.
	\end{proof}
	
	\begin{lemma}
		\label{cycle-not-has-loop}
		Let $\tau=\tau_1\cdots\tau_d$ be a primitive relation cycle with $d>1$. Then none of the arrows $\tau_j$ is a loop, where $1\leq j\leq d$.
	\end{lemma}
	
	\begin{proof}
		Suppose that $\tau_j=\omega$ is a loop. One must have $\omega^2\in I^S$:
		otherwise every positive power of $\omega$ would lie in $\mB$, contradicting
		the finite-dimensionality of $A$. Condition \textup{(G3)} then implies that
		any arrow forming a relation with $\omega$ on either side must itself be
		$\omega$. Hence every arrow on the relation cycle containing $\omega$ equals
		$\omega$, so the underlying primitive relation cycle has length one,
		contradicting $d>1$.
	\end{proof}
	
	\begin{lemma}
		\label{loop-cycle-unique}
		For every vertex $v\in Q_0^S$, there is at most one oriented cycle in $\mB$ based at $v$.
	\end{lemma}
	
	\begin{proof}
		First observe that if $c=c_1\cdots c_m\in\mB$ is an oriented cycle, then $c_m c_1\in I^S$. Otherwise every positive power $c^N$ would lie in $\mB$, contradicting the finite-dimensionality of $A$.
		
		Suppose that $p$ and $q$ are distinct oriented cycles in $\mB$ based at
		$v$. If their first arrows are distinct, then their last arrows are also
		distinct: otherwise the common last arrow would form relations with two
		distinct following arrows, contrary to \textup{(G3)}. Since the closing
		pairs of $p$ and $q$ lie in $I^S$, condition \textup{(G3)} also shows that
		the two cross-boundary pairs do not lie in $I^S$. Hence every positive power
		of $pq$ lies in $\mB$, a contradiction.
		
		It remains to consider the case in which $p$ and $q$ have the same first
		arrow. Condition \textup{(G4)} forces the two paths to agree for as long as
		both continue without passing through a relation. Since they are distinct,
		after interchanging them if necessary we may write $q=pr$, where $r\in\mB$
		is a nontrivial oriented cycle based at $v$ and the boundary between $p$ and
		$r$ is not a relation. Let $\alpha$ be the common first arrow of $p$ and
		$q$, let $\beta$ be the first arrow of $r$, and let $\delta$ be its last
		arrow. Then $\beta\neq\alpha$, while both $\delta\alpha$ and
		$\delta\beta$ belong to $I^S$, because they are the closing pairs of $q$
		and $r$, respectively. This contradicts \textup{(G3)} directly.
	\end{proof}
	
	\begin{lemma}\label{gamma-cases}
		Let $\gamma=\gamma_1\cdots\gamma_n\in\Gamma_n$ with $n\geq2$. Then $\gamma$ contains a repeated arrow if and only if, after cyclically rotating the underlying primitive relation cycle so that its first arrow is $\gamma_1$, one of the following mutually exclusive cases occurs:
		\begin{enumerate}[label=\textup{(\arabic*)}]
			\item $\gamma=\omega^n$, where $\omega$ is a loop with $\omega^2\in I^S$;
			\item $\gamma=\tau^w$, where $\tau=\tau_1\cdots\tau_d$ is a primitive relation cycle of length $d>1$, $w\geq2$, and $n=wd$;
			\item $\gamma=\tau^w\tau_1\cdots\tau_r$, where $\tau=\tau_1\cdots\tau_d$ is a primitive relation cycle of length $d>1$, $w\geq1$, $1\leq r<d$, and $n=wd+r$.
		\end{enumerate}
	\end{lemma}
	
	\begin{proof}
		Each of the three displayed forms contains a repeated arrow. Conversely, suppose that $\gamma_i=\gamma_j$ for some $1\leq i<j\leq n$. Since $\gamma\in\Gamma_n$, we have $\gamma_k\gamma_{k+1}\in I^S$ for $1\leq k<n$. In particular,
		$\gamma_{j-1}\gamma_i=\gamma_{j-1}\gamma_j\in I^S$. Hence
		$c=\gamma_i\gamma_{i+1}\cdots\gamma_{j-1}$
		is a relation cycle. Let $\tau$ be a primitive relation cycle and let $q\geq1$ be such that, up to cyclic rotation, $c=\tau^q$.
		
		By \textup{(G3)}, whenever two consecutive arrows form a relation, either
		arrow uniquely determines the other on the corresponding side. Consequently,
		once the cycle $\tau$ occurs inside $\gamma$, all arrows of $\gamma$ to its
		right and left are forced to continue periodically along $\tau$. After
		cyclically rotating $\tau$ so that $\tau_1=\gamma_1$ and writing
		$d=\ell(\tau)$, this periodicity is expressed by
		\[
		\gamma_j=\tau_{1+((j-1)\bmod d)}
		\quad(1\leq j\leq n).
		\]

		If $\ell(\tau)=1$, write $\tau=\omega$. Then $\omega$ is a loop and $\omega^2\in I^S$. The periodicity just established gives
		$\gamma=\omega^n$,
		which is case \textup{(1)}.
		
		Now suppose that $\ell(\tau)>1$. The arrows of $\tau$ are pairwise
		distinct. Indeed, if an arrow occurred twice in $\tau$, condition
		\textup{(G3)} would force the cyclic word $\tau$ to be periodic with a
		strictly smaller period. This would express $\tau$ as a proper power of a
		shorter relation cycle, contrary to its primitivity.
		
		By Euclidean division, there exist unique integers $w\geq0$ and $0\leq r<\ell(\tau)$ such that
		$n=w\ell(\tau)+r$. The preceding periodicity formula gives
		\[
		\gamma=
		\begin{cases}
			\tau^w, & \text{if } r=0,\\
			\tau^w\tau_1\cdots\tau_r, & \text{if } 1\leq r<\ell(\tau).
		\end{cases}
		\]
		
		If $r=0$, the assumption that $\gamma$ contains a repeated arrow forces $w\geq2$, because the arrows in a single copy of $\tau$ are pairwise distinct. Thus case \textup{(2)} occurs. If $1\leq r<\ell(\tau)$, the same assumption forces $w\geq1$, since for $w=0$ the path $\tau_1\cdots\tau_r$ would be a proper initial segment of $\tau$ and would again contain no repeated arrow. Hence case \textup{(3)} occurs.
	\end{proof}
	
	\begin{lemma}\label{basis-case}
		Let $c\in\mB$. Then every arrow occurs at most once in $c$.
	\end{lemma}
	
	\begin{proof}
		Suppose that an arrow $\alpha$ occurs twice in $c$. The segment beginning with the first occurrence of $\alpha$ and ending immediately before the second is an oriented cycle in $\mB$. Its closing pair occurs as a subpath of $c$, so every positive power of this cycle again lies in $\mB$. These powers represent nonzero basis elements of arbitrarily large length, contradicting the finite-dimensionality of $A$.
	\end{proof}
	
	We now fix the notation and standing conventions used throughout the classification.
	
	\begin{enumerate}[label=\textup{(N\arabic*)}]
		\item The two elementary cochains are written as
		$f=\para{x}{a}\in B^m=\mK\Para{\Gamma_{m+1}}{\mB}$ and
		$g=\para{y}{b}\in B^n=\mK\Para{\Gamma_{n+1}}{\mB}$, where
		$m,n\geq-1$. The tensor $f\otimes g$ has degree pair $(m,n)$.
		The paths $x$ and $y$ are the supports of the cochains, while $a$ and $b$
		are their value paths in $\mB$.
		
		\item If a nontrivial path $c\in\mB$ is decomposed into arrows, we write
		\[
		c=c_1\cdots c_{\ell(c)}.
		\]
		In particular, whenever the value paths $a$ and $b$ are nontrivial, we use
		the fixed notation
		$a=\alpha_1\cdots\alpha_{\ell(a)}$ and $b=\beta_1\cdots\beta_{\ell(b)}$.
		If $c=e_i$ is a trivial path, then $\ell(c)=0$, and no arrow symbols are
		attached to $c$.
		
		\item Let $\alpha$ be an arrow, let $c\in\mB$ be parallel to $\alpha$, and let
		$d\in\mB$. Define the path-replacement operator by
		\[
		\Sub_\alpha^c(d)
		=
		\sum_{d=p\alpha q}\pi(pcq),
		\]
		where the sum is taken over all factorizations $d=p\alpha q$, equivalently,
		over all occurrences of $\alpha$ in $d$. The paths $p$ and $q$ may be
		trivial, and terms whose image under $\pi$ is zero are omitted.
		
		\item For an arrow $\alpha$ and a path $c\in\mB$, possibly trivial, define
		\[
		L_\alpha(c)=
		\begin{cases}
			p, & \text{if } c=p\alpha,\\
			0, & \text{otherwise},
		\end{cases}
		\qquad
		R_\alpha(c)=
		\begin{cases}
			q, & \text{if } c=\alpha q,\\
			0, & \text{otherwise.}
		\end{cases}
		\]
		Here again $p$ and $q$ may be trivial. In particular,
		$L_\alpha(e_i)=R_\alpha(e_i)=0$.
		
		\item For a statement $P$, let $\one_P$ denote its indicator, equal to $1$ if
		$P$ holds and to $0$ otherwise.
		
		\item For $N\geq1$, a relation concatenation is written as
		\[
		\gamma=\gamma_1\cdots\gamma_N\in\Gamma_N.
		\]
		An element of $\Gamma_0$ is written simply as $e_i$. In periodic cases,
		we use the notation $\omega^N$, $\tau^w$, and
		$\tau^w\tau_1\cdots\tau_r$, where $1\leq r<\ell(\tau)$, as in
		\Cref{gamma-cases}. Here $\omega$ is a loop satisfying $\omega^2\in I^S$ and
		\[
		\tau=\tau_1\cdots\tau_{\ell(\tau)}
		\]
		is a primitive relation cycle. The indices of the arrows of $\tau$ are read
		cyclically whenever necessary.
		
		\item The symbols $p$ and $q$ denote path factors attached on the left and on
		the right, respectively. They may be trivial. Any additional local condition
		on $p$ or $q$, such as being nontrivial or forming a boundary loop or cycle,
		is stated where it is used.
		
		\item In path-replacement arguments, $\lambda$ and $\mu$ denote the local left
		and right factors immediately adjacent to the arrow being replaced. Each is
		either the appropriate trivial path or an adjacent special loop; the precise
		admissibility conditions are specified in the relevant statements.
		
		\item Whenever a formal concatenation of paths occurs in a cochain value,
		an outer coefficient, or a path replacement, its image in $A$ is understood
		via the canonical projection $\pi:\mK Q^S\to A$. When needed, this image is
		expressed in the basis $\mB$; zero images are omitted.
		
		\item From this point onward, tensor-product subscripts are suppressed whenever
		they are determined by the context. Recall that
		\[
		E=\bigoplus_{i\in Q_0}\mK e_i,
		\]
		and that, for a right $E$-module $W$ and a left $E$-module $V$, there is a
		canonical identification of $\mK$-vector spaces
		\[
		W\otimes_E V
		\cong
		\bigoplus_{i\in Q_0}
		We_i\otimes_{\mK}e_iV.
		\]
		Thus an $E$-relative tensor may be regarded as a $\mK$-tensor whose adjacent
		components have matching vertex idempotents. In particular, tensors of paths
		occurring in the relative bar complex, the CS resolution, and the comparison
		maps are composable.
		
		Tensor products of cochains occurring as inputs of the $L_\infty$-operations
		are taken over $\mK$. Nevertheless, in the basis-level formulas considered
		below, only the composable path components selected by the supports and the
		comparison maps can contribute. We therefore write simply $\otimes$ for both
		tensor products from this point onward.
	\end{enumerate}
	
	With these preparations in place, we now begin the classification of the four
	types, starting with brackets having a shifted degree $-1$ input.
	
	%===========================
	\section{Brackets of type \texorpdfstring{$(-1,n)$}{(-1,n)}}
	\label{sec:degree-minus-one-classification}
	\label{subsec:degree-minus-one}
	
	This section gives a complete basis-level classification of brackets of type
	$(-1,n)$, that is,
	\[
	l_2:B^{-1}\otimes B^n\longrightarrow B^{n-1},\quad n\geq-1.
	\]
	Formula \eqref{eq:general-l2} applies to degree pairs $(m,n)$ with
	$m,n\geq0$.  Up to graded skew-symmetry, the only remaining type is
	$(-1,n)$ with $n\geq-1$, so the shifted degree $-1$ input is placed first
	throughout this section.
	
	A basis element of $B^{-1}$ has vertex support and value
	$a\in e_iAe_i\cap\mB$.  Accordingly, we distinguish whether $a$ is the
	trivial path, a loop, or a cycle in $\mB$ of length greater than one.  The
	last case is treated separately for $n=0$ and $n\geq1$.
	
	\begin{proposition}
		\label{prop:degree-minus-one-brackets}
		Let $f=\para{e_i}{a}\in B^{-1}$ and
		$g=\para{y}{b}\in B^n$, where $n\geq-1$ and $a\in e_iAe_i$ is a path in $\mB$ based at $i$. Then the following
		statements hold.
		
		\begin{enumerate}
			\item If $n=-1$, then $l_2(f\otimes g)=0$.

			\item If $n\geq0$ and $a=e_i$, then $l_2(f\otimes g)=0$.
			
			\item Suppose that $n\geq0$ and $a=\omega$ is a loop at $i$.
			A nonzero bracket is possible only when
			$
			y=\omega^{n+1}.
			$
			With the convention $\omega^0=e_i$, one has
			\[
			l_2(
			\para{e_i}{\omega}
			\otimes
			\para{\omega^{n+1}}{b}
			)
			=
			(-1)^{n+1}
			\sum_{j=0}^{n}
			(-1)^j
			\para{\omega^n}{b}
			=
			\begin{cases}
				-\para{\omega^n}{b},
				& \text{if } n\text{ is even},\\
				0,
				& \text{if } n\text{ is odd}.
			\end{cases}
			\]
			
			\item Suppose that
			$a=\alpha_1\cdots\alpha_{\ell(a)}$ with
			$\ell(a)>1$
			is a cycle in $\mB$ based at $i$.
			
			If $n=0$, then
			\[
			l_2(
			\para{e_i}{a}
			\otimes
			\para{y}{b}
			)
			=
			-\para{e_i}{\Sub_y^b(a)}.
			\]
			Consequently,
			\[
			l_2(
			\para{e_i}{a}
			\otimes
			\para{y}{b}
			)\neq0
			\Longleftrightarrow
			\Sub_y^b(a)\neq0,
			\]
			and the nonzero substitutions are classified by
			\Cref{lem:replacement-criterion}.
			
			If $n\geq1$, write $y=y_1\cdots y_{n+1}$.
			Then
			\begin{align}
				l_2(
				\para{e_i}{a}\otimes\para{y}{b}
				)
				={}&
				(-1)^{n+1}
				\one_{y_1=\alpha_{\ell(a)}}
				\para{y_2\cdots y_{n+1}}
				{\pi(L_{\alpha_{\ell(a)}}(a)b)}
				\notag\\
				&-
				\one_{y_{n+1}=\alpha_1}
				\para{y_1\cdots y_n}
				{\pi(bR_{\alpha_1}(a))}.
				\label{eq:degree-minus-one-cycle}
			\end{align}
			As usual, terms with zero values under $\pi$ are omitted.
			
			If both terms in \eqref{eq:degree-minus-one-cycle} are nonzero, their
			support paths are distinct. Consequently, the bracket is nonzero if
			and only if at least one of the following two pairs of conditions holds:
			\[
			y_1=\alpha_{\ell(a)},\qquad
			\pi(L_{\alpha_{\ell(a)}}(a)b)\neq0,
			\]
			or
			\[
			y_{n+1}=\alpha_1,\qquad
			\pi(bR_{\alpha_1}(a))\neq0.
			\]
		\end{enumerate}
		
		Brackets with degree pair $(n,-1)$ are obtained from those of type
		$(-1,n)$ by graded skew-symmetry:
		\[
		l_2(g\otimes f)
		=
		-(-1)^n l_2(f\otimes g).
		\]
	\end{proposition}
	
	\begin{proof}
		The shifted complex
		$B=B_{\mathrm{CS}}^*(A)[1]$
		is concentrated in degrees at least $-1$. Since $l_2$ has degree zero,
		the case $n=-1$ takes values in $B^{-2}=0$, proving the first assertion.
		
		Assume now that $n\geq0$. Since $G^*f$ is a Hochschild $0$-cochain, the
		composition
		$G^*f\circ G^*g$
		vanishes.
		
		For $n=0$ and $\gamma\in\Gamma_0$, the remaining composition gives
		\[
		l_2(f\otimes g)(\gamma)
		=
		-\one_{\gamma=e_i}
		\para{y}{b}G_1(\varpi(a)).
		\]
		
		For $n\geq1$, let
		$\gamma=\gamma_1\cdots\gamma_n\in\Gamma_n$.
		Inserting the value $a$ at every occurrence of the vertex $i$ along
		$\gamma$ gives
		\begin{align}
			&l_2(f\otimes g)(\gamma)
			=
			(-1)^{n+1}
			\one_{e_{s(\gamma_1)}=e_i}
			\para{y}{b}G_{n+1}
			(
			\varpi(a)\otimes
			\gamma_1\otimes\cdots\otimes\gamma_n
			)
			\notag\\
			&+
			(-1)^{n+1}
			\sum_{\substack{1\leq j\leq n\\ e_{t(\gamma_j)}=e_i}}
			(-1)^j
			\para{y}{b}G_{n+1}
			(
			\gamma_1\otimes\cdots\otimes\gamma_j
			\otimes\varpi(a)
			\otimes\gamma_{j+1}\otimes\cdots\otimes\gamma_n
			).
			\label{eq:degree-minus-one-general}
		\end{align}
		
		If $a=e_i$, then $\varpi(a)=0$, so the second assertion follows.
		
		Suppose that $a=\omega$ is a loop. Finite-dimensionality implies
		$\omega^2\in I^S$.
		By the gentle uniqueness conditions, every relation concatenation
		containing $\omega$ is a pure power of $\omega$. Hence a nonzero term
		forces
		$y=\omega^{n+1}$.
		For $n\geq1$, the corresponding output support is
		$\gamma=\omega^n$.
		There are $n+1$ possible insertion positions, all producing the same
		value $b$, and their total coefficient is
		$(-1)^{n+1}\sum_{j=0}^{n}(-1)^j$.
		For $n=0$, the same formula follows from the preceding computation. This
		proves the third assertion.
		
		Finally, suppose that
		$a=\alpha_1\cdots\alpha_{\ell(a)}$ with $\ell(a)>1$ is a cycle in $\mB$. Its closing product satisfies
		$\alpha_{\ell(a)}\alpha_1\in I^S$;
		otherwise every positive power of $a$ would lie in $\mB$ and hence be nonzero in $A$,
		contradicting the finite-dimensionality of $A$.
		
		When $n=0$, the preceding evaluation on $\Gamma_0$ is governed by
		$G_1$. By \Cref{basis-case}, the arrow $y$ occurs in $a$ at most once.
		Expanding $G_1(\varpi(a))$ therefore gives precisely the path-replacement
		sum $\Sub_y^b(a)$. Taking account of the global negative sign yields
		\[
		l_2(
		\para{e_i}{a}\otimes\para{y}{b}
		)
		=
		-\para{e_i}{\Sub_y^b(a)}.
		\]
		
		Now let $n\geq1$. At an internal insertion position
		$1\leq j\leq n-1$,
		the path $a$ is an internal input of $G_{n+1}$. By the definition of
		the comparison map, every internal input contributing to an
		$(n+1)$-arrow relation concatenation must be a single arrow. Since
		$\ell(a)>1$, all internal insertions vanish.
		
		At the left endpoint, write
		$a=L_{\alpha_{\ell(a)}}(a)\alpha_{\ell(a)}$.
		After applying the cochain $\para{y}{b}$, the resulting term can be
		nonzero only when
		$y=\alpha_{\ell(a)}\gamma$.
		In this case its value under $\pi$ is
		$\pi(L_{\alpha_{\ell(a)}}(a)b)$,
		and its insertion sign is $(-1)^{n+1}$. This gives the first term in
		\eqref{eq:degree-minus-one-cycle}.
		
		At the right endpoint, write
		$a=\alpha_1R_{\alpha_1}(a)$.
		The resulting term can be nonzero only when
		$y=\gamma\alpha_1$.
		In this case its value under $\pi$ is
		$\pi(bR_{\alpha_1}(a))$.
		The insertion sign at $j=n$ is
		$(-1)^{n+1+n}=-1$,
		which gives the second term in
		\eqref{eq:degree-minus-one-cycle}.
		
		It remains to show that the two nonzero terms cannot cancel. If their
		support paths were equal, then
		$y_2\cdots y_{n+1} = y_1\cdots y_n$.
		Hence
		$y_1=y_2=\cdots=y_{n+1}$.
		If both terms are nonzero, their endpoint conditions then imply
		$\alpha_{\ell(a)} = y_1 = y_{n+1} = \alpha_1$.
		This contradicts \Cref{basis-case}, since the same arrow would occur
		twice in the path $a\in\mB$ and $\ell(a)>1$. Thus the two nonzero
		terms have distinct supports and cannot cancel.
	\end{proof}
	
	Thus \Cref{prop:degree-minus-one-brackets}, together with
	\Cref{lem:replacement-criterion} in the case $n=0$, gives the complete
	basis-level classification of brackets of type $(-1,n)$, $n\geq-1$, up to
	graded skew-symmetry.  A compact lookup table is given in
	\Cref{app:degree-minus-one-lookup}.
	
	We now turn to the remaining three types $(0,0)$, $(0,+)$, and $(+,+)$, for
	which both inputs have nonnegative shifted degree.
	%===========================
	\section{Brackets of type \texorpdfstring{$(0,0)$}{(0,0)}}\label{sec:degree-zero-classification}
	
	This section gives a complete basis-level classification of brackets of type
	$(0,0)$, that is,
	\[
	l_2:B^0\otimes B^0\longrightarrow B^0.
	\]
	We first establish a finite-dimensional criterion for when path replacement in
	$\mB$ can be nonzero and then analyze the cases $x=y$ and $x\neq y$.
	
	%===========================
	\subsection{General formula for brackets of type \texorpdfstring{$(0,0)$}{(0,0)}}
	
	\begin{proposition}\label{DZSF}
		For $x,y\in\Gamma_1$ and paths $a,b\in\mB$ parallel to $x,y$, respectively,
		\[
		l_2(\para{x}{a}\otimes \para{y}{b})
		=
		\para{y}{\Sub_x^a(b)}-
		\para{x}{\Sub_y^b(a)}.
		\]
		In particular, every nonzero summand is supported on $x$ or on $y$.
	\end{proposition}
	
	\begin{proof}
		Let $\gamma\in\Gamma_1$.  By \Cref{eq:general-l2},
		\[
		l_2(\para{x}{a}\otimes \para{y}{b})(\gamma)
		=
		\para{x}{a}G_1(\varpi(\para{y}{b}(\gamma)))
		-
		\para{y}{b}G_1(\varpi(\para{x}{a}(\gamma))).
		\]
		If $\gamma\neq x,y$, both terms vanish.  At $\gamma=y$, the first term is
		$\para{x}{a}G_1(\varpi(b))$, which is exactly the sum obtained by replacing every
		occurrence of $x$ in $b$ by $a$, namely $\Sub_x^a(b)$.  The second term
		contributes at $\gamma=y$ only when $x=y$, in which case it is already included
		in $-\para{x}{\Sub_y^b(a)}$.  The evaluation at $\gamma=x$ is symmetric and
		gives the second substitution term.  This proves the formula and the support
		claim.
	\end{proof}
	
	\begin{corollary}\label{00-x-equal-to-a}
		Let $\para{x}{a},\para{y}{b}\in B^0$.  If $a=x$, then
		\[
		l_2(\para{x}{x}\otimes\para{y}{b})
		=
		\begin{cases}
			(\num{x}{b}-1)\para{y}{b}, & \text{if } x=y,\\
			\num{x}{b}\para{y}{b}, & \text{if } x\neq y.
		\end{cases}
		\]
		If $b=y$, then
		\[
		l_2(\para{x}{a}\otimes\para{y}{y})
		=
		\begin{cases}
			(1-\num{y}{a})\para{x}{a}, & \text{if } x=y,\\
			-\num{y}{a}\para{x}{a}, & \text{if } x\neq y.
		\end{cases}
		\]
		In particular, these brackets are nonzero precisely when the displayed
		counting coefficients give a nonzero linear combination.
	\end{corollary}
	
	\begin{proof}
		If $a=x$, then replacing an occurrence of $x$ in $b$ by $x$ leaves $b$
		unchanged, so $\Sub_x^x(b)=\num{x}{b}b$.  Since $x$ and $y$ are arrows,
		$\num{y}{x}=1$ if $x=y$ and $\num{y}{x}=0$ otherwise.  The first formula
		therefore follows from \Cref{DZSF}.
		
		The case $b=y$ is obtained by interchanging the two inputs and using
		\eqref{sym}.  Equivalently, applying \Cref{DZSF} directly gives
		$\Sub_x^a(y)=\num{x}{y}a$ and $\Sub_y^y(a)=\num{y}{a}a$,
		which yields the stated formula.
	\end{proof}
	
	\begin{remark}
		The nonvanishing in \Cref{00-x-equal-to-a} is determined entirely by the
		occurrence numbers in the displayed formulas.
	\end{remark}
	
	\begin{lemma}\label{x-is-loop}
		For any $\para{x}{a}\in B^0$, if $x$ is a loop, then
		$a=x$ or $a=e_{s(x)}$.  If $a$ is a loop, then $x=a$.
	\end{lemma}
	
	\begin{proof}
		Let $i=s(x)=t(x)$.  If $x$ is a loop and $a$ is trivial, then necessarily
		$a=e_i$.  Suppose that $a$ is nontrivial.  Since $x,a\in\mB$ are both
		oriented cycles based at $i$, \Cref{loop-cycle-unique} gives $a=x$.  This
		proves the first assertion.
		
		If $a$ is a loop, then $a\parallel x$ implies that $x$ is also a loop at the
		same vertex.  Applying the first assertion to $\para{x}{a}$ gives
		$a=x$ or $a=e_{s(x)}$.  Since $a$ is an arrow, it is nontrivial, and hence
		$a=x$.
	\end{proof}
	
	We shall use the following elementary criterion repeatedly.  It is where the
	finite-dimensional hypothesis enters the analysis of type $(0,0)$.
	
	\begin{lemma}\label{lem:replacement-criterion}
		Let $b=\beta_1\cdots\bet\in\mB$, let $x\in\Gamma_1$, and let
		$a\in\mB$ be parallel to $x$.  If $\ell(b)=1$, then
		\[
		\Sub_x^a(b)=
		\begin{cases}
			a, & \text{if } x=b,\\
			0, & \text{if } x\neq b.
		\end{cases}
		\]
		
		Assume now that $\ell(b)>1$.  If $x$ does not occur in $b$, then
		$\Sub_x^a(b)=0$.  Otherwise let $i_0$ be the unique index such that
		$x=\beta_{i_0}$.  Define
		\[
		\mL_{i_0}=
		\{e_{s(x)}\}
		\cup
		\{\beta_{i_0-1} \mid i_0>1,\ \beta_{i_0-1}\in\Sp\},
		\]
		and
		\[
		\mR_{i_0}=
		\{e_{t(x)}\}
		\cup
		\{\beta_{i_0+1} \mid i_0<\ell(b),\ \beta_{i_0+1}\in\Sp\}.
		\]
		Every nontrivial element of $\mL_{i_0}$ or $\mR_{i_0}$ is therefore a
		special loop and is idempotent in $A$.
		
		Then $\Sub_x^a(b)$ is given as follows.
		
		\begin{enumerate}
			\item If $1<i_0<\ell(b)$, then
			\[
			\Sub_x^a(b)=
			\begin{cases}
				b, & \text{if } a=\lambda x\mu\text{ for some }
				\lambda\in\mL_{i_0},\ \mu\in\mR_{i_0},\\
				0, & \text{otherwise.}
			\end{cases}
			\]
			
			\item If $i_0=1$, then
			\[
			\Sub_x^a(b)=
			\begin{cases}
				\beta_2\cdots\bet,
				& \text{if } x=\beta_1\text{ is a loop and }a=e_{s(x)},\\
				\pi(pb),
				& \text{if } a=px\mu\text{ for some path }p
				\text{ based at }s(x)\text{ and }\mu\in\mR_1,\\
				0, & \text{otherwise.}
			\end{cases}
			\]
			The path $p$ may be trivial.
			
			\item If $i_0=\ell(b)$, then
			\[
			\Sub_x^a(b)=
			\begin{cases}
				\beta_1\cdots\beta_{\ell(b)-1},
				& \text{if } x=\bet\text{ is a loop and }a=e_{s(x)},\\
				\pi(bq),
				& \text{if } a=\lambda xq\text{ for some path }q
				\text{ based at }t(x)\text{ and }\lambda\in\mL_{\ell(b)},\\
				0, & \text{otherwise.}
			\end{cases}
			\]
			The path $q$ may be trivial.
		\end{enumerate}
	\end{lemma}
	
	\begin{proof}
		By the formula for $G_1$,
		\[
		G_1(\varpi(b))
		=
		\sum_{i=1}^{\ell(b)}
		\beta_1\cdots\beta_{i-1}\otimes \beta_i
		\otimes \beta_{i+1}\cdots\bet.
		\]
		Hence
		\[
		\Sub_x^a(b)
		=
		\sum_{i=1}^{\ell(b)}
		\beta_1\cdots\beta_{i-1} \para{x}{a}(\beta_i)
		\beta_{i+1}\cdots\bet.
		\]
		A summand can be nonzero only if $x=\beta_i$.  By \Cref{basis-case},
		$b$ has no repeated arrows, so there is at most one such index.
		
		If $\ell(b)=1$, the displayed formula gives $\Sub_x^a(b)=a$ when
		$x=b$, and zero otherwise.  Thus assume $\ell(b)>1$, and let $i_0$ be
		the unique index with $x=\beta_{i_0}$.
		
		The only possible obstructions to the nonvanishing of
		\[
		\beta_1\cdots\beta_{i_0-1} a
		\beta_{i_0+1}\cdots\bet
		\]
		occur at the new interfaces created by the replacement.
		A monomial relation in $I^S$ kills the product, except when the interface
		is a square $\varepsilon^2$ of a special loop, which reduces in $A$ to
		$\varepsilon$.
		
		Suppose first that $1<i_0<\ell(b)$.  If $a=e_{s(x)}$, then $x$ is a
		loop, and deleting it joins $\beta_{i_0-1}$ directly to $\beta_{i_0+1}$.
		Since $\beta_{i_0-1}x\notin I^S$ and $\beta_{i_0+1}\neq x$,
		condition \textup{(G4)} forces
		$\beta_{i_0-1}\beta_{i_0+1}\in I^S$.
		By \Cref{basis-case}, these two arrows are distinct, so this relation
		is not a special-loop square.  The replacement is therefore zero in $A$.
		
		Now let $a$ be nontrivial and suppose that the replacement is nonzero.
		Since $\beta_{i_0-1}\beta_{i_0}$ is a subpath in $\mB$ of $b$, the gentle
		uniqueness conditions imply that the first arrow of $a$ must be
		$\beta_{i_0}$, unless $\beta_{i_0-1}$ is a special loop and the left
		interface is the idempotent square $\beta_{i_0-1}^2$.  Similarly, from
		the subpath in $\mB$ $\beta_{i_0}\beta_{i_0+1}$, the last arrow of $a$
		must be $\beta_{i_0}$, unless $\beta_{i_0+1}$ is a special loop and the
		right interface is $\beta_{i_0+1}^2$.
		In the left idempotent-square alternative, $a\parallel x$ and
		\Cref{x-is-loop,basis-case} exclude $a=\beta_{i_0-1}$; applying
		\textup{(G4)} to the next arrow of $a$ forces it to be $x$.
		The right alternative is symmetric.
		Since $a\in\mB$ and has no repeated arrows by \Cref{basis-case},
		these possibilities force $a=\lambda\beta_{i_0}\mu$
		with $\lambda\in\mL_{i_0}$ and $\mu\in\mR_{i_0}$.  Conversely, every
		such $a$ reconstructs $b$ after reducing the possible idempotent squares.
		This proves the interior case.
		
		Now let $i_0=1$.  If $a=e_{s(x)}$, then $a\parallel x$ forces $x$ to be
		a loop, and the replacement gives $\beta_2\cdots\bet$.  Otherwise the
		right interface with $\beta_2$ must be nonzero.  The same gentle
		uniqueness argument shows that the terminal part of $a$ is either
		$x=\beta_1$ or $\beta_1\beta_2$ with $\beta_2$ a special loop.  Since
		$a\parallel x$, the part of $a$ preceding $x$ is a path $p$ based at
		$s(x)$, possibly trivial.  Hence
		$a=px\mu$, where $\mu\in\mR_1$,
		and the replacement is $\pi(pb)$.  No other right interface can survive.
		
		The case $i_0=\ell(b)$ is left-right symmetric.  If $a=e_{s(x)}$, one
		obtains the right endpoint loop deletion.  Otherwise
		$a=\lambda xq$, $\lambda\in\mL_{\ell(b)}$, where $q$ is a path based at $t(x)$, possibly trivial, and the replacement
		is $\pi(bq)$.  This exhausts all possibilities.
	\end{proof}
	
	\begin{remark}
		We shall use three of the alternatives in \Cref{lem:replacement-criterion}
		repeatedly below.  Replacing an endpoint loop by the trivial path will be
		referred to as loop deletion; the two endpoint alternatives will be called
		left and right endpoint extensions; and a replacement in which an adjacent
		special loop is inserted and then absorbed by $\varepsilon^2=\varepsilon$
		will be referred to as idempotent absorption.  These terms describe only the
		corresponding local replacement mechanisms.
	\end{remark}
	
	The rest of the section proves the following classification statement.
	
	\begin{proposition}\label{prop:degree-zero-classification}
		Let $f=\para{x}{a}\in B^0$ and $g=\para{y}{b}\in B^0$.  Up to the
		skew-symmetry relation \eqref{sym}, the nonzero brackets of type $(0,0)$ are
		exactly the following:
		\begin{enumerate}
			\item the cases in \Cref{00-x-equal-to-a} for which the displayed
			counting coefficient is nonzero in $\mK$;
			\item the remaining cases displayed in
			\eqref{eq:xneqy-left-loop-deletion}--\eqref{eq:xneqy-idempotent-absorption},
			subject to the hypotheses stated there.
		\end{enumerate}
		No other bracket of type $(0,0)$ occurs.
	\end{proposition}
	
	To verify \Cref{prop:degree-zero-classification}, we treat first the case
	$x=y$ and then the case $x\neq y$.  The cases $a=x$ or $b=y$ are already
	covered by \Cref{00-x-equal-to-a} and will not be listed again unless needed
	to close a subcase.  When $x\neq y$, skew-symmetry allows us to classify the
	cases for which $\Sub_x^a(b)\neq0$ and then record whether the opposite term
	$-\Sub_y^b(a)$ also contributes.
	
	%===========================
	\subsection{The case \texorpdfstring{$x=y$}{x=y} and \texorpdfstring{$a\neq b$}{a not equal b}}\label{subsec:degree-zero-xeqy}
	Assume first that the two cochains have the same support arrow $x=y$, while
	their value paths are distinct.  By \Cref{DZSF}, the only possible support of
	the bracket is $x$, and
	\[l_2(\para{x}{a}\otimes \para{x}{b})(x)=\Sub_x^a(b)-\Sub_x^b(a).\]
	
	\begin{remark}
		The two substitutions $\Sub_x^a(b)$ and $\Sub_x^b(a)$ are interchanged by
		swapping the two inputs.  Since $l_2$ is skew-symmetric in the shifted
		grading, it is enough to analyze the following two situations:
		\begin{enumerate}
			\item $\Sub_x^a(b)\neq0$ and $\Sub_x^b(a)=0$;
			\item $\Sub_x^a(b)\neq0$ and $\Sub_x^b(a)\neq0$.
		\end{enumerate}
		The case $\Sub_x^a(b)=0$ and $\Sub_x^b(a)\neq0$ is obtained from the first by
		exchanging the two inputs.
	\end{remark}
	
	A necessary condition for $\Sub_x^a(b)\neq0$ is that
	$b$ be nontrivial and contain the arrow $x$.  Thus $x=\bi$ for a unique index
	$i_0$, by \Cref{basis-case}.  We keep the vertex-label convention
	$s(\bi)=i_0$.
	
	We organize the verification by the length of $a$.
	
	%=======================================
	\subsubsection{$\ell(a)=0$}
	Let $a$ be trivial.  Since $a\parallel x$, the arrow $x=\bi$ is a loop.  In
	this case $\Sub_x^b(a)=0$ and
	\[
	l_2(\para{\bi}{\ei}\otimes\para{\bi}{b})
	=
	\para{\bi}{\Sub_{\bi}^{\ei}(b)}.
	\]
	By \Cref{x-is-loop}, $b=\ei$ or $b=\bi$.  Since $a\neq b$, the only nonzero
	possibility is $b=\bi$, and then
	\[
	l_2(\para{\bi}{\ei}\otimes\para{\bi}{\bi})
	=
	\para{\bi}{\ei}.
	\]
	This case is already covered by \Cref{00-x-equal-to-a}.
	
	%===========================================
	\subsubsection{$\ell(a)=1$}
	If $a$ is a loop, then \Cref{x-is-loop} gives $a=x$, so this case is already
	covered by \Cref {00-x-equal-to-a}.  We may therefore assume that $a$ is a non-loop
	arrow.  If $a=x$ or $b=x$, we are again in \Cref {00-x-equal-to-a}; hence assume
	$a\neq x$ and $b\neq x$.  Since $a\parallel x$, the arrows $a$ and $x$ are
	distinct parallel arrows:
	\[
	\begin{tikzcd}
		\bullet && \bullet
		\arrow["a"', shift right, from=1-1, to=1-3]
		\arrow["x", shift left, from=1-1, to=1-3]
	\end{tikzcd}
	\]
	If $\ell(b)=1$, then $b$ is an arrow parallel to $x$.  Since $a$ and $x$
	already exhaust the two possible outgoing arrows and $a\neq b$, we must have
	$b=x$, which is covered by \Cref {00-x-equal-to-a}.  If $\ell(b)>1$, then
	\Cref{lem:replacement-criterion} shows that a nonzero replacement of $x$ in
	$b$ requires the value path $a$ to contain $x$ as the distinguished replaced
	arrow.  Since $a$ is a single arrow distinct from $x$, this is impossible.
	Thus $\Sub_x^a(b)=0$, and no new nonzero case occurs.
	
	%==============================
	\subsubsection{$\ell(a)>1$}\label{subsec:xeqy-ella-gt-one}
	By \Cref{x-is-loop}, the support arrow $x=\bi$ is not a loop.
	
	\begin{enumerate}
		\item Suppose that $\num{\bi}{a}\neq0$.  By \Cref{basis-case},
		$\num{\bi}{a}=1$.  Since $a\parallel\bi$, the portions of $a$ before and after its unique
		occurrence of $\bi$, whenever nontrivial, are oriented cycles in
		$\mB$ based at $s(\bi)$ and $t(\bi)$, respectively.
		By \Cref{loop-cycle-unique},
		there is at most one such nontrivial cycle at each endpoint.
		Let $p$ and $q$, when they exist, denote these possible nontrivial
		oriented cycles at $s(\bi)$ and $t(\bi)$, respectively, independently of
		the particular choices of $a$ and $b$.  We use
		\[\Pi_1=\{\bi, p\bi, \bi q, p\bi q\}\]
		as a list of possible forms, omitting nonexistent terms or terms not lying in $\mB$.
		Every parallel path in $\mB$ containing $\bi$ has one of these forms.
		Since $\Sub_x^a(b)\neq0$, the path $b$ also contains $\bi$, so
		$b\in\Pi_1$; moreover, $a\in\Pi_1\setminus\{\bi\}$ because $\ell(a)>1$.
		
		A direct substitution check gives
		\[l_2(\para{\bi}{p\bi}\otimes\para{\bi}{p\bi})=
		\begin{cases}
			\para{\bi}{p\bi}-\para{\bi}{p\bi}=0, & \text{if } p^2=p,\\
			0-0=0, & \text{if } p^2=0.
		\end{cases}\]
		All other admissible choices of $a\in\Pi_1\setminus\{\bi\}$ and $b\in\Pi_1$ likewise give zero by the same substitution calculation.
		Hence the case $\num{x}{a}=1$ gives no new nonzero bracket of type $(0,0)$.
		
		\item Suppose that $\num{\bi}{a}=0$.  Since $\ell(a)>1$ and
		$\bi\parallel a$, the local configuration contains the arrow $\bi$ and a
		distinct path $a\in\mB$ joining the same two endpoints:
		\[
		\begin{tikzcd}
			{i_0} && {i_0+1}
			\arrow["\bi", shift left, from=1-1, to=1-3]
			\arrow["a"', shift right, dotted, from=1-1, to=1-3]
		\end{tikzcd}
		\]
		
		Since $\Sub_x^a(b)\neq0$, the path $b$ contains $\bi=x$.  If
		$\ell(b)>1$, then \Cref{lem:replacement-criterion}, together with the
		fact that $x$ is not a loop, shows that every nonzero replacement of
		$x$ in $b$ forces the value path $a$ to contain $x$.  This contradicts
		$\num{x}{a}=0$.  Hence no new nonzero case with $\ell(b)>1$ occurs.
		
		It follows that the only remaining possibility is $\ell(b)=1$.  Since
		$b$ contains $x=\bi$, we must have $b=x$.  In this case
		\[
		l_2(\para{x}{a}\otimes\para{x}{x})=\para{x}{a},
		\]
		because $\num{x}{a}=0$.  Thus a nonzero bracket does occur, but it is
		precisely the case already covered by
		\Cref {00-x-equal-to-a}, rather than a new local configuration.
	\end{enumerate}
	
	Thus, when $x=y$ and $a\neq b$, all nonzero brackets are already covered by
	\Cref {00-x-equal-to-a}.
	
	%===============================================
	\subsection{The case \texorpdfstring{$x\neq y$}{x not equal y}}\label{subsec:degree-zero-xneqy}
	We now assume $x\neq y$.  By \Cref{DZSF}, the two possible supports of the
	bracket are $x$ and $y$:
	\begin{align*}
		l_2(\para{x}{a}\otimes \para{y}{b})(x)
		&=-\Sub_y^b(a),\\
		l_2(\para{x}{a}\otimes \para{y}{b})(y)
		&=\Sub_x^a(b).
	\end{align*}
	The two formulas are the same replacement calculation applied in the
	two possible directions.  Hence, up to \eqref{sym}, it is enough to classify
	the cases in which the $y$-supported term $\Sub_x^a(b)$ is nonzero.
	For each representative we also determine whether the $x$-supported term
	$-\Sub_y^b(a)$ contributes.
	
	A necessary condition for $\Sub_x^a(b)$ to be nonzero is that
	$x=\bi$ for a unique index $i_0$.  As above, we keep the convention
	$s(\bi)=i_0$.  We again split the verification according to $\ell(a)$.
	
	%==========================
	\subsubsection{$\ell(a)=0$}\label{subsec:xneqy-ella-zero}
	Let $a$ be trivial.  Then $\Sub_y^b(a)=0$ and
	$x=\bi$ is a loop.  By the replacement criterion, an internal occurrence
	$1<i_0<\ell(b)$ cannot contribute: deleting $\bi$ would join the two
	neighboring arrows, and the resulting product is zero.  Thus only the two
	endpoint occurrences can survive.
	
	\begin{enumerate}
		\item Suppose that $i_0=1$.  Then $x=\beta_1$ is a loop.  Since
		$\ell(b)>1$, the support arrow $y$ cannot be a loop by \Cref{x-is-loop}.
		The source of $y$ is the source of $b$, and the gentle valency condition
		shows that the only possible non-loop outgoing arrow is $\beta_2$.
		Hence $y=\beta_2$.  We may write
		$b=\beta_1\beta_2 q$,
		where $q$ is a path based at $t(\beta_2)$, possibly trivial.  The local
		configuration is
		\[
		\begin{tikzcd}
			1 && \bullet
			\arrow["{x=\beta_1}", from=1-1, to=1-1, loop, in=145, out=215, distance=10mm]
			\arrow["{y=\beta_2}", from=1-1, to=1-3]
			\arrow["q", dotted, from=1-3, to=1-3, loop, in=325, out=35, distance=10mm]
		\end{tikzcd}
		\]
		and
		\begin{align}
			l_2(\para{\beta_1}{e_1}\otimes \para{\beta_2}{\beta_1\beta_2 q})
			&=\para{\beta_2}{\beta_2 q}.
			\label{eq:xneqy-left-loop-deletion}
		\end{align}
		
		\item Suppose that $i_0=\ell(b)$.  This is the left-right symmetric
		case.  We have $y=\beta_{\ell(b)-1}$ and may write
		$b=p\beta_{\ell(b)-1}\bet$,
		where $p$ is a path based at $s(\beta_{\ell(b)-1})$, possibly trivial.
		Thus
		\[
		\begin{tikzcd}
			\bullet && \bullet
			\arrow["p", dotted, from=1-1, to=1-1, loop, in=145, out=215, distance=10mm]
			\arrow["{y=\beta_{\ell(b)-1}}", from=1-1, to=1-3]
			\arrow["{x=\bet}", from=1-3, to=1-3, loop, in=325, out=35, distance=10mm]
		\end{tikzcd}
		\]
		and
		\begin{align}
			l_2(\para{\bet}{e_{\ell(b)}}\otimes \para{\beta_{\ell(b)-1}}{p\beta_{\ell(b)-1}\bet})
			&=\para{\beta_{\ell(b)-1}}{p\beta_{\ell(b)-1}}.
			\label{eq:xneqy-right-loop-deletion}
		\end{align}
	\end{enumerate}
	
	Thus the trivial-value case with $x\neq y$ contributes precisely the two
	loop deletions
	\eqref{eq:xneqy-left-loop-deletion} and
	\eqref{eq:xneqy-right-loop-deletion}.
	
	%==========================
	\subsubsection{$\ell(a)=1$}\label{subsec:xneqy-ella-one}
	If $a=x$, then the bracket is covered by \Cref {00-x-equal-to-a}.  Assume $a\neq x$.
	If either $a$ or $x$ were a loop, \Cref{x-is-loop} would force $a=x$, a
	contradiction.  Hence $a$ and $x=\bi$ are distinct parallel arrows:
	\[
	\begin{tikzcd}
		i_0 && i_0+1
		\arrow["a"', shift right, from=1-1, to=1-3]
		\arrow["x=\bi", shift left, from=1-1, to=1-3]
	\end{tikzcd}
	\]
	
	\begin{enumerate}
		\item Let $\ell(b)=1$.  Since $\Sub_x^a(b)\neq0$, we must have $b=x$.  The support
		arrow $y$ is parallel to $x$, and the gentle valency condition forces
		$y=a$.  Hence
		\[
		\begin{tikzcd}
			& {} & \\
			{i_0} && {i_0+1}
			\arrow["{x=\bi=b}", shift left, from=2-1, to=2-3]
			\arrow["{y=a}"', shift right, from=2-1, to=2-3]
		\end{tikzcd}
		\]
		and
		\begin{align}
			l_2(\para{x}{y}\otimes\para{y}{x})
			&=
			\para{y}{y}-\para{x}{x}.
			\label{eq:xneqy-parallel-swap}
		\end{align}
		
		\item Let $\ell(b)>1$.  By \Cref{lem:replacement-criterion}, a
		nonzero replacement of $x$ in $b$ would require the value path $a$ to
		contain $x$.  Since $a$ is a single arrow distinct from $x$, this is
		impossible, and therefore
		$\Sub_x^a(b)=0$.
		The $x$-supported term can still be nonzero when $y=a$.  In that case,
		\[
		l_2(\para{x}{a}\otimes\para{a}{b})=-\para{x}{b},
		\]
		which is the skew-symmetric counterpart of the case $\ell(b)=1$
		considered below.  We therefore do not list it as a
		separate representative.
	\end{enumerate}
	
	%==========================
	\subsubsection{$\ell(a)>1$}\label{subsec:xneqy-ella-gt-one}
	We may assume throughout that $a\neq x$ and $b\neq y$.  By
	\Cref{lem:replacement-criterion}, a nonzero contribution to the $y$-supported term $\Sub_x^a(b)$ can occur only when $b=x$, or through left or
	right endpoint extension, or idempotent absorption.
	
	\begin{enumerate}
		\item Let $\ell(b)=1$, so $b=x$.  By \Cref{x-is-loop} and $\ell(a)>1$, the
		arrow $x=b$ is not a loop.  Since $y\neq x$ and $b\parallel y$, the
		arrows $x$ and $y$ are distinct parallel arrows.  The first ordered
		replacement gives $\Sub_x^a(x)=a$.  For the reverse replacement,
		\Cref{lem:replacement-criterion} shows that
		$\Sub_y^x(a)=0$: every nonzero replacement in a path of length greater
		than one would require the value path $x$ to contain the replaced arrow
		$y$, which would force $x=y$.  Hence
		\begin{align}
			l_2(\para{x}{a}\otimes\para{y}{x})
			=
			\para{y}{a}.
			\label{eq:xneqy-y-support-only}
		\end{align}
		We assume $\ell(b)>1$ in the remaining cases.
		
		\item Suppose that
		$\para{x}{a}=\para{\beta_1}{p\beta_1\mu}$,
		where $p$ is a nontrivial path based at $s(\beta_1)$ such that $pb\neq0$,
		and
		\[\mu\in\mR_1 = \{e_{t(\beta_1)}\} \cup \{\beta_2\mid\beta_2\in\Sp\}.\]
		The local configuration is
		
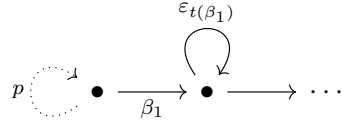
\[
		\begin{tikzcd}
			\bullet & \bullet & \cdots
			\arrow["p", dotted, from=1-1, to=1-1, loop, in=145, out=215, distance=10mm]
			\arrow["{\beta_1}"', from=1-1, to=1-2]
			\arrow["\ep_{t(\beta_1)}", from=1-2, to=1-2, loop, in=55, out=125, distance=10mm]
			\arrow[from=1-2, to=1-3]
		\end{tikzcd}
		\]
		The $y$-supported contribution is
		$\Sub_{\beta_1}^{p\beta_1\mu}(b)=\pi(pb)=pb$,
		where the last equality uses $pb\neq0$ and the fact that any adjacent
		special-loop square has already been absorbed.
		
		It remains to show that the reverse replacement vanishes.
		Let $p=p_1\cdots p_{\ell(p)}$.
		If $y$ does not occur in
		$p\beta_1\mu$, then $\Sub_y^b(p\beta_1\mu)=0$.  Otherwise $y$ must
		occur inside $p$, say $y=p_j$.  It is neither $\beta_1=x$ nor the
		special loop $\mu$.  Since $\ell(b)>1$, \Cref{x-is-loop} shows that
		$y$ is not a loop.  Replacing $y$ by $b$ then produces
		two occurrences of $\beta_1$, so the resulting path cannot lie in $\mB$ by
		\Cref{basis-case}.  Thus
		$\Sub_y^b(p\beta_1\mu)=0$,
		and therefore
		\begin{align}
			l_2(\para{\beta_1}{p\beta_1\mu}\otimes\para{y}{b})
			=
			\para{y}{pb}.
			\label{eq:xneqy-left-boundary}
		\end{align}
		
		\item This is left-right symmetric to the preceding case.  Suppose that
		$\para{x}{a}=\para{\bet}{\lambda\bet q}$,
		where $q$ is a nontrivial path based at $t(\bet)$ such that $bq\neq0$,
		and
		\[
		\lambda\in\mL_{\ell(b)}
		=
		\{e_{s(\bet)}\}
		\cup
		\{\beta_{\ell(b)-1}\mid\beta_{\ell(b)-1}\in\Sp\}.
		\]
		The local configuration is
		
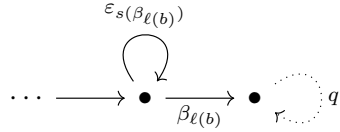
\[
		\begin{tikzcd}
			\cdots & \bullet & \bullet
			\arrow[from=1-1, to=1-2]
			\arrow["{\ep_{s(\bet)}}", from=1-2, to=1-2, loop, in=55, out=125, distance=10mm]
			\arrow["\bet"', from=1-2, to=1-3]
			\arrow["q", dotted, from=1-3, to=1-3, loop, in=325, out=35, distance=10mm]
		\end{tikzcd}
		\]
		The same argument gives
		\begin{align}
			l_2(\para{\bet}{\lambda\bet q}\otimes\para{y}{b})
			=
			\para{y}{bq}.
			\label{eq:xneqy-right-boundary}
		\end{align}
		
		\item We finally consider the cases in which a nonzero replacement is produced
		by an adjacent special loop.
		\begin{enumerate}
			\item Suppose that $\beta_1$ is a special loop and $x=\beta_2$.
			This is the special instance of the left endpoint extension with
			$p=\beta_1$.  Since $\ell(b)>1$, the support arrow $y$ is not a loop
			by \Cref{x-is-loop}.  At the source of $b$, the special loop
			$\beta_1$ and the arrow $\beta_2$ already exhaust the two possible
			outgoing arrows allowed by \textup{(G2)}.  Hence $y=\beta_2=x$,
			contrary to $x\neq y$.  Thus this case does not occur here.
			
			\item Suppose that $\bet$ is a special loop and
			$x=\beta_{\ell(b)-1}$.  The left-right symmetric argument again
			forces $y=x$, so this case is also excluded.
			
			\item It remains to consider an internal occurrence $x=\bi$ with
			adjacent special loops.  Since $\ell(a)>1$, the arrow $x=\bi$ is not
			a loop by \Cref{x-is-loop}.  Up to left-right symmetry, the local
			configuration is
			
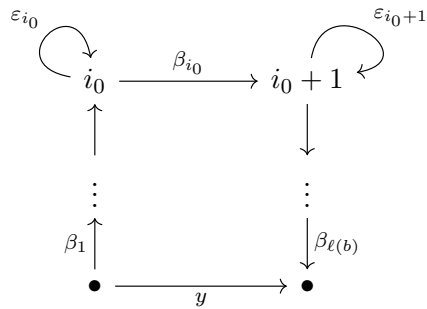
\[
			\begin{tikzcd}
				i_0 && i_0+1 \\
				\vdots && \vdots \\
				\bullet && \bullet
				\arrow["{\ep_{i_0}}", from=1-1, to=1-1, loop, in=100, out=170, distance=10mm]
				\arrow["\bi", from=1-1, to=1-3]
				\arrow["{\ep_{i_0+1}}", from=1-3, to=1-3, loop, in=10, out=80, distance=10mm]
				\arrow[from=1-3, to=2-3]
				\arrow[from=2-1, to=1-1]
				\arrow["\bet", from=2-3, to=3-3]
				\arrow["{\beta_1}", from=3-1, to=2-1]
				\arrow["y"', from=3-1, to=3-3]
			\end{tikzcd}
			\]
			If $\epi$ exists, then $\beta_{i_0-1}=\epi$ with $i_0>2$; if
			$\ep_{i_0+1}$ exists, then
			$\beta_{i_0+1}=\ep_{i_0+1}$ with $i_0<\ell(b)-1$.  Hence an
			absorbed adjacent special loop is never the first or last arrow of
			$b$.  The first formula below occurs only when $\ep_{i_0+1}$ exists,
			the second only when $\epi$ exists, and the third only when both exist.
			
			In each case the substitution on the $y$-support gives $b$, because
			the extra special loop is absorbed by $\varepsilon^2=\varepsilon$.
			The reverse substitution is zero.  Indeed, $y$ is not a loop by
			\Cref{x-is-loop}, and $y\neq x=\bi$, so $y$ does not occur in any of
			$\bi\ep_{i_0+1}$, $\epi\bi$, or $\epi\bi\ep_{i_0+1}$.
			Therefore
			\begin{equation}\label{eq:xneqy-idempotent-absorption}
				\begin{aligned}
					l_2(\para{\bi}{\bi\ep_{i_0+1}}\otimes\para{y}{b})
					&=\para{y}{b},
					&&\text{where }1\leq i_0<\ell(b)-1,\\
					l_2(\para{\bi}{\epi\bi}\otimes\para{y}{b})
					&=\para{y}{b},
					&&\text{where }2<i_0\leq\ell(b),\\
					l_2(\para{\bi}{\epi\bi\ep_{i_0+1}}\otimes\para{y}{b})
					&=\para{y}{b},
					&&\text{where }2<i_0<\ell(b)-1.
				\end{aligned}
			\end{equation}
		\end{enumerate}
	\end{enumerate}
	
	\begin{proof}[Proof of \Cref{prop:degree-zero-classification}]
		By \Cref{DZSF}, every nonzero contribution is obtained by replacing an
		occurrence of one support arrow in the value path of the other cochain.  If
		$x=y$, the preceding analysis shows that all nonzero brackets are precisely
		the cases in \Cref{00-x-equal-to-a}.  If $x\neq y$, then, up to
		\eqref{sym}, it is enough to classify the term $\Sub_x^a(b)$; the
		replacement criterion \Cref{lem:replacement-criterion} and the three cases
		$\ell(a)=0$, $\ell(a)=1$, and $\ell(a)>1$ give exactly
		\eqref{eq:xneqy-left-loop-deletion}--\eqref{eq:xneqy-idempotent-absorption}.
		These two alternatives exhaust $x=y$ and $x\neq y$, so no further bracket
		of type $(0,0)$ can occur.
	\end{proof}
	
	A compact lookup table for brackets of type $(0,0)$ is given in
	\Cref{app:degree-zero-lookup}.
	%=============================================
	\section{Brackets of type \texorpdfstring{$(0,+)$}{(0,+)}}\label{sec:mixed-classification}
	
	This section gives a complete basis-level classification of brackets of type
	$(0,+)$, that is,
	\[
	l_2:B^0\otimes B^n\longrightarrow B^n,\quad n\geq1.
	\]
	We first record the general formula and several consequences that will be used
	throughout the classification.
	
	%===========================
	\subsection{General formula for brackets of type \texorpdfstring{$(0,+)$}{(0,+)}}
	
	Let $f=\para{x}{a}\in B^0$ and $g=\para{y}{b}\in B^n$, where $n\ge1$.
	For
	$\gamma=\gamma_1\cdots\gamma_{n+1}\in\Gamma_{n+1}$,
	the general formula \eqref{eq:general-l2} becomes
	\begin{align*}
		&l_2(f\otimes g)(\gamma_1\cdots\gamma_{n+1}) \\
		={}&\para{x}{a}G_1(\varpi(\para{y}{b}(\gamma_1\cdots\gamma_{n+1})))\\
		&-\sum_{j=1}^{n+1}
		\para{y}{b}
		G_{n+1}(
		\gamma_1 \otimes \cdots \otimes \gamma_{j-1}
		\otimes \varpi(\para{x}{a}(\gamma_j))
		\otimes \gamma_{j+1}\otimes \cdots \otimes \gamma_{n+1}).
	\end{align*}
	A nonzero contribution to either term forces $\gamma=y$.  Hence, whenever the
	formula is evaluated below, we write
	$y=\gamma_1\cdots\gamma_{n+1}$.
	We also set
	\[\Int_x(y)=\#\{j \mid 2\le j\le n,\ \gamma_j=x\}.\]
	
	%=============================
	\begin{proposition}\label{MSBF}
		For $n\ge1$,
		\[
		l_2(\para{x}{a}\otimes\para{y}{b})
		=
		\para{y}{\Sub_x^a(b)-\mC_x^a(y;b)},
		\]
		where
		\[
		\mC_x^a(y;b)
		=
		\one_{x=\gamma_1}\pi(L_x(a)b)
		+\one_{x=\gamma_{n+1}}\pi(bR_x(a))
		+\one_{x=a}\Int_x(y)b.
		\]
		Products which vanish in $A$ are omitted.  In particular, every bracket of
		this type is supported on the single relation concatenation $y$.
	\end{proposition}
	
	\begin{proof}
		The support assertion follows immediately from the displayed transfer
		formula above.  We therefore evaluate at $\gamma=y$.
		
		The substitution term is
		\[
		\para{x}{a}G_1(\varpi(b))=\Sub_x^a(b).
		\]
		Consider first $j=1$.  The comparison map can
		reassemble the tensor to $y$ only when $x=\gamma_1$ and $a$ ends in $x$, say
		$a=L_x(a)x$.  This contributes
		\[
		-\pi(L_x(a)b).
		\]
		Similarly, at $j=n+1$ a nonzero contribution is possible only when
		$x=\gamma_{n+1}$ and $a=xR_x(a)$, giving
		\[
		-\pi(bR_x(a)).
		\]
		At an interior position $2\le j\le n$, reassembly to $y$ is possible only in
		the identity case $a=x=\gamma_j$.  Each such occurrence contributes one copy
		of $-b$.  Adding these terms gives the stated formula.
	\end{proof}
	%============================
	\begin{corollary}\label{mixed-boundary-summand-eq-cases}
		The term $\mC_x^a(y;b)$ can be nonzero only
		through one of the following mechanisms: $x=a=\gamma_j$ for some
		$j$, $\para{x}{a}=\para{\gamma_1}{p\gamma_1}$, or
		$\para{x}{a}=\para{\gamma_{n+1}}{\gamma_{n+1}q}$.
		Consequently, it vanishes whenever $\num{x}{y}=0$, and also whenever no arrow
		of $y$ occurs in $a$.
	\end{corollary}
	%===========================
	\begin{corollary}\label{01-x-equal-to-a}
		Let $f=\para{x}{a}\in B^0$ and $g=\para{y}{b}\in B^n$, where $n\ge1$.
		If $a=x$, then
		\[
		l_2(\para{x}{x}\otimes\para{y}{b})
		=
		(\num{x}{b}-\num{x}{y})\para{y}{b}.
		\]
	\end{corollary}
	
	\begin{proof}
		For $a=x$, the substitution term is
		$\Sub_x^x(b)=\num{x}{b}b$.  The three terms in $\mC_x^x(y;b)$ together
		count all occurrences of $x$ in $y$ and therefore contribute
		$\num{x}{y}b$.  The result follows from \Cref{MSBF}.
	\end{proof}
	
	\begin{remark}\label{rem:mixed-identity-counting}
		The nonvanishing in \Cref{01-x-equal-to-a} is governed by the coefficient
		$\num{x}{b}-\num{x}{y}$.
	\end{remark}
	
	%==============================
	\begin{corollary}
		\label{cor:mixed-endpoint-extension-general}
		Let $f=\para{x}{a}\in B^0$ and $g=\para{y}{b}\in B^n$, where $n\geq1$, and
		write $b=\beta_1\cdots\beta_{\ell(b)}$ with $\ell(b)>1$.
		Assume that $\mC_x^a(y;b)=0$.
		\begin{enumerate}
			\item Suppose that $x=\beta_1$ and
			$a=p\beta_1\mu$,
			where $\mu\in
			\{e_{t(\beta_1)}\}
			\cup
			\{\beta_2 \mid \beta_2\in\Sp\}$
			and $p$ is a nontrivial path in $\mB$ based at $s(\beta_1)$.  
			Then
			\[
			l_2(\para{\beta_1}{p\beta_1\mu}\otimes\para{y}{b})
			=
			\para{y}{\pi(pb)}.
			\]
			
			\item Suppose that $x=\beta_{\ell(b)}$ and
			$a=\lambda\beta_{\ell(b)}q$, where
			$\lambda\in
			\{e_{s(\beta_{\ell(b)})}\}
			\cup
			\{\beta_{\ell(b)-1} \mid \beta_{\ell(b)-1}\in\Sp\}$
			and $q$ is a nontrivial path in $\mB$ based at
			$t(\beta_{\ell(b)})$.  Then
			\[
			l_2(
			\para{\beta_{\ell(b)}}{\lambda\beta_{\ell(b)}q}
			\otimes\para{y}{b}
			)
			=
			\para{y}{\pi(bq)}.
			\]
		\end{enumerate}
	\end{corollary}
	
	\begin{proof}
		In the first case, $x=\beta_1$ is the unique occurrence of $x$ in $b$ by
		\Cref{basis-case}.  The left-endpoint part of
		\Cref{lem:replacement-criterion} gives $\Sub_x^a(b)=\pi(pb)$.
		Since $\mC_x^a(y;b)=0$, the formula follows from \Cref{MSBF}.  The second
		case is the left-right symmetric application of the right-endpoint part of
		\Cref{lem:replacement-criterion}.
	\end{proof}
	
	The two endpoint extensions in \Cref{cor:mixed-endpoint-extension-general}
	apply to all four possible forms of $y$.  Each displayed term is retained
	precisely when its projected value under $\pi$ is nonzero.
	
	%==============================
	\begin{corollary}\label{01-sub-only}
		Let $f=\para{x}{a}\in B^0$ and $g=\para{y}{x}\in B^n$, where
		$n\ge1$, and assume that $a$ is nontrivial and $a\neq x$.  If $\mC_x^a(y;x)=0$, then
		\[
		l_2(\para{x}{a}\otimes\para{y}{x})
		=
		\para{y}{a}.
		\]
	\end{corollary}
	
	\begin{proof}
		Since the value path of the second cochain is $b=x$, one has
		$\Sub_x^a(x)=a$.
		The formula follows from $\mC_x^a(y;x)=0$ and \Cref{MSBF}.
	\end{proof}
	
	\begin{remark}\label{rem:substitution-only-conditions}
		For the brackets in \Cref{01-sub-only}, the condition
		$\mC_x^a(y;x)=0$ is equivalent to
		\[
		\one_{x=\gamma_1}\pi(L_x(a)x)
		+
		\one_{x=\gamma_{n+1}}\pi(xR_x(a))
		=0.
		\]
		In particular, it holds in each of the following situations:
		\begin{enumerate}
			\item $\num{x}{y}=0$;
			\item $x=\gj$ for some $1<j_0<n+1$ and $\num{x}{a}=0$; here $x$ may
			also occur at one or both endpoints of $y$, but the corresponding terms in
			$\mC_x^a(y;x)$ vanish because $a$ contains no occurrence of $x$;
			\item $x=\gamma_1\neq\gamma_{n+1}$ and $\pi(L_x(a)x)=0$;
			\item $x=\gamma_{n+1}\neq\gamma_1$ and $\pi(xR_x(a))=0$;
			\item $x=\gamma_1=\gamma_{n+1}$ and
			$\pi(L_x(a)x)+\pi(xR_x(a))=0$.
		\end{enumerate}
	\end{remark}
	%===========================
	\begin{corollary}
		\label{cor:mixed-loop-deletion}
		Let $f=\para{x}{e_{s(x)}}\in B^0$ and
		$g=\para{y}{b}\in B^n$, where $n\geq1$. Then $x$ is a loop and
		\[
		l_2(f\otimes g)
		=
		\para{y}{\Sub_x^{e_{s(x)}}(b)}.
		\]
		More explicitly, if
		$b=\beta_1\cdots\beta_{\ell(b)}$
		is nontrivial, then
		\[
		l_2(\para{x}{e_{s(x)}}\otimes\para{y}{b})
		=
		\begin{cases}
			\para{y}{e_{s(x)}},
			& \text{if } b=x,\\
			\para{y}{\beta_2\cdots\beta_{\ell(b)}},
			& \text{if } x=\beta_1\text{ and }\ell(b)>1,\\
			\para{y}{\beta_1\cdots\beta_{\ell(b)-1}},
			& \text{if } x=\beta_{\ell(b)}\text{ and }\ell(b)>1,\\
			0,
			& \text{otherwise}.
		\end{cases}
		\]
		In particular, such a bracket is nonzero precisely when $b=x$, or when
		$x$ is the first or last arrow of $b$ and $\ell(b)>1$.
	\end{corollary}
	
	\begin{proof}
		Since $e_{s(x)}\parallel x$, the arrow $x$ is a loop.  Moreover,
		$L_x(e_{s(x)})=R_x(e_{s(x)})=0$
		and $e_{s(x)}\neq x$.  Hence
		$\mC_x^{e_{s(x)}}(y;b)=0$
		by \Cref{MSBF}, and therefore
		\[
		l_2(\para{x}{e_{s(x)}}\otimes\para{y}{b})
		=
		\para{y}{\Sub_x^{e_{s(x)}}(b)}.
		\]
		If $b$ is trivial, then the substitution is zero.  If $b$ is nontrivial,
		the explicit alternatives follow directly from
		\Cref{lem:replacement-criterion}, applied with $a=e_{s(x)}$.  For
		$\ell(b)=1$, the only nonzero case is $b=x$, which gives
		$\Sub_x^{e_{s(x)}}(x)=e_{s(x)}$.  For $\ell(b)>1$, the same lemma shows
		that a nonzero substitution is possible only when $x$ is the first or the
		last arrow of $b$, in which case the corresponding endpoint loop is deleted;
		all internal occurrences give zero.  This proves the stated formula.
	\end{proof}
	%===========================
	\begin{proposition}\label{prop:mixed-classification}
		Let $f=\para{x}{a}\in B^0$ and $g=\para{y}{b}\in B^n$, where $n\geq1$.
		Up to the skew-symmetry relation \eqref{sym}, the nonzero brackets of type
		$(0,+)$ are exactly the following:
		\begin{enumerate}
			\item the cases in \Cref{01-x-equal-to-a} for which
			$\num{x}{b}-\num{x}{y}\neq0$ in $\mK$;
			\item the two endpoint extensions in
			\Cref{cor:mixed-endpoint-extension-general};
			\item the case described in \Cref{01-sub-only};
			\item the cases in \Cref{cor:mixed-loop-deletion};
			\item the remaining cases listed explicitly in
			\eqref{eq:mixed-local-first}--\eqref{eq:mixed-local-last}, subject to
			the hypotheses stated there.
		\end{enumerate}
		The case $y=\omega^{n+1}$ in \Cref{subsec:mixed-loop-powers} produces no
		additional representative beyond the cases already listed.  No other bracket
		of type $(0,+)$ occurs.
	\end{proposition}
	
	To verify \Cref{prop:mixed-classification}, we follow the four possible forms
	of $y$ in \Cref{gamma-cases}.  The cases already covered by
	\Cref{01-x-equal-to-a,01-sub-only,cor:mixed-loop-deletion} will not be
	repeated below, and the endpoint extensions in
	\Cref{cor:mixed-endpoint-extension-general} will be cited rather than listed
	again.
	
	Unless one of these previously isolated cases is explicitly invoked, we assume
	$a\neq x$ and $\ell(a)>0$.  By \Cref{x-is-loop}, the arrow $x$ is then not a
	loop.  If $b=x$ and $\mC_x^a(y;x)=0$, the bracket is already covered by
	\Cref{01-sub-only}; hence, whenever $\mC_x^a(y;b)=0$, the remaining local
	analysis may assume $\ell(b)>1$.
	%==============================
	\subsection{The case where \texorpdfstring{$y$}{y} has no repeated arrows}\label{subsec:mixed-no-repetition}
	
	Assume that
	$y=\gamma_1\cdots\gamma_{n+1}\in\Gamma_{n+1}$
	has no repeated arrows.  Since each arrow occurs at most once in $y$, one has
	\[
	\mC_x^a(y;b)=
	\begin{cases}
		b, & \text{if } x=a=\gj\text{ for some }1\leq j_0\leq n+1
		\text{ (identity case)},\\
		pb, & \text{if }\para{x}{a}=\para{\gamma_1}{p\gamma_1},\\
		bq, & \text{if }\para{x}{a}=\para{\gamma_{n+1}}{\gamma_{n+1}q},\\
		0, & \text{otherwise.}
	\end{cases}
	\]
	Here $p$ and $q$ are nontrivial left and right path factors, respectively;
	in the displayed endpoint configurations they are closed at the corresponding
	endpoint whenever the products $pb$ and $bq$ are nonzero.  Although $y$ has no
	repeated arrows, the value path $b\in\mB$ may still contain an arrow of $y$.
	We first evaluate the two endpoint terms in $\mC_x^a(y;b)$ and then turn to
	the substitution terms.
	
	%===========================================
	Assume first that
	$\para{x}{a}=\para{\gamma_1}{p\gamma_1}$,
	where $p=p_1\cdots p_{\ell(p)}$ is a nontrivial left path factor and
	$p\gamma_1\in\mB$.  The arrow $\gamma_1$ cannot be a loop.  Indeed, if it
	were, finite-dimensionality would imply $\gamma_1^2\in I^S$, whereas
	$\gamma_1\gamma_2\in I^S$ because $y\in\Gamma_{n+1}$ and $n\geq1$.
	Condition \textup{(G3)}, applied to
	$\gamma_1^2,\gamma_1\gamma_2\in I^S$, would then force
	$\gamma_2=\gamma_1$, contradicting the assumption that $y$ has no repeated arrows.  Moreover,
	$p_1\neq\gamma_1$ by \Cref{basis-case}, since $p\gamma_1\in\mB$.
	
	A cancellation with the substitution term is possible only when $\gamma_1$
	occurs in $b$.  If such an occurrence exists, it is unique by
	\Cref{basis-case}.  The local configuration is
	\[
	\begin{tikzcd}
		1 && \bullet
		\arrow["p", dotted, from=1-1, to=1-1, loop, in=145, out=215, distance=10mm]
		\arrow["{\gamma_1}", from=1-1, to=1-3]
	\end{tikzcd}
	\]
	If $\ell(b)>0$, then the first arrow of $b$ is either $p_1$ or $\gamma_1$.
	We distinguish the three possible lengths of $b$.
	
	\begin{enumerate}
		\item Suppose $\ell(b)=0$.  Then $b=e_{s(y)}$ and $y$ is a closed relation
		concatenation.  The substitution term vanishes and
		$\mC_x^a(y;b)=p$.  Thus
		
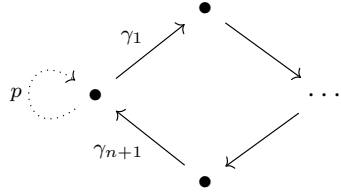
\[
		\begin{tikzcd}
			& \bullet & \\
			\bullet && \cdots \\
			& \bullet
			\arrow[from=1-2, to=2-3]
			\arrow["{\gamma_1}", from=2-1, to=1-2]
			\arrow["{p}",dotted, from=2-1, to=2-1, loop, in=145, out=215, distance=10mm]
			\arrow[from=2-3, to=3-2]
			\arrow["{\gamma_{n+1}}", from=3-2, to=2-1]
		\end{tikzcd}
		\]
		and
		\begin{align}
			l_2(\para{\gamma_1}{p\gamma_1}\otimes\para{y}{e_{s(y)}})
			=-\para{y}{p}.
			\label{eq:mixed-local-first}
		\end{align}
		
		\item Suppose $\ell(b)=1$.  Then $b=\gamma_1$ or $b=p_1$.
		\begin{enumerate}
			\item If $b=\gamma_1$, the substitution term is $pb$.  The local
			configuration is
			\[
			\begin{tikzcd}
				\bullet && \bullet
				\arrow["p", dotted, from=1-1, to=1-1, loop, in=145, out=215, distance=10mm]
				\arrow["{\gamma_1=b}", from=1-1, to=1-3]
			\end{tikzcd}
			\]
			In this case,
			\[
			\Sub_x^a(b)-\mC_x^a(y;b)=pb-pb=0.
			\]
			
			\item If $b=p_1$, then $pb\neq0$ forces $p=p_1=b$ to be a special
			loop.  Since $x=\gamma_1\neq b$, the substitution term vanishes.  The
			local configuration is
			
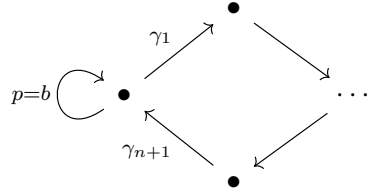
\[
			\begin{tikzcd}
				& \bullet & \\
				\bullet && \cdots \\
				& \bullet
				\arrow[from=1-2, to=2-3]
				\arrow["{\gamma_1}", from=2-1, to=1-2]
				\arrow["{p=b}", from=2-1, to=2-1, loop, in=145, out=215, distance=10mm]
				\arrow[from=2-3, to=3-2]
				\arrow["{\gamma_{n+1}}", from=3-2, to=2-1]
			\end{tikzcd}
			\]
			and hence
			\begin{align}
				l_2(\para{\gamma_1}{p\gamma_1}\otimes\para{y}{p})
				=-\para{y}{p}.
			\end{align}
		\end{enumerate}
		
		\item Suppose $\ell(b)>1$.
		\begin{enumerate}
			\item If $\beta_1=\gamma_1$, then the occurrence of $x$ is the left
			endpoint of $b$.  The local configuration is
			\[
			\begin{tikzcd}
				\bullet && \bullet
				\arrow["p", dotted, from=1-1, to=1-1, loop, in=145, out=215, distance=10mm]
				\arrow["{\beta_1=\gamma_1}", from=1-1, to=1-3]
			\end{tikzcd}
			\]
			and again the two terms cancel:
			\[
			\Sub_x^a(b)-\mC_x^a(y;b)=pb-pb=0.
			\]
			
			\item If $\beta_1=p_1$, then the simultaneous nonvanishing of
			$p\gamma_1$ and $pb$ forces $p=\beta_1$ to be a special loop and
			$\beta_2=\gamma_1$.  Thus
			\[
			\begin{tikzcd}
				\bullet && \bullet
				\arrow["{p=\beta_1}", from=1-1, to=1-1, loop, in=145, out=215, distance=10mm]
				\arrow["{\beta_2=\gamma_1}", from=1-1, to=1-3]
			\end{tikzcd}
			\]
			and $pb=b$.  The substitution term is also $b$, so
			\[
			\Sub_x^a(b)-\mC_x^a(y;b)=b-b=0.
			\]
		\end{enumerate}
	\end{enumerate}
	
	%========================
	Next assume that
	$\para{x}{a}=\para{\gamma_{n+1}}{\gamma_{n+1}q}$,
	where $q$ is a nontrivial right path factor and $\gamma_{n+1}q\in\mB$.
	The argument is obtained from the preceding case by left-right symmetry.  After the identity case is removed,
	all cases with $\ell(b)>0$ cancel except for the special-loop value below.  The
	two nonzero representatives are therefore the following.
	
	\begin{enumerate}
		\item If $b=e_{t(y)}$, then $s(y)=t(y)$ and
		
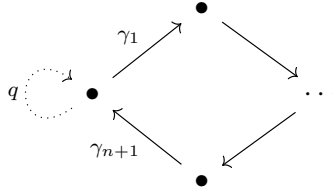
\[
		\begin{tikzcd}
			& \bullet & \\
			\bullet && \cdots \\
			& \bullet
			\arrow[from=1-2, to=2-3]
			\arrow["{\gamma_1}", from=2-1, to=1-2]
			\arrow["{q}",dotted, from=2-1, to=2-1, loop, in=145, out=215, distance=10mm]
			\arrow[from=2-3, to=3-2]
			\arrow["{\gamma_{n+1}}", from=3-2, to=2-1]
		\end{tikzcd}
		\]
		with
		\begin{align}
			l_2(\para{\gamma_{n+1}}{\gamma_{n+1}q}
			\otimes\para{y}{e_{t(y)}})
			=-\para{y}{q}.
		\end{align}
		
		\item If $b=q$ is a special loop, then
		
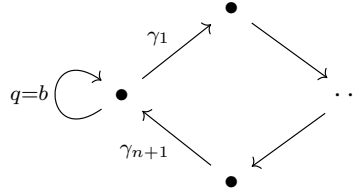
\[
		\begin{tikzcd}
			& \bullet & \\
			\bullet && \cdots \\
			& \bullet
			\arrow[from=1-2, to=2-3]
			\arrow["{\gamma_1}", from=2-1, to=1-2]
			\arrow["{q=b}", from=2-1, to=2-1, loop, in=145, out=215, distance=10mm]
			\arrow[from=2-3, to=3-2]
			\arrow["{\gamma_{n+1}}", from=3-2, to=2-1]
		\end{tikzcd}
		\]
		and
		\begin{align}
			l_2(\para{\gamma_{n+1}}{\gamma_{n+1}q}\otimes\para{y}{q})
			=-\para{y}{q}.
		\end{align}
	\end{enumerate}
	%=========================
	It remains to consider the substitution terms.  We now impose
	$\mC_x^a(y;b)=0$
	and $\Sub_x^a(b)\neq0$.
	Then $\ell(b)\ge1$, and there is a unique index $i_0$ such that
	$x=\bi$.  If $\ell(b)=1$, then $b=x$ and the bracket is covered by
	\Cref{01-sub-only}:
	\[
	l_2(\para{x}{a}\otimes\para{y}{x})=\para{y}{a}.
	\]
	Hence we assume $\ell(b)>1$.  By \Cref{lem:replacement-criterion} and the
	standing assumptions, the remaining nonzero substitutions are the endpoint alternatives in
	\Cref{lem:replacement-criterion} or arise from idempotent absorption.
	
	\begin{enumerate}
		\item Consider a left endpoint extension
		$\para{x}{a}=\para{\beta_1}{p\beta_1\mu}$, where
		$\mu\in\mR_1=\{e_{t(\beta_1)}\}\cup
		\{\beta_2 \mid \beta_2\in\Sp\}$ and $p=p_1\cdots p_{\ell(p)}$ is nontrivial.
		\begin{enumerate}
			\item First let $\mu=e_{t(\beta_1)}$, so $a=p\beta_1$.  The assumption $\mC_x^a(y;b)=0$ implies
			$\gamma_1\neq\beta_1$; otherwise $\pi(L_x(a)b)=pb\neq0$.
			Since $p\beta_1\in\mB$, \Cref{basis-case} gives
			$p_1\neq\beta_1$. Together with $\gamma_1\neq\beta_1$ and the fact
			that $y\parallel b$, condition \textup{(G2)} forces $\gamma_1=p_1$.
			The path $p$ cannot be a loop: otherwise $\gamma_1=p_1$ would also be
			a loop, and the relation $\gamma_1\gamma_2\in I^S$ would, by
			\textup{(G3)}, force $\gamma_2=\gamma_1$, contradicting the
			assumption of no repeated arrows. Thus $p$ is a left boundary cycle.  The local
			configuration is
			
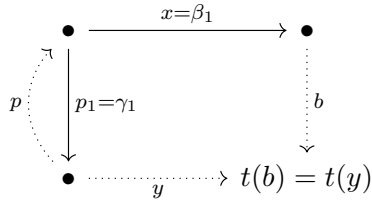
\[
			\begin{tikzcd}
				\bullet && \bullet \\
				\\
				\bullet && {t(b)=t(y)}
				\arrow["{x=\beta_1}", from=1-1, to=1-3]
				\arrow["{p_1=\gamma_1}", from=1-1, to=3-1]
				\arrow["b", dotted, from=1-3, to=3-3]
				\arrow["p", curve={height=-18pt}, dotted, from=3-1, to=1-1]
				\arrow["y"', dotted, from=3-1, to=3-3]
			\end{tikzcd}
			\]
			and
			\begin{align}
				l_2(\para{\beta_1}{p\beta_1}\otimes\para{y}{b})
				=\para{y}{pb},
			\end{align}
			where $\beta_1\neq\gamma_1$.
			
			\item Now suppose $\mu=\beta_2\in\Sp$, so
			$a=p\beta_1\beta_2$.  The local configuration is
			
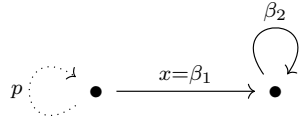
\[
			\begin{tikzcd}
				\bullet && \bullet \\
				\\
				{}
				\arrow["p", dotted, from=1-1, to=1-1, loop, in=145, out=215, distance=10mm]
				\arrow["{x=\beta_1}", from=1-1, to=1-3]
				\arrow["{\beta_2}", from=1-3, to=1-3, loop, in=55, out=125, distance=10mm]
			\end{tikzcd}
			\]
			The substitution introduces one additional copy of the special loop
			$\beta_2$, which is absorbed by $\beta_2^2=\beta_2$.  Since $\mC_x^a(y;b)=0$,
			\begin{align}
				l_2(\para{\beta_1}{p\beta_1\beta_2}\otimes\para{y}{b})
				=\para{y}{pb}.
			\end{align}
		\end{enumerate}
		
		\item The right endpoint extension is obtained by left-right symmetry.  Let
		$\para{x}{a}=\para{\bet}{\lambda\bet q}$, where
		$\lambda\in\mL_{\ell(b)}=\{e_{s(\bet)}\}\cup
		\{\beta_{\ell(b)-1} \mid \beta_{\ell(b)-1}\in\Sp\}$ and $q$ is nontrivial.  The two representatives are
		\begin{align}
			l_2(\para{\bet}{\bet q}\otimes\para{y}{b})
			&=\para{y}{bq},
			 \text{ where } \bet\neq\gamma_{n+1},\\
			l_2(\para{\bet}{\beta_{\ell(b)-1}\bet q}\otimes\para{y}{b})
			&=\para{y}{bq},
			 \text{ where } \beta_{\ell(b)-1}\in\Sp.
		\end{align}
		In the first row, $q$ is a right boundary cycle; in the second, it may be a
		right boundary loop or cycle.
		
		\item It remains to consider idempotent absorption.  Let
		$1\le i_0\le\ell(b)$ and
		$\para{x}{a}=\para{\bi}{\lambda\bi\mu}$, where
		$\lambda\in\mL_{i_0}$, $\mu\in\mR_{i_0}$, and $\lambda$ and $\mu$ are not both trivial.  The local configuration is
		
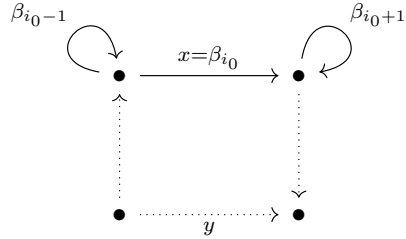
\[
		\begin{tikzcd}
			\bullet && \bullet \\
			\\
			\bullet && \bullet
			\arrow["{\beta_{i_0-1}}", from=1-1, to=1-1, loop, in=100, out=170, distance=10mm]
			\arrow["{x=\bi}", from=1-1, to=1-3]
			\arrow["{\beta_{i_0+1}}", from=1-3, to=1-3, loop, in=10, out=80, distance=10mm]
			\arrow[dotted, from=1-3, to=3-3]
			\arrow[dotted, from=3-1, to=1-1]
			\arrow["y"', dotted, from=3-1, to=3-3]
		\end{tikzcd}
		\]
		Two endpoint configurations are excluded by the assumption $\mC_x^a(y;b)=0$.  If
		$\beta_1\in\Sp$ and
		$\para{x}{a}=\para{\beta_2}{\beta_1\beta_2}$, then
		$\gamma_1=\beta_2$ and
		$\pi(L_x(a)b)=\beta_1b=b\neq0$.  The right endpoint configuration with
		$\bet\in\Sp$ and
		$\para{x}{a}=\para{\beta_{\ell(b)-1}}{\beta_{\ell(b)-1}\bet}$ is
		symmetric and satisfies $\pi(bR_x(a))=b\bet=b\neq0$.
		
		The remaining terms arising from idempotent absorption are
		\begin{equation}\label{eq:mix-idempotent-absorption}
			\begin{aligned}
				l_2(\para{\bi}{\bi\beta_{i_0+1}}\otimes\para{y}{b})
				&=\para{y}{b},
				&&\text{where } 1\le i_0<\ell(b)-1,\\
				l_2(\para{\bi}{\beta_{i_0-1}\bi}\otimes\para{y}{b})
				&=\para{y}{b},
				&&\text{where } 2<i_0\le\ell(b),\\
				l_2(\para{\bi}{\beta_{i_0-1}\bi\beta_{i_0+1}}\otimes\para{y}{b})
				&=\para{y}{b},
				&&\text{where } 2\le i_0\le\ell(b)-1.
			\end{aligned}
		\end{equation}
		In the first row $\beta_{i_0+1}\in\Sp$; in the second row
		$\beta_{i_0-1}\in\Sp$; and in the third row both adjacent arrows are
		special loops.
	\end{enumerate}
	
	\begin{remark}\label{rem:mixed-idempotent-absorption-general}
		The formulas in \eqref{eq:mix-idempotent-absorption} depend only on the local
		shape of $b\in\mB$ and remain valid whenever $\mC_x^a(y;b)=0$,
		independently of which of the four forms of $y$ occurs.
	\end{remark}
	%==========================
	\subsection{The case \texorpdfstring{$y=\omega^{n+1}$}{y = omega power n+1}}\label{subsec:mixed-loop-powers}
	
	Assume that
	$y=\omega^{n+1}$,
	where $\omega$ is a loop.  Since $y\parallel b$ and $A$ is finite-dimensional,
	the value path $b\in\mB$ is either $e_{s(\omega)}$ or $\omega$.  Any nonzero
	bracket of type $(0,+)$ with $y=\omega^{n+1}$ must have $x=\omega$.  If $a=x$, the bracket is covered by \Cref{01-x-equal-to-a}; if
	$a=e_{s(x)}$, it is covered by \Cref{cor:mixed-loop-deletion}.  By
	\Cref{x-is-loop}, there is no third possibility.  Hence the case $y=\omega^{n+1}$ contributes no new representative, whether $\omega$ is a special loop
	(with $\omega^2=\omega$ in $A$) or an ordinary relation loop (with
	$\omega^2=0$ in $A$).
	
	%===========================
	\subsection{The case \texorpdfstring{$y=\tau^w$}{y = tau power w}}\label{subsec:mixed-primitive-powers}
	
	Assume that $y=\tau^w$, where $w>1$ and
	$\tau=\tau_1\cdots\taue$ is a primitive relation cycle.  Then
	$\ell(\tau)w=n+1$.  The underlying local cycle is
	
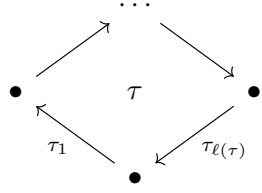
\[
	\begin{tikzcd}
		& \cdots & \\
		\bullet & \tau & \bullet \\
		& \bullet
		\arrow[from=1-2, to=2-3]
		\arrow[from=2-1, to=1-2]
		\arrow["\taue", from=2-3, to=3-2]
		\arrow["{\tau_1}", from=3-2, to=2-1]
	\end{tikzcd}
	\]
	The analysis is divided according to the length of the value path $b$.
	
	\subsubsection{$\ell(b)=0$}
	
	Let $b=e_{s(\tau)}$.  The substitution term is zero.  After the case
	$a=x$ is removed, $\mC_x^a(y;b)$ can be nonzero only at the first or last
	arrow of the displayed copy of $\tau$.  Hence
	\begin{align}
		l_2(\para{\tau_1}{p\tau_1}\otimes\para{y}{e_{s(\tau)}})
		&=-\para{y}{p},\\
		l_2(\para{\taue}{\taue q}\otimes\para{y}{e_{s(\tau)}})
		&=-\para{y}{q}.
	\end{align}
	Here $p$ and $q$ are the corresponding boundary loops or cycles, when they
	exist.
	
	%============
	\subsubsection{$\ell(b)=1$}
	
	Then $b$ is a loop based at $s(\tau)$.  The local configuration is
	
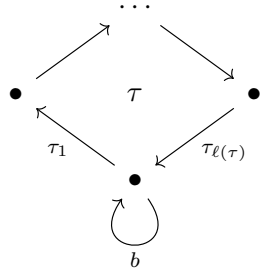
\[
	\begin{tikzcd}
		& \cdots & \\
		\bullet & \tau & \bullet \\
		& \bullet
		\arrow[from=1-2, to=2-3]
		\arrow[from=2-1, to=1-2]
		\arrow["\taue", from=2-3, to=3-2]
		\arrow["{\tau_1}", from=3-2, to=2-1]
		\arrow["b", from=3-2, to=3-2, loop, in=235, out=305, distance=10mm]
	\end{tikzcd}
	\]
	After the case $a=x$ is omitted, the nonzero possibilities are
	\begin{align}
		l_2(\para{\tau_1}{b\tau_1}\otimes\para{y}{b})
		&=
		\begin{cases}
			-\para{y}{b}, & \text{if } b^2=b,\\
			0, & \text{if } b^2=0,
		\end{cases}\\
		l_2(\para{\taue}{\taue b}\otimes\para{y}{b})
		&=
		\begin{cases}
			-\para{y}{b},&\text{if }b^2=b,\\
			0,&\text{if }b^2=0.
		\end{cases}
	\end{align}
	%============
	\subsubsection{$\ell(b)>1$}
	
	Since $y\parallel b$, \Cref{loop-cycle-unique} implies that $b$ is the unique
	cycle in $\mB$ based at $s(\tau)$, and finite-dimensionality gives $b^2=0$.
	Moreover, neither $\beta_1$ nor $\bet$ is a loop.  We compare the endpoints of
	$b$ with those of the primitive relation cycle.
	
	\begin{enumerate}
		\item Suppose $\beta_1\neq\tau_1$ and $\bet\neq\taue$.  The local
		configuration is
		
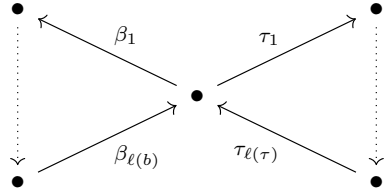
\[
		\begin{tikzcd}
			\bullet &&&& \bullet \\
			&  & \bullet & \\
			\bullet &&&& \bullet
			\arrow[dotted, from=1-1, to=3-1]
			\arrow[dotted, from=1-5, to=3-5]
			\arrow["{\beta_1}"', from=2-3, to=1-1]
			\arrow["{\tau_1}", from=2-3, to=1-5]
			\arrow["\bet"', from=3-1, to=2-3]
			\arrow["\taue", from=3-5, to=2-3]
		\end{tikzcd}
		\]
		Apart from $a=x$, the only possible nonzero terms in $\mC_x^a(y;b)$ come
		from $\para{\tau_1}{p\tau_1}$ and $\para{\taue}{\taue q}$.  By uniqueness of
		the cycle in $\mB$, the corresponding boundary cycles are $p=b=q$, and their
		contributions vanish because $b^2=0$.  Thus only substitution terms
		remain.
		
		Let $x=\bi$, where $1\leq i_0\leq\ell(b)$. Since no further endpoint
		extension of $b$ is possible and neither $\beta_1$ nor $\bet$ is a
		loop, the remaining brackets are the idempotent-absorption cases
		\begin{equation}\label{mix-cycle-idempotent-absorption}
			\begin{aligned}
				l_2(\para{\bi}{\bi\beta_{i_0+1}}\otimes\para{y}{b})
				&=\para{y}{b},
				&&\text{where }1\le i_0<\ell(b)-1,\\
				l_2(\para{\bi}{\beta_{i_0-1}\bi}\otimes\para{y}{b})
				&=\para{y}{b},
				&&\text{where }2<i_0\le\ell(b),\\
				l_2(\para{\bi}{\beta_{i_0-1}\bi\beta_{i_0+1}}\otimes\para{y}{b})
				&=\para{y}{b},
				&&\text{where }3\le i_0\le\ell(b)-2.
			\end{aligned}
		\end{equation}
		Here $\beta_{i_0-1}$ and/or $\beta_{i_0+1}$ are special loops whenever
		they occur in the corresponding row.
		
		\item Suppose $\beta_1=\tau_1$ and $\bet\neq\taue$.  Since
		$\taue\tau_1\in I^S$ and $\bet\neq\taue$, condition
		\textup{(G3)} gives
		$\bet\beta_1=\bet\tau_1\notin I^S$.
		The path $b$ is a cycle in $\mB$, so the only new junction in a product of
		two consecutive copies of $b$ is $\bet\beta_1$.  Hence every power
		$b^N\in\mB$ and is therefore nonzero in $A$, contradicting
		finite-dimensionality.  Thus this endpoint configuration cannot occur.
		
		\item The case $\beta_1\neq\tau_1$ and $\bet=\taue$ is excluded by the
		left-right symmetric argument.
		
		\item Suppose $\beta_1=\tau_1$ and $\bet=\taue$.  The local configuration
		is
		
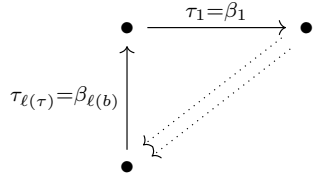
\[
		\begin{tikzcd}
			\bullet && \bullet \\
			\\
			\bullet
			\arrow["{\tau_1=\beta_1}", from=1-1, to=1-3]
			\arrow[shift right, dotted, from=1-3, to=3-1]
			\arrow[shift left, dotted, from=1-3, to=3-1]
			\arrow["{\taue=\bet}", from=3-1, to=1-1]
		\end{tikzcd}
		\]
		Again $b$ is the unique cycle in $\mB$ and $b^2=0$.  No left or right
		boundary loop or cycle can occur; hence the only nonzero brackets are those in
		\eqref{mix-cycle-idempotent-absorption}.
	\end{enumerate}
	%===========================
	\subsection{The case \texorpdfstring{$y=\tau^w\tau_1\cdots\tau_r$}{y = tau power w with suffix}}\label{subsec:mixed-suffix-powers}
	
	Assume that
	$y=\tau^w\tau_1\cdots\tau_r$, where $1\le r<\ell(\tau)$ and
	$\tau=\tau_1\cdots\taue$ is a primitive relation cycle.  By
	\Cref{cycle-not-has-loop}, no arrow of $\tau$ is a loop.  The underlying local
	configuration is
	
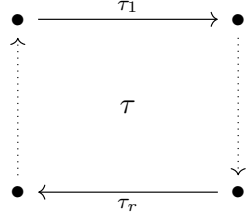
\[
	\begin{tikzcd}
		\bullet && \bullet \\
		& \tau \\
		\bullet && \bullet
		\arrow["{\tau_1}", from=1-1, to=1-3]
		\arrow[dotted, from=1-3, to=3-3]
		\arrow[dotted, from=3-1, to=1-1]
		\arrow["{\tau_r}", from=3-3, to=3-1]
	\end{tikzcd}
	\]
	The endpoint structure differs substantially according to whether $r=1$ or
	$r>1$.
	
	\subsubsection{$r=1$}
	
	Here $b\parallel y$ is equivalent to $b\parallel\tau_1$.
	The value paths in $\mB$ are described relative to a chosen
	connecting path. 
	Starting from a nontrivial path in $\mB$ parallel to $\tau_1$, remove, whenever possible,
	a nontrivial cycle in $\mB$ from its left end and from its right end, as long as a
	nontrivial connecting path remains.  Denote the resulting connecting path by
	$\rho$.  We allow $\rho=\tau_1$; any other choice is called an alternative
	connector.
	
	By \textup{(G2)}, there are at most two possible first arrows at
	$s(\tau_1)$. Once the first arrow has been chosen, \textup{(G4)} uniquely
	determines each subsequent arrow as long as the path remains in $\mB$.
	Hence there are at most two possible connectors.  For a fixed connector $\rho$, let $p$ and $q$ denote the
	possible nontrivial cycles based at $s(\tau_1)$ and $t(\tau_1)$,
	respectively, for which
	\[
	p\rho\in\mB,
	\qquad
	\rho q\in\mB.
	\]
	By \Cref{loop-cycle-unique}, each of $p$ and $q$ is unique whenever it exists,
	but its availability depends on the chosen connector.  Schematically, the
	local configuration is
	
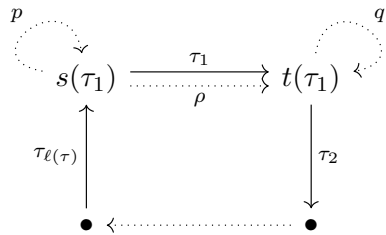
\[
	\begin{tikzcd}
		{s(\tau_1)} && {t(\tau_1)} \\
		\\
		\bullet && \bullet
		\arrow["p", dotted, from=1-1, to=1-1, loop, in=100, out=170, distance=10mm]
		\arrow["{\tau_1}", shift left, from=1-1, to=1-3]
		\arrow["\rho"', shift right, dotted, from=1-1, to=1-3]
		\arrow["q", dotted, from=1-3, to=1-3, loop, in=10, out=80, distance=10mm]
		\arrow["{\tau_2}", from=1-3, to=3-3]
		\arrow["\taue", from=3-1, to=1-1]
		\arrow[dotted, from=3-3, to=3-1]
	\end{tikzcd}
	\]
	The diagram is schematic: the boundary pieces $p$ and $q$ are attached to the
	chosen connector $\rho$, and are not asserted to be simultaneously compatible
	with every connector shown.  For this fixed choice, put
	\[
	\Pi_2(\rho)=\{\rho,\ p\rho,\ \rho q,\ p\rho q\},
	\]
	with nonexistent terms or terms not lying in $\mB$ omitted.  Every nontrivial path in $\mB$
	parallel to $\tau_1$ belongs to $\Pi_2(\rho)$ for at least one connector
	$\rho$; if a path admits more than one such description, we fix one compatible
	description when it is used below.
	
	In particular, there is no connector-independent exclusivity between an
	alternative path and boundary cycles.  A boundary cycle may fail to concatenate
	with $\tau_1$ but concatenate to a path in $\mB$ with an alternative connector; for
	example, $p\tau_1=0$ does not imply $p\rho=0$.  This is why the endpoint-extension
	rules of \Cref{cor:mixed-endpoint-extension-general} are kept separate from the
	discussion for the present form of $y$ below.
	
	\begin{enumerate}
		\item Since
		$\gamma_1=\gamma_{n+1}=\tau_1$, after the case $a=x$ is removed one has
		$\mC_x^a(y;b)\neq0$ only for cochains of the forms
		\[
		\para{\tau_1}{p\tau_1}
		\qquad\text{and}\qquad
		\para{\tau_1}{\tau_1q},
		\]
		where here $p$ and $q$ are taken relative to the connector $\tau_1$, and
		the displayed value path lies in $\mB$.  For an arbitrary value path in $\mB$
		$b\parallel\tau_1$, \Cref{MSBF} gives
		\begin{align}
			l_2(\para{\tau_1}{p\tau_1}\otimes\para{y}{b})
			&=
			\para{y}{\Sub_{\tau_1}^{p\tau_1}(b)-\pi(pb)},
			\label{eq:r1-left-correction-general}\\
			l_2(\para{\tau_1}{\tau_1q}\otimes\para{y}{b})
			&=
			\para{y}{\Sub_{\tau_1}^{\tau_1q}(b)-\pi(bq)}.
			\label{eq:r1-right-correction-general}
		\end{align}
		When $b$ is built from the connector $\tau_1$ and the corresponding
		substitution is defined, the two terms in each displayed difference coincide
		and cancel.  For an alternative connector, one instead uses the two displayed
		formulas directly; \Cref{MSBF} already accounts for every remaining term.
		
		\item After the cases with $\mC_x^a(y;b)\neq0$ have been treated, impose
		$\mC_x^a(y;b)=0$.  The left and right endpoint extensions
		for every $b\in\Pi_2(\rho)$ are exactly those in
		\Cref{cor:mixed-endpoint-extension-general} and are not repeated below.
		Likewise, the formulas in \eqref{eq:mix-idempotent-absorption} remain in force by
		\Cref{rem:mixed-idempotent-absorption-general}.  The remaining formulas record
		only convenient explicit reductions for the present form $y=\tau^w\tau_1\cdots\tau_r$.
		
		\item If $\ell(b)=1$, then $b=x$ in every nonzero substitution, and the
		bracket is covered by \Cref{01-sub-only}.  Hence we assume $\ell(b)>1$ in the
		remaining cases.
		
		\item Consider the connector $\rho=\tau_1$, and let $p$ and $q$ be boundary
		cycles compatible with this connector.  If $a=p\tau_1q$ and
		$b\in\Pi_2(\tau_1)\setminus\{\tau_1\}$, the explicit reductions are
		\begin{align}
			l_2(\para{\tau_1}{p\tau_1q}\otimes\para{y}{p\tau_1})
			&=
			\begin{cases}
				\para{y}{p\tau_1q},&\text{if }p^2=p,\\
				0,&\text{otherwise,}
			\end{cases}\\
			l_2(\para{\tau_1}{p\tau_1q}\otimes\para{y}{\tau_1q})
			&=
			\begin{cases}
				\para{y}{p\tau_1q},&\text{if }q^2=q,\\
				0,&\text{otherwise,}
			\end{cases}\\
			l_2(\para{\tau_1}{p\tau_1q}\otimes\para{y}{p\tau_1q})
			&=
			\begin{cases}
				\para{y}{p\tau_1q},&\text{if }p^2=p\text{ and }q^2=q,\\
				0,&\text{otherwise.}
			\end{cases}
		\end{align}
		The first two rows are endpoint-extension reductions with a possible
		idempotent absorption at the opposite endpoint; the third is the corresponding
		two-sided absorption case.
		
		\item Finally, let $\rho\neq\tau_1$ be an alternative connector and let
		$b\in\Pi_2(\rho)$.  By \Cref{lem:replacement-criterion}, once the direct
		substitution case has been removed, every nonzero substitution is either an
		endpoint extension from \Cref{cor:mixed-endpoint-extension-general} or an
		idempotent-absorption replacement from
		\eqref{eq:mix-idempotent-absorption}.  Thus there is no further
		$r=1$-specific substitution.  In particular, when $b=\rho$ and
		$\ell(\rho)>1$, ordinary endpoint extensions are allowed and are precisely
		those covered by \Cref{cor:mixed-endpoint-extension-general}.  Since $\rho$
		is reduced with respect to boundary-cycle removal, its first and last arrows
		are ordinary; an endpoint special loop would itself be a removable boundary
		cycle.  After excluding \Cref{cor:mixed-endpoint-extension-general}, the only remaining
		possibilities are the restricted idempotent-absorption cases in
		\eqref{mix-cycle-idempotent-absorption}.
	\end{enumerate}
	
	%=============
	\subsubsection{$r>1$}
	
	Assume first that $\ell(b)\geq1$; the case $\ell(b)=0$ will be treated
	separately in the first item below. Since $b\parallel y$, every such $b$ is a
	path in $\mB$ from $s(\tau_1)$ to $t(\tau_r)$. Starting from $b$, remove,
	whenever possible, a nontrivial cycle in $\mB$ based at $s(\tau_1)$ from its
	left end and a nontrivial cycle in $\mB$ based at $t(\tau_r)$ from its right
	end, as long as a nontrivial connecting path remains. Denote the resulting
	connecting path in $\mB$ by
	$\rho=\rho_1\cdots\rho_{\ell(\rho)}$.
	Thus $\rho$ is shortest relative to the removal of boundary cycles; it is not
	required to be globally shortest among all paths in $\mB$ with the same
	endpoints.
	
	The path $\rho$ is not necessarily unique. By \textup{(G2)}, there are at
	most two possible first arrows at $s(\tau_1)$. Once the first arrow has been
	chosen, \textup{(G4)} uniquely determines each subsequent arrow as long as
	the path remains in $\mB$. Hence there are at most two possible choices of
	$\rho$.
	
	For a fixed choice of $\rho$, let $p$ and $q$ denote the possible nontrivial
	boundary loops or cycles in $\mB$ based at $s(\tau_1)$ and $t(\tau_r)$,
	respectively, for which
	$p\rho\in\mB$, $\rho q\in\mB$.
	By \Cref{loop-cycle-unique}, each of $p$ and $q$ is unique whenever it exists,
	but its availability depends on the chosen $\rho$. The local configuration is
	\[
	\begin{tikzcd}
		\bullet && \bullet \\
		\\
		\bullet && \bullet
		\arrow["p", dotted, from=1-1, to=1-1, loop, in=100, out=170, distance=10mm]
		\arrow["{\tau_1}", from=1-1, to=1-3]
		\arrow["\rho"', curve={height=24pt}, dotted, from=1-1, to=3-1]
		\arrow[dotted, from=1-3, to=3-3]
		\arrow[dotted, from=3-1, to=1-1]
		\arrow["q", dotted, from=3-1, to=3-1, loop, in=190, out=260, distance=10mm]
		\arrow["{\tau_r}", from=3-3, to=3-1]
	\end{tikzcd}
	\]
	
	By \Cref{loop-cycle-unique,basis-case}, every nontrivial path in $\mB$ parallel
	to $y$ is obtained by adjoining at most one compatible boundary piece on each
	side of one of the possible choices of $\rho$. We therefore retain the
	notation
	\[
	\Pi_3=\{\rho,\ p\rho,\ \rho q,\ p\rho q\},
	\]
	with the convention that $\rho$ ranges over all possible choices, while $p$
	and $q$ are taken relative to that choice; nonexistent terms or terms not lying in $\mB$
	are omitted. The same value path may admit more than one such description.
	Whenever a particular $b\in\Pi_3$ is considered below, we fix one compatible
	description and use $\rho,p,q$ for the corresponding data.
	
	Since $a\neq x$, one has $\mC_x^a(y;b)\neq0$ only for
	$\para{x}{a}=\para{\tau_1}{p\tau_1}$ or
	$\para{x}{a}=\para{\tau_r}{\tau_rq}$.
	In the first case $p\tau_1\in\mB$. Whenever the same $p$ is compatible
	with the chosen $\rho$, one also has $p\rho\in\mB$. Hence both
	$p_{\ell(p)}\tau_1$ and $p_{\ell(p)}\rho_1$ lie outside $I^S$, and
	\textup{(G4)} gives $\rho_1=\tau_1$.
	Similarly, in the second case $\tau_rq\in\mB$, and whenever the same $q$ is
	compatible with the chosen $\rho$, one also has $\rho q\in\mB$. Thus both
	$\tau_rq_1$ and $\rho_{\ell(\rho)}q_1$ lie outside $I^S$, so
	\textup{(G4)} gives $\rho_{\ell(\rho)}=\tau_r$.
	In the cases below with $\mC_x^a(y;b)\neq0$, the nonvanishing of the
	corresponding projected product allows the description of $b$ to be chosen
	compatibly with the path factor occurring in that product.
	
	\begin{enumerate}
		\item Assume first that $\mC_x^a(y;b)\neq0$.
		\begin{enumerate}
			\item If $\ell(b)=0$, then $b=e_{s(y)}$ and
			$s(\tau_1)=t(\tau_r)$.  The local configuration is
			
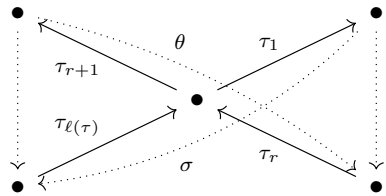
\[
			\begin{tikzcd}
				\bullet &&&& \bullet \\
				&& \bullet \\
				\bullet &&&& \bullet
				\arrow[dotted, from=1-1, to=3-1]
				\arrow["\theta"{xshift=-15pt,yshift=8pt}, curve={height=-12pt}, dotted, from=1-1, to=3-5]
				\arrow["\sigma"{xshift=-15pt,yshift=-8pt}, curve={height=-18pt}, dotted, from=1-5, to=3-1]
				\arrow[dotted, from=1-5, to=3-5]
				\arrow["{\tau_{r+1}}", from=2-3, to=1-1]
				\arrow["{\tau_1}", from=2-3, to=1-5]
				\arrow["\taue", from=3-1, to=2-3]
				\arrow["{\tau_r}", from=3-5, to=2-3]
			\end{tikzcd}
			\]
			Here one may write $p=\tau_{r+1}\theta\tau_r$ and
			$q=\tau_1\sigma\taue$ for suitable connecting paths $\theta$ and $\sigma$.  The substitution
			term is zero, and the two nonzero terms in $\mC_x^a(y;b)$ give
			\begin{align}
				l_2(\para{\tau_1}{p\tau_1}\otimes\para{y}{e_{s(y)}})
				&=-\para{y}{p},\\
				l_2(\para{\tau_r}{\tau_rq}\otimes\para{y}{e_{s(y)}})
				&=-\para{y}{q}.
			\end{align}
			We assume $\ell(b)\ge1$ from now on.
			
			\item Let $\para{x}{a}=\para{\tau_1}{p\tau_1}$, and suppose that $p$
			does not pass through $t(\tau_r)$.  Then $\rho_1=\tau_1$ and
			$b\in\Pi_3$.  The local configuration is
			
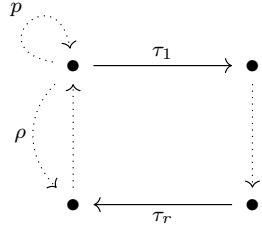
\[
			\begin{tikzcd}
				\bullet && \bullet \\
				\\
				\bullet && \bullet
				\arrow["p", dotted, from=1-1, to=1-1, loop, in=100, out=170, distance=10mm]
				\arrow["{\tau_1}", from=1-1, to=1-3]
				\arrow["\rho"', curve={height=18pt}, dotted, from=1-1, to=3-1]
				\arrow[dotted, from=1-3, to=3-3]
				\arrow[dotted, from=3-1, to=1-1]
				\arrow["{\tau_r}", from=3-3, to=3-1]
			\end{tikzcd}
			\]
			For every such $b$, the two terms in \Cref{MSBF} cancel:
			\[
			l_2(\para{\tau_1}{p\tau_1}\otimes\para{y}{b})
			=\para{y}{pb}-\para{y}{pb}=0.
			\]
			
			\item Let $\para{x}{a}=\para{\tau_1}{p\tau_1}$, and suppose that $p$
			passes through $t(\tau_r)$.  Write $p=\eta \zeta$ as in
			
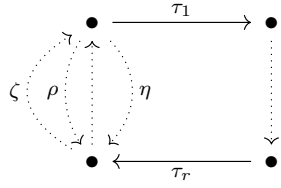
\[
			\begin{tikzcd}
				\bullet && \bullet \\
				\\
				\bullet && \bullet
				\arrow["{\tau_1}", from=1-1, to=1-3]
				\arrow["\rho"', curve={height=12pt}, dotted, from=1-1, to=3-1]
				\arrow["\eta", curve={height=-18pt}, dotted, from=1-1, to=3-1]
				\arrow[dotted, from=1-3, to=3-3]
				\arrow[dotted, from=3-1, to=1-1]
				\arrow["\zeta", curve={height=-30pt}, dotted, from=3-1, to=1-1]
				\arrow["{\tau_r}", from=3-3, to=3-1]
			\end{tikzcd}
			\]
			Here $\rho$ begins with $\tau_1$, while $\eta$ does not contain
			$\tau_1$, and $b\in\{\rho,\eta,p\rho\}$.  The only additional
			value path is $b=\eta$, for which
			\[
			l_2(\para{\tau_1}{p\tau_1}\otimes\para{y}{\eta})
			=-\para{y}{p\eta}.
			\]
			Condition \textup{(G4)} forces the junction $\zeta\eta$ in
			$p\eta=\eta\zeta\eta$ to belong to $I^S$. Hence
			$p\eta=0$, and this bracket vanishes.
			
			\item The case
			$\para{x}{a}=\para{\tau_r}{\tau_rq}$ is excluded by the left-right
			symmetric argument and produces no further nonzero bracket.
		\end{enumerate}
		
		\item We now impose
		$\mC_x^a(y;b)=0$
		and
		$\Sub_x^a(b)\neq0$.
		Thus the two boundary cochains
		$\para{\tau_1}{p\tau_1}$ and $\para{\tau_r}{\tau_rq}$ are excluded, even
		though the boundary paths $p$ and $q$ themselves may still exist.  Let
		$x=\bi$, $1\le i_0\le\ell(b)$.
		
		\begin{enumerate}
			\item If $\ell(b)=1$, then $b=x$ and the bracket is covered by
			\Cref{01-sub-only}:
			\[
			l_2(\para{x}{a}\otimes\para{y}{x})=\para{y}{a}.
			\]
			Hence we assume $\ell(b)>1$ below.
			
			\item Consider a left endpoint extension
			$\para{x}{a}=\para{\beta_1}{p\beta_1\mu}$, where
			$\mu\in\mR_1=\{e_{t(\beta_1)}\}\cup
			\{\beta_2 \mid \beta_2\in\Sp\}$ and $p$ is nontrivial.
			
			If $\mu=\beta_2\in\Sp$, the additional special loop is absorbed and,
			under the assumption $\mC_x^a(y;b)=0$,
			\begin{align}
				l_2(\para{\beta_1}{p\beta_1\beta_2}\otimes\para{y}{b})
				=\para{y}{pb}.
			\end{align}
			
			Now let $\mu=e_{t(\beta_1)}$, so $a=p\beta_1$. The equality $\mC_x^a(y;b)=0$ forces $\beta_1\neq\tau_1$. Since $p\beta_1\in\mB$, \Cref{basis-case} gives $p_1\neq\beta_1$. Because $p_1$,
			$\beta_1$, and $\tau_1$ start at the same vertex, condition
			\textup{(G2)} therefore forces $p_1=\tau_1$. At the other endpoint,
			$p_{\ell(p)}\beta_1\notin I^S$ because $p\beta_1\in\mB$.
			Since $\beta_1\neq\tau_1$ are the two arrows starting at
			$s(\tau_1)$, condition \textup{(G4)} gives
			$p_{\ell(p)}\tau_1\in I^S$.
			Comparing this relation with $\taue\tau_1\in I^S$, condition
			\textup{(G3)} gives
			$p_{\ell(p)}=\taue$.
			Consequently,
			$p=\tau_1\sigma\taue$
			for a suitable path $\sigma$.
			The local configuration is
			
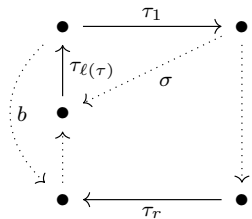
\[
			\begin{tikzcd}
				\bullet && \bullet \\
				\bullet \\
				\bullet && \bullet
				\arrow["{\tau_1}", from=1-1, to=1-3]
				\arrow["b", curve={height=24pt}, dotted, from=1-1, to=3-1]
				\arrow["\sigma", dotted, from=1-3, to=2-1]
				\arrow[dotted, from=1-3, to=3-3]
				\arrow["\taue"', from=2-1, to=1-1]
				\arrow[dotted, from=3-1, to=2-1]
				\arrow["{\tau_r}", from=3-3, to=3-1]
			\end{tikzcd}
			\]
			and
			\begin{align}
				l_2(\para{\beta_1}{p\beta_1}\otimes\para{y}{b})
				=\para{y}{pb},
			\end{align}
			where $\beta_1\neq\tau_1$.
			
			\item The right endpoint extension is obtained by left-right symmetry.  In the
			non-idempotent case one has
			$q=\tau_{r+1}\sigma\tau_r$
			for a suitable path $\sigma$, and the local configuration is
			
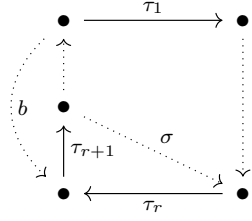
\[
			\begin{tikzcd}
				\bullet && \bullet \\
				\bullet \\
				\bullet && \bullet
				\arrow["{\tau_1}", from=1-1, to=1-3]
				\arrow["b", curve={height=24pt}, dotted, from=1-1, to=3-1]
				\arrow[dotted, from=1-3, to=3-3]
				\arrow[dotted, from=2-1, to=1-1]
				\arrow["\sigma", dotted, from=2-1, to=3-3]
				\arrow["{\tau_{r+1}}"', from=3-1, to=2-1]
				\arrow["{\tau_r}", from=3-3, to=3-1]
			\end{tikzcd}
			\]
			The resulting brackets are
			\begin{align}
				l_2(\para{\bet}{\bet q}\otimes\para{y}{b})
				&=\para{y}{bq},
				\text{ where } \bet\neq\gamma_{n+1},\\
				l_2(\para{\bet}{\beta_{\ell(b)-1}\bet q}\otimes\para{y}{b})
				&=\para{y}{bq},
				\text{ where } \beta_{\ell(b)-1}\in\Sp.
				\label{eq:mixed-local-last}
			\end{align}
			
			\item Finally, for every $b$ with $\ell(b)>1$, the remaining terms arising from idempotent absorption are precisely those in
			\eqref{eq:mix-idempotent-absorption}.
		\end{enumerate}
	\end{enumerate}
	
	%=============================================
	\begin{proof}[Proof of \Cref{prop:mixed-classification}]
		By \Cref{MSBF}, every bracket of type $(0,+)$ is determined by
		$\Sub_x^a(b)$ and $\mC_x^a(y;b)$.  The first four cases are given
		by
		\Cref{01-x-equal-to-a,cor:mixed-endpoint-extension-general,01-sub-only,cor:mixed-loop-deletion}.  For the remaining cases, the four possible forms
		of $y$ in \Cref{gamma-cases} exhaust all possibilities.  The case
		$y=\omega^{n+1}$ produces no additional bracket, while the other three forms
		are treated in
		\Cref{subsec:mixed-no-repetition,subsec:mixed-primitive-powers,subsec:mixed-suffix-powers} and give exactly
		\eqref{eq:mixed-local-first}--\eqref{eq:mixed-local-last}, subject to the
		hypotheses stated there.  Hence the list is exhaustive, and no other bracket
		of type $(0,+)$ can occur.
	\end{proof}
	
	A compact lookup table for brackets of type $(0,+)$ is given in
	\Cref{app:mixed-lookup}.
	%=============================================
	\section{Brackets of type \texorpdfstring{$(+,+)$}{(+,+)}}\label{sec:positive-positive-classification}
	
	This section gives explicit basis-level formulas for brackets of type $(+,+)$,
	that is,
	\[
	l_2:B^m\otimes B^n\longrightarrow B^{m+n},\quad m,n\geq1.
	\]
	We first record the general occurrence formula and the reductions that will be
	used in all cases.
	
	%===========================
	\subsection{General formula for brackets of type \texorpdfstring{$(+,+)$}{(+,+)}}
	
	Assume $m,n\geq1$ and let
	$\gamma=\gamma_1\cdots\gamma_{m+n+1}\in\Gamma_{m+n+1}$.
	For $z\in\Gamma_{m+1}$, $c,\xi\in\mB$, and $0\leq k\leq m$, define
	\[
	\mD_k^{(m,n)}(z,c;\gamma;\xi)=
	\begin{cases}
		\pi(pc),
		&\text{if }k=0,\ \xi=p\nu
		\text{ and } z=\nu\gamma_{n+2}\cdots\gamma_{m+n+1},\\
		c,
		&\text{if } 0<k<m,\ \xi=\nu
		\text{ and } z=\gamma_1\cdots\gamma_k\nu\gamma_{k+n+2}\cdots\gamma_{m+n+1},\\
		\pi(cq),
		&\text{if }k=m,\ \xi=\nu q
		\text{ and }z=\gamma_1\cdots\gamma_m\nu,\\
		0,&\text{otherwise.}
	\end{cases}
	\]
	Here $\nu$ is an arrow and $p,q$ may be trivial paths.
	The notation $\mD_k^{(n,m)}$ uses the same definition with $m$ and $n$ interchanged.
	
	\begin{proposition}\label{HOBF}
		For $f=\para{x}{a}\in B^m$, $g=\para{y}{b}\in B^n$, $m,n\geq1$, and every
		$\gamma\in\Gamma_{m+n+1}$,
		\begin{align*}
			l_2(f\otimes g)(\gamma)
			={}&\sum_{\substack{0\leq i\leq m\\
					\gamma_{i+1}\cdots\gamma_{i+n+1}=y}}
			(-1)^{i(n+2)}
			\mD_i^{(m,n)}(x,a;\gamma;b)\\
			&-(-1)^{mn}
			\sum_{\substack{0\leq j\leq n\\
					\gamma_{j+1}\cdots\gamma_{j+m+1}=x}}
			(-1)^{j(m+2)}
			\mD_j^{(n,m)}(y,b;\gamma;a).
		\end{align*}
		Consequently, a nonzero value can occur only when $y$ occurs as a contiguous
		subpath of $\gamma$ or $x$ occurs as a contiguous subpath of $\gamma$.
	\end{proposition}
	
	\begin{proof}
		Apply \Cref{eq:general-l2}.  In the first insertion direction,
		$\para{y}{b}$ can be nonzero only when the contiguous subpath
		$\gamma_{i+1}\cdots\gamma_{i+n+1}$ equals $y$.  Once this condition is imposed,
		the comparison map $G_{m+1}$ can reassemble the inserted value $b$ only at
		the left boundary, at an interior single-arrow overlap, or at the right
		boundary.  These are exactly the three nonzero cases in
		$\mD_i^{(m,n)}(x,a;\gamma;b)$.  The second insertion direction is obtained
		by exchanging $(x,a,m)$ and $(y,b,n)$ and multiplying by the global sign
		$-(-1)^{mn}$.
	\end{proof}
	%==============================
	\begin{corollary}
		\label{cor:higher-value-length-reduction}
		Let $f=\para{x}{a}\in B^m$ and $g=\para{y}{b}\in B^n$, where $m,n\geq1$.
		Then the following statements hold.
		
		\begin{enumerate}
			\item If $b$ is trivial, then every term in the first sum of
			\Cref{HOBF} vanishes.  If $a$ is trivial, then every term in the
			second sum vanishes.
			
			\item If $\ell(b)\neq1$, then no interior term
			$\mD_i^{(m,n)}(x,a;\gamma;b)$, $0<i<m$,
			can be nonzero.  Similarly, if $\ell(a)\neq1$, then no interior term
			in the second sum can be nonzero.
			
			\item Consequently, if neither $a$ nor $b$ is an arrow, then only
			the four boundary positions
			$i=0$, $i=m$, $j=0$ and $j=n$
			can contribute.  If both $a$ and $b$ are trivial, then $l_2(f\otimes g)=0$.
		\end{enumerate}
	\end{corollary}
	%================================
	\begin{corollary}
		\label{cor:higher-arrow-reassembly}
		Let $f=\para{x}{a}\in B^m$ and $g=\para{y}{b}\in B^n$, where
		$m,n\geq1$ and $b$ is an arrow.
		For
		$\gamma\in\Gamma_{m+n+1}$, define
		\[
		\mI_{y,b}^{x}(\gamma)
		=
		\left\{
		i\in\{0,\ldots,m\}
		\ \middle|\
		\begin{array}{l}
			y=\gamma_{i+1}\cdots\gamma_{i+n+1},\\
			x=\gamma_1\cdots\gamma_i
			b
			\gamma_{i+n+2}\cdots\gamma_{m+n+1}
		\end{array}
		\right\}.
		\]
		If $\mI_{y,b}^{x}(\gamma)=\varnothing$, then the corresponding
		contribution from the first sum in \Cref{HOBF} is zero.  Otherwise,
		the corresponding cochain contribution supported on $\gamma$ is
		\[
		\left(
		\sum_{i\in\mI_{y,b}^{x}(\gamma)}
		(-1)^{i(n+2)}
		\right)\para{\gamma}{a}.
		\]
		
		Similarly, suppose that $a$ is an arrow.  If
		$\mI_{x,a}^{y}(\gamma)=\varnothing$, then the corresponding
		contribution from the second sum is zero.  Otherwise, the corresponding
		cochain contribution supported on $\gamma$ is
		\[
		-(-1)^{mn}
		\left(
		\sum_{j\in\mI_{x,a}^{y}(\gamma)}
		(-1)^{j(m+2)}
		\right)\para{\gamma}{b}.
		\]
	\end{corollary}
	
	%============================
	\begin{corollary}
		\label{cor:higher-trivial-value}
		Let $f=\para{x}{e_{s(x)}}\in B^m$ and $g=\para{y}{b}\in B^n$, where
		$m,n\geq1$.
		Then the second sum in \Cref{HOBF} vanishes identically.  Each nonzero
		term in the first sum gives a cochain
		contribution supported on $\gamma$ of exactly one of the following forms.
		\begin{enumerate}
			\item Left boundary:
			$i=0$, $b=p\nu$ and $x=\nu\gamma_{n+2}\cdots\gamma_{m+n+1}$. The contribution is
			$\para{\gamma}{p}$.
			\item Interior:
			$0<i<m$, $b=\nu$ and $x=\gamma_1\cdots\gamma_i\nu\gamma_{i+n+2}\cdots\gamma_{m+n+1}$. 
			The contribution is $(-1)^{i(n+2)}\para{\gamma}{e_{s(x)}}$.
			\item Right boundary:
			$i=m$, $b=\nu q$ and $x=\gamma_1\cdots\gamma_m\nu$.
			The contribution is $(-1)^{m(n+2)}\para{\gamma}{q}$.
		\end{enumerate}
		Here $\nu$ is an arrow.  The corresponding formulas when the value of $g$
		is trivial are obtained by graded skew-symmetry.
	\end{corollary}
	%============================
	Before stating the next lemma, recall that for a real number $\vartheta$,
	$\lfloor \vartheta\rfloor$
	denotes the greatest integer not exceeding $\vartheta$.
	
	Let $z=z_1\cdots z_r\in\Gamma_r$ and
	$\gamma=\gamma_1\cdots\gamma_N\in\Gamma_N$, with $1\leq r\leq N$.
	We say that $z$ occurs in $\gamma$ at position $i$ if
	\[
	z_j=\gamma_{i+j}\quad(1\leq j\leq r),
	\]
	where $0\leq i\leq N-r$, and set
	\[
	\Occ_\gamma(z)=\{i \mid z\text{ occurs in }\gamma\text{ at position }i\}.
	\]
	If $\Occ_\gamma(z)\neq\varnothing$, we call $z$ a \emph{relation subpath}
	of $\gamma$. Distinct positions are counted as distinct occurrences, even when
	the corresponding relation subpaths are equal as paths. This convention agrees
	with the position numbering in \Cref{HOBF}: the first arrow $\gamma_1$ has
	position $0$.
	
	\begin{lemma}
		\label{lem:periodic-occurrences}
		Let
		$\tau=\tau_1\cdots\tau_{\ell(\tau)}$
		be a primitive relation cycle with $\ell(\tau)>1$, and let
		$\gamma=\tau^w\tau_1\cdots\tau_r$, where
		$0\leq r<\ell(\tau)$.
		Let $z=z_1\cdots z_s\in\Gamma_s$, with $s\geq1$, and suppose that its arrows follow the cyclic order of $\tau$ starting with $\tau_u$, that is,
		\[
		z_j=\tau_{1+((u+j-2)\bmod \ell(\tau))}
		\quad(1\leq j\leq s),
		\]
		for some $1\leq u\leq\ell(\tau)$.  Then
		\[
		\Occ_{\gamma}(z)
		=
		\left\{
		u-1+k\ell(\tau)
		\ \middle|\
		\begin{array}{l}
			k\in\mathbb Z_{\geq0},\\[1mm]
			u-1+k\ell(\tau)+\ell(z)
			\leq w\ell(\tau)+r
		\end{array}
		\right\}.
		\]
		In particular,
		\[
		\#\Occ_{\gamma}(z)
		=
		\max\left\{
		0,
		1+
		\left\lfloor
		\frac{w\ell(\tau)+r-\ell(z)-(u-1)}{\ell(\tau)}
		\right\rfloor
		\right\}.
		\]
	\end{lemma}
	
	\begin{proof}
		As observed in the proof of \Cref{gamma-cases}, the arrows of a primitive
		relation cycle of length greater than one are pairwise distinct. Hence the
		first arrow $\tau_u$, together with the cyclic-order condition above,
		determines the residue class of the starting position of an occurrence of $z$
		modulo $\ell(\tau)$. In the indexing convention of
		\Cref{HOBF}, the first possible occurrence therefore starts at position
		$u-1$, and all later possible occurrences start one complete period later, at
		$u-1+k\ell(\tau)$, where $k\in\mathbb Z_{\geq0}$.
		Such a starting position gives an actual occurrence in $\gamma$ precisely
		when the whole path $z$ lies inside $\gamma$, namely when
		$u-1+k\ell(\tau)+\ell(z)\leq w\ell(\tau)+r$.
		Solving this inequality for $k$ gives the displayed description of
		$\Occ_\gamma(z)$ and its cardinality.
	\end{proof}
	%===============================
	\begin{corollary}
		\label{cor:signed-periodic-multiplicity}
		Under the assumptions of
		\Cref{lem:periodic-occurrences}, let $M\in\mathbb Z$.  Then
		\[
		\sum_{i\in\Occ_\gamma(z)}(-1)^{iM}
		=
		\begin{cases}
			(-1)^{(u-1)M}\#\Occ_\gamma(z),
			&\text{if }\ell(\tau)M\text{ is even},\\
			0,
			&\text{if }\ell(\tau)M\text{ is odd and }
			\#\Occ_\gamma(z)\text{ is even},\\
			(-1)^{(u-1)M},
			&\text{if }\ell(\tau)M\text{ is odd and }
			\#\Occ_\gamma(z)\text{ is odd}.
		\end{cases}
		\]
		The integer coefficients on the right are viewed in $\mK$.
	\end{corollary}
	
	\begin{proof}
		By \Cref{lem:periodic-occurrences}, the occurrence positions are
		$i=u-1+k\ell(\tau)$, where $0\leq k<\#\Occ_\gamma(z)$.
		Hence
		\[
		\sum_{i\in\Occ_\gamma(z)}(-1)^{iM}
		=
		(-1)^{(u-1)M}
		\sum_{k=0}^{\#\Occ_\gamma(z)-1}
		(-1)^{k\ell(\tau)M}.
		\]
		The result follows by evaluating the final alternating sum.
	\end{proof}
	%================================
	\begin{remark}
		The role of \Cref{lem:periodic-occurrences} is only to determine where a
		periodic support occurs and how many occurrences it has.  The value produced
		at each occurrence is still determined by \Cref{HOBF}.
		
		For example, if every occurrence of $z$ contributes the same value through
		the first insertion direction, then the corresponding coefficient is
		\[
		\sum_{i\in\Occ_{\gamma}(z)}
		(-1)^{i(n+2)}.
		\]
		By the lemma, every $i\in\Occ_{\gamma}(z)$ has the form
		$i=u-1+k\ell(\tau)$.
		Hence the signed coefficient becomes
		\[
		\sum_{k=0}^{\#\Occ_{\gamma}(z)-1}
		(-1)^{(u-1+k\ell(\tau))(n+2)}.
		\]
		The analogous coefficient in the second insertion direction is obtained by
		replacing $n+2$ with $m+2$.
	\end{remark}
	%============================
	\begin{corollary}
		\label{cor:periodic-read-off}
		Let $f=\para{x}{a}\in B^m$ and
		$g=\para{y}{b}\in B^n$, where $m,n\geq1$. 
		Let $\tau=\tau_1\cdots\tau_{\ell(\tau)}$ be a primitive relation cycle
		with $\ell(\tau)>1$.
		Let $\gamma=\tau^w\tau_1\cdots\tau_r$, where
		$0\leq r<\ell(\tau)$ and $w\ell(\tau)+r=m+n+1$.  
		Then
		\begin{align*}
			l_2(f\otimes g)(\gamma)
			={}&
			\sum_{i\in\Occ_\gamma(y)}
			(-1)^{i(n+2)}
			\mD_i^{(m,n)}(x,a;\gamma;b)
			\\
			&-
			(-1)^{mn}
			\sum_{j\in\Occ_\gamma(x)}
			(-1)^{j(m+2)}
			\mD_j^{(n,m)}(y,b;\gamma;a).
		\end{align*}
		In particular, the occurrence positions and their signs are determined by
		\Cref{lem:periodic-occurrences}; it remains only to determine which local
		terms $\mD_i^{(m,n)}$ are nonzero.
	\end{corollary}
	%============================
	\begin{lemma}
		\label{lem:higher-boundary-dichotomy}
		Let $f=\para{x}{a}\in B^m$ and $g=\para{y}{b}\in B^n$, where $m,n\geq1$.
		\begin{enumerate}
			\item Suppose that a left boundary term in the first sum of
			\Cref{HOBF} is defined by
			$b=p\nu$,
			where $\nu$ is an arrow and $\pi(pa)\neq0$.
			Assume that $p$ and $a$ are nontrivial. Then exactly one of the
			following alternatives occurs:
			\begin{enumerate}
				\item $a$ starts with $\nu$;
				\item $p_{\ell(p)}=\alpha_1=\ep$ for some $\ep\in\Sp$.
			\end{enumerate}
			
			\item Suppose that a right boundary term in the first sum is defined by
			$b=\nu q$,
			where $\nu$ is an arrow and $\pi(aq)\neq0$.
			Assume that $a$ and $q$ are nontrivial. Then exactly one of the
			following alternatives occurs:
			\begin{enumerate}
				\item $a$ ends with $\nu$;
				\item $\alpha_{\ell(a)}=q_1=\ep$ for some $\ep\in\Sp$.
			\end{enumerate}
		\end{enumerate}
		
		The corresponding statements for the second sum are obtained by exchanging
		$(x,a,m)$ and $(y,b,n)$.
	\end{lemma}
	
	\begin{proof}
		For the first assertion, write $p=p_1\cdots p_{\ell(p)}$ and $a=\alpha_1\cdots \alpha_{\ell(a)}$.
		Since $b=p\nu\in\mB$,
		$p_{\ell(p)}\nu\notin I^S$.
		If $p_{\ell(p)}\alpha_1\notin I^S$, then
		$p_{\ell(p)}\nu\notin I^S$ and condition \textup{(G4)} give
		$\alpha_1=\nu$, so $a$ starts with $\nu$.  If $p_{\ell(p)}\alpha_1\in I^S$, then the product
		$pa$ can survive in $A$ only through a special idempotent square, and
		hence
		$p_{\ell(p)}=\alpha_1=\varepsilon\in\Sp$.
		The two alternatives are mutually exclusive: if both held, then
		$\nu=\alpha_1=p_{\ell(p)}=\varepsilon$, so the path $b=p\nu\in\mB$
		would contain the forbidden subpath $\varepsilon^2$.
		
		The second assertion follows by the left-right symmetric argument.
	\end{proof}
	%=============================
	\begin{remark}
		\label{rem:higher-boundary-cancellation}
		Let $f=\para{x}{a}\in B^m$ and $g=\para{y}{b}\in B^n$, where $m,n\geq1$.
		Let $\gamma=\gamma_1\cdots\gamma_{m+n+1}$.
		\begin{enumerate}
			\item Suppose $y=\gamma_1\cdots\gamma_{n+1}$, $x=\gamma_{n+1}\cdots\gamma_{m+n+1}$
			and put $\nu=\gamma_{n+1}$.
			If $b=p\nu$ and $a=\nu q$,
			then the corresponding first and fourth boundary terms cancel:
			\[
			\para{\gamma}{\pi(pa)}-\para{\gamma}{\pi(bq)}
			=\para{\gamma}{p\nu q}-\para{\gamma}{p\nu q}=0.
			\]
			
			\item Suppose
			$x=\gamma_1\cdots\gamma_{m+1}$, $y=\gamma_{m+1}\cdots\gamma_{m+n+1}$
			and put $\nu=\gamma_{m+1}$.
			If $a=p\nu$ and $b=\nu q$,
			then the corresponding second and third boundary terms cancel too:
			\[
			(-1)^{mn}\para{\gamma}{\pi(aq)}
			-
			(-1)^{mn}\para{\gamma}{\pi(pb)}
			=(-1)^{mn}(\para{\gamma}{p\nu q}-\para{\gamma}{p\nu q})=0.
			\]
		\end{enumerate}
	\end{remark}
	%===============================
	\begin{corollary}
		\label{cor:higher-long-value-reduction}
		Let $f=\para{x}{a}\in B^m$ and $g=\para{y}{b}\in B^n$, where
		$m,n\geq1$ and $\ell(a),\ell(b)>1$.
		Then no interior term can occur.  Moreover, every boundary term arising
		from the first alternative in \Cref{lem:higher-boundary-dichotomy} is paired
		with a boundary term from the opposite insertion direction and cancels by
		\Cref{rem:higher-boundary-cancellation}.
		
		Consequently, every surviving contribution is an idempotent-absorption
		term of one of the following four forms.
		
		\begin{enumerate}
			\item Suppose
			$y=\gamma_1\cdots\gamma_{n+1}$,
			$x=\gamma_{n+1}\cdots\gamma_{m+n+1}$,
			$b=p\gamma_{n+1}$ and
			$p_{\ell(p)}=\alpha_1=\varepsilon\in\Sp$.
			Then the first boundary term contributes
			$\para{\gamma}{\pi(pa)}$.
			
			\item Under the same support configuration, suppose
			$a=\gamma_{n+1}q$ and
			$\beta_{\ell(b)}=q_1=\varepsilon\in\Sp$.
			Then the fourth boundary term contributes
			$-\para{\gamma}{\pi(bq)}$.
			
			\item Suppose
			$x=\gamma_1\cdots\gamma_{m+1}$,
			$y=\gamma_{m+1}\cdots\gamma_{m+n+1}$,
			$b=\gamma_{m+1}q$ and
			$\alpha_{\ell(a)}=q_1=\varepsilon\in\Sp$.
			Then the second boundary term contributes
			$(-1)^{mn}\para{\gamma}{\pi(aq)}$.
			
			\item Under the same support configuration, suppose
			$a=p\gamma_{m+1}$ and
			$p_{\ell(p)}=\beta_1=\varepsilon\in\Sp$.
			Then the third boundary term contributes
			$-(-1)^{mn}\para{\gamma}{\pi(pb)}$.
		\end{enumerate}
		
		Terms whose images under $\pi$ vanish are omitted.  If several of the
		above configurations occur simultaneously, the corresponding terms are
		added.
	\end{corollary}
	
	\begin{proof}
		By \Cref{cor:higher-value-length-reduction}, only the four boundary
		positions can contribute.  Apply
		\Cref{lem:higher-boundary-dichotomy} to each boundary term.  In the
		first alternative of \Cref{lem:higher-boundary-dichotomy}, condition
		\textup{(G3)} applied at the two junctions identifies $\nu$ with the
		corresponding junction arrow of $\gamma$.  Hence the opposite boundary
		term is present, and the two contributions cancel by
		\Cref{rem:higher-boundary-cancellation}.  Only the idempotent-absorption
		alternatives remain, with the signs prescribed by \Cref{HOBF}.
	\end{proof}
	%===============================
	\begin{proposition}
		\label{prop:positive-positive-classification}
		Let $f=\para{x}{a}\in B^m$ and $g=\para{y}{b}\in B^n$, where
		$m,n\geq1$.  For every $\gamma\in\Gamma_{m+n+1}$, the value
		$l_2(f\otimes g)(\gamma)$ is determined by exactly one of the
		following four mutually exclusive cases:
		\begin{enumerate}
			\item if $\gamma$ has no repeated arrows, the value is given by
			\Cref{NRA-read-off};
			\item if $\gamma=\omega^{m+n+1}$, the value is given by
			\Cref{prop:higher-pure-loop-read-off};
			\item if $\gamma=\tau^w$, where $w\geq2$ and $\tau$ is a primitive
			relation cycle of length greater than one, the value is given by
			\Cref{prop:higher-primitive-cycle-power};
			\item if $\gamma=\tau^w\tau_1\cdots\tau_r$, where $w\geq1$ and
			$1\leq r<\ell(\tau)$, the value is given by
			\Cref{prop:higher-primitive-cycle-suffix}.
		\end{enumerate}
		These four cases exhaust the possible forms of $\gamma$.
	\end{proposition}
	
	To verify \Cref{prop:positive-positive-classification}, we organize the local
	analysis in occurrence form.  For a fixed occurrence of $x$ or $y$, there are
	three possible local positions: the left boundary, an interior single-arrow
	position, and the right boundary.  Periodic relation concatenations may
	contain several occurrences of the same support, so the resulting bracket is
	obtained by summing the corresponding local terms with their insertion signs.
	For $\gamma=\tau^w$ and $\gamma=\tau^w\tau_1\cdots\tau_r$,
	\Cref{lem:periodic-occurrences,cor:signed-periodic-multiplicity} determine the
	occurrence positions and their signed multiplicities; the case
	$\gamma=\omega^{m+n+1}$ is treated separately below.  Throughout the local
	analysis, $\xi\in\mB$ denotes the inserted path, while $\nu$ denotes the
	single arrow appearing in a local term.
	
	By \Cref{HOBF}, a nonzero coefficient requires at least one of the supports
	$x$ or $y$ to occur in $\gamma$, and the complementary support is then
	determined by the corresponding local condition.  We follow the four
	mutually exclusive forms of $\gamma$ in \Cref{gamma-cases}, in the order
	listed in the proposition.
	%============================
	\subsection{The case where \texorpdfstring{$\gamma$}{gamma} has no repeated arrows}\label{sec:higher-NRA}
	\begin{proposition}
		\label{prop:higher-no-repetition-boundary-reduction}
		Assume that
		$\gamma=\gamma_1\cdots\gamma_{m+n+1}\in\Gamma_{m+n+1}$
		has no repeated arrows. Then no interior overlap term contributes to
		$l_2(\para{x}{a}\otimes\para{y}{b})(\gamma)$.
		More precisely,
		\begin{align*}
			l_2(\para{x}{a}\otimes \para{y}{b})(\gamma)
			={}&
			\one_{y=\gamma_1\cdots\gamma_{n+1}}
			\mD_0^{(m,n)}(x,a;\gamma;b)
			\\
			&+(-1)^{m(n+2)}
			\one_{y=\gamma_{m+1}\cdots\gamma_{m+n+1}}
			\mD_m^{(m,n)}(x,a;\gamma;b)
			\\
			&-(-1)^{mn}
			\one_{x=\gamma_1\cdots\gamma_{m+1}}
			\mD_0^{(n,m)}(y,b;\gamma;a)
			\\
			&-
			\one_{x=\gamma_{n+1}\cdots\gamma_{m+n+1}}
			\mD_n^{(n,m)}(y,b;\gamma;a).
		\end{align*}
		In particular, when $\gamma$ has no repeated arrows, the value of $l_2$ is controlled only by the two boundary pairs.
	\end{proposition}
	
	\begin{proof}
		Suppose first that an interior term in the first insertion direction is
		nonzero, at some position $0<i<m$.  Then the definition of
		$\mD_i^{(m,n)}$ forces $b=\nu$ to be an arrow and
		\[
		x=\gamma_1\cdots\gamma_i\nu
		\gamma_{i+n+2}\cdots\gamma_{m+n+1}\in\Gamma_{m+1}.
		\]
		Hence
		$\gamma_i\nu\in I^S$ and $\nu\gamma_{i+n+2}\in I^S$.
		Since $\gamma\in\Gamma_{m+n+1}$, we also have
		$\gamma_i\gamma_{i+1}\in I^S$ and $\gamma_{i+n+1}\gamma_{i+n+2}\in I^S$.
		Applying \textup{(G3)} to these four relations therefore gives
		$\nu=\gamma_{i+1}=\gamma_{i+n+1}$,
		contradicting the assumption that $\gamma$ has no repeated arrows.  Thus no
		interior term in the first insertion direction can occur.  The second
		insertion direction is identical after exchanging $(x,a,m)$ and $(y,b,n)$.
		The four displayed boundary terms are therefore exactly the surviving terms
		in \Cref{HOBF}.
	\end{proof}
	%=======================================
	\begin{proposition}\label{NRA-read-off}
		Assume that
		$\gamma=\gamma_1\cdots\gamma_{m+n+1}\in\Gamma_{m+n+1}$ has no repeated arrows.
		Let $f=\para{x}{a}\in B^m$ and $g=\para{y}{b}\in B^n$, where $m,n\geq1$.
		Then the contribution to $l_2(f\otimes g)$ supported on $\gamma$ can be
		nonzero only in one of the following two boundary types.
		%========================================
		\begin{enumerate}
			\item[Type 1:] The junction at $\gamma_{n+1}$.
			Assume that
			$y=\gamma_1\cdots\gamma_{n+1}$ and
			$x=\gamma_{n+1}\cdots\gamma_{m+n+1}$.
			Then the contribution supported on $\gamma$ is
			\[
			\begin{cases}
				\para{\gamma}{\pi(pa)},
				&\text{if }
				b=p\gamma_{n+1}
				\text{ and } a \text{ does not start with } \gamma_{n+1},
				\\
				-\para{\gamma}{\pi(bq)},
				&\text{if }
				a=\gamma_{n+1}q
				\text{ and } b \text{ does not end with } \gamma_{n+1},
				\\
				0,
				&
				\text{otherwise.}
			\end{cases}
			\]
			Here $p$ and $q$ are possibly trivial paths in $\mB$.  If the displayed projected value is zero, the corresponding term is understood to be zero.  In
			particular, if both
			$b=p\gamma_{n+1}$ and $a=\gamma_{n+1}q$
			hold, then the two boundary contributions cancel:
			\[
			\para{\gamma}{\pi(pa)}
			-
			\para{\gamma}{\pi(bq)}
			=
			0,
			\]
			because $pa=bq=p\gamma_{n+1}q$.
			%================================
			\item[Type 2:] The junction at $\gamma_{m+1}$.
			Assume that
			$x=\gamma_1\cdots\gamma_{m+1}$ and
			$y=\gamma_{m+1}\cdots\gamma_{m+n+1}$.
			Then the contribution supported on $\gamma$ is
			\[
			\begin{cases}
				(-1)^{mn}\para{\gamma}{\pi(aq)},
				&\text{if }
				b=\gamma_{m+1}q
				\text{ and } a \text{ does not end with } \gamma_{m+1},
				\\
				-(-1)^{mn}\para{\gamma}{\pi(pb)},
				&\text{if }
				a=p\gamma_{m+1}
				\text{ and } b \text{ does not start with } \gamma_{m+1},
				\\
				0,
				&
				\text{otherwise.}
			\end{cases}
			\]
			Again $p$ and $q$ are possibly trivial paths in $\mB$, and a term whose projected value is zero is interpreted as zero.  If both
			$a=p\gamma_{m+1}$ and $b=\gamma_{m+1}q$
			hold, then the two boundary contributions cancel:
			\[
			(-1)^{mn}\para{\gamma}{\pi(aq)}
			-
			(-1)^{mn}\para{\gamma}{\pi(pb)}
			=
			0,
			\]
			because $aq=pb=p\gamma_{m+1}q$.
		\end{enumerate}
	\end{proposition}
	\begin{proof}
		Since $\gamma$ has no repeated arrows, no interior overlap term can
		occur.  Hence only the four boundary terms in \Cref{HOBF} remain.
		
		The first and fourth boundary terms can be nonzero only when $y$ is the
		initial subpath $\gamma_1\cdots\gamma_{n+1}$ of $\gamma$ and $x$ is its
		terminal subpath $\gamma_{n+1}\cdots\gamma_{m+n+1}$.
		This is precisely Type~1.  In this case the first boundary term appears exactly
		when $b$ ends with $\gamma_{n+1}$, say $b=p\gamma_{n+1}$, and its value is
		$\para{\gamma}{\pi(pa)}$.  The fourth boundary term appears exactly when
		$a$ starts with $\gamma_{n+1}$, say $a=\gamma_{n+1}q$, and its value is
		$-\para{\gamma}{\pi(bq)}$.  If both occur, then both projected products are
		the same product $p\gamma_{n+1}q$, with opposite signs, so they cancel.
		
		Similarly, the second and third boundary terms can be nonzero only when $x$
		is the initial subpath $\gamma_1\cdots\gamma_{m+1}$ of $\gamma$ and $y$
		is its terminal subpath $\gamma_{m+1}\cdots\gamma_{m+n+1}$.
		This is Type~2.  The second boundary term occurs exactly when
		$b=\gamma_{m+1}q$, and gives
		\[
		(-1)^{m(n+2)}\para{\gamma}{\pi(aq)}
		=
		(-1)^{mn}\para{\gamma}{\pi(aq)}.
		\]
		The third boundary term occurs exactly when $a=p\gamma_{m+1}$, and gives
		\[
		-(-1)^{mn}\para{\gamma}{\pi(pb)}.
		\]
		If both occur, then $aq=pb=p\gamma_{m+1}q$, and the two contributions cancel.
		
		Finally, Type~1 and Type~2 cannot occur simultaneously.  Indeed, otherwise
		$x$ would have to be both $\gamma_{n+1}\cdots\gamma_{m+n+1}$ and
		$\gamma_1\cdots\gamma_{m+1}$,
		which would force $\gamma_{n+1}=\gamma_1$, contradicting the assumption that
		$\gamma$ has no repeated arrows.  Therefore the two displayed types exhaust
		all possible nonzero contributions.
	\end{proof}
	%========================================
	\begin{remark}
		The preceding proposition does not imply that
		\[
		l_2(\para{x}{a}\otimes\para{y}{b})(\gamma)=0.
		\]
		The two boundary pairs may cancel, but they need not do so.  The following
		finite-dimensional gentle example shows explicitly that a single surviving
		boundary contribution can be nonzero.  In particular, boundary contributions
		cannot be discarded in the general finite-dimensional gentle setting.
	\end{remark}
	%===========================
	\begin{remark}
		Since \Cref{NRA-read-off} gives the complete local description without
		further diagrammatic subdivision, we illustrate the four surviving
		boundary alternatives by examples rather than by separate local
		configuration tables.
	\end{remark}
	%==============================
	\begin{example}
		\label{ex:higher-no-repetition-four-types}
		Consider the gentle algebra
		$A=\mK Q/\langle I\rangle$,
		where $Q$ is as follows and
		$I=\{\gamma_1\gamma_2, \gamma_2\gamma_3\}$.
		\[
		\begin{tikzcd}[column sep=5em]
			1
			\arrow[r,"\gamma_2"]
			&
			2
			\arrow[l,bend left=35,"\gamma_1"]
			\arrow[l,bend right=35,swap,"\gamma_3"]
		\end{tikzcd}
		\]
		Set
		$\gamma=\gamma_1\gamma_2\gamma_3\in\Gamma_3$.
		The arrows of $\gamma$ are pairwise distinct.  In all four cases below, we have $m=n=1$.
		
		\begin{enumerate}
			%=========================================================
			\item[(1)] \textbf{Type 1, first boundary term.}
			
			Let
			$x=\gamma_2\gamma_3$, $y=\gamma_1\gamma_2$
			and take
			$f=\para{\gamma_2\gamma_3}{e_1}$, $g=\para{\gamma_1\gamma_2}{\gamma_3\gamma_2}$.
			Here
			$b=\gamma_3\gamma_2=p\gamma_2$, $p=\gamma_3$,
			while $a=e_1$ does not start with $\gamma_2$.  Moreover,
			$\pi(pa)=\pi(\gamma_3e_1)=\gamma_3$.
			Thus only the first Type~1 boundary term is nonzero, and
			\[
			l_2(f\otimes g)
			=
			\para{\gamma_1\gamma_2\gamma_3}{\gamma_3}.
			\]
			
			%=========================================================
			\item[(2)] \textbf{Type 1, fourth boundary term.}
			
			Let
			$x=\gamma_2\gamma_3$, $y=\gamma_1\gamma_2$,
			and take
			$f=\para{\gamma_2\gamma_3}{\gamma_2\gamma_1}$, $g=\para{\gamma_1\gamma_2}{e_2}$.
			Here
			$a=\gamma_2\gamma_1=\gamma_2q$, $q=\gamma_1$,
			while $b=e_2$ does not end with $\gamma_2$.  Moreover,
			$\pi(bq)=\pi(e_2\gamma_1)=\gamma_1$.
			Thus only the fourth Type~1 boundary term is nonzero, and
			\[
			l_2(f\otimes g)
			=
			-\para{\gamma_1\gamma_2\gamma_3}{\gamma_1}.
			\]
			
			%=========================================================
			\item[(3)] \textbf{Type 2, second boundary term.}
			
			Let
			$x=\gamma_1\gamma_2$, $y=\gamma_2\gamma_3$
			and take
			$f=\para{\gamma_1\gamma_2}{e_2}$, $g=\para{\gamma_2\gamma_3}{\gamma_2\gamma_1}$.
			Here
			$b=\gamma_2\gamma_1=\gamma_2q$, $q=\gamma_1$,
			while $a=e_2$ does not end with $\gamma_2$.  Moreover,
			$\pi(aq)=\pi(e_2\gamma_1)=\gamma_1$.
			Since
			$(-1)^{mn}=-1$,
			only the second Type~2 boundary term is nonzero, and
			\[
			l_2(f\otimes g)
			=
			-\para{\gamma_1\gamma_2\gamma_3}{\gamma_1}.
			\]
			
			%=========================================================
			\item[(4)] \textbf{Type 2, third boundary term.}
			
			Let
			$x=\gamma_1\gamma_2$, $y=\gamma_2\gamma_3$
			and take
			$f=\para{\gamma_1\gamma_2}{\gamma_3\gamma_2}$, $g=\para{\gamma_2\gamma_3}{e_1}$.
			Here
			$a=\gamma_3\gamma_2=p\gamma_2$, $p=\gamma_3$,
			while $b=e_1$ does not start with $\gamma_2$.  Moreover,
			$\pi(pb)=\pi(\gamma_3e_1)=\gamma_3$.
			Again
			$(-1)^{mn}=-1$,
			so only the third Type~2 boundary term is nonzero, and
			\[
			l_2(f\otimes g)
			=
			-(-1)^{mn}
			\para{\gamma_1\gamma_2\gamma_3}{\gamma_3}
			=
			\para{\gamma_1\gamma_2\gamma_3}{\gamma_3}.
			\]
		\end{enumerate}
	\end{example}
	%===============================
	\subsection{The case \texorpdfstring{$\gamma=\omega^{m+n+1}$}{gamma = omega power m+n+1}}
	\label{sec:higher-power-loops}
	
	Assume that
	$\gamma=\omega^{m+n+1}$,
	where $\omega$ is a loop based at the vertex $1$.
	
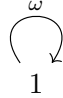
\[\begin{tikzcd}
		1
		\arrow["\omega", from=1-1, to=1-1, loop, in=55, out=125, distance=10mm]
	\end{tikzcd}\]
	Every relation subpath of $\gamma$ is a pure power of $\omega$.  Hence a
	nonzero bracket on $\gamma$ requires
	$x=\omega^{m+1}$ and $y=\omega^{n+1}$.  Moreover, the only paths in $\mB$
	parallel to $x$ or $y$ are $e_1$ and $\omega$.  Indeed, any nontrivial path in
	$\mB$ parallel to a power of $\omega$ is an oriented cycle in $\mB$ based at
	$1$, and \Cref{loop-cycle-unique} shows that $\omega$ is the unique such cycle.
	
	The case $a=b=e_1$ is zero by
	\Cref{cor:higher-value-length-reduction}.  The remaining cases are obtained
	from
	\Cref{cor:higher-trivial-value,cor:higher-arrow-reassembly}.
	
	\begin{proposition}
		\label{prop:higher-pure-loop-read-off}
		Let
		$\gamma=\omega^{m+n+1}$, where $m,n\geq1$.
		Then the possible nonzero brackets are
		\begin{align}
			l_2(
			\para{\omega^{m+1}}{e_1}
			\otimes
			\para{\omega^{n+1}}{\omega}
			)
			&=
			\left(
			\sum_{i=0}^{m}
			(-1)^{i(n+2)}
			\right)
			\para{\omega^{m+n+1}}{e_1},
			\label{eq:higher-loop-trivial-arrow}
			\\
			l_2(
			\para{\omega^{m+1}}{\omega}
			\otimes
			\para{\omega^{n+1}}{e_1}
			)
			&=
			-(-1)^{mn}
			\left(
			\sum_{j=0}^{n}
			(-1)^{j(m+2)}
			\right)
			\para{\omega^{m+n+1}}{e_1},
			\label{eq:higher-loop-arrow-trivial}
			\\
			l_2(
			\para{\omega^{m+1}}{\omega}
			\otimes
			\para{\omega^{n+1}}{\omega}
			)
			&=
			\left[
			\sum_{i=0}^{m}
			(-1)^{i(n+2)}
			-
			(-1)^{mn}
			\sum_{j=0}^{n}
			(-1)^{j(m+2)}
			\right]
			\para{\omega^{m+n+1}}{\omega}.
			\label{eq:higher-loop-arrow-arrow}
		\end{align}
		
		Each displayed bracket is nonzero precisely when its scalar coefficient is
		nonzero in $\mK$.  The first two formulas are related by the
		skew-symmetry relation \eqref{sym}.
	\end{proposition}
	
	\begin{proof}
		The support $\omega^{n+1}$ occurs in $\gamma$ at the positions
		\[
		0,1,\ldots,m,
		\]
		while $\omega^{m+1}$ occurs at the positions
		\[
		0,1,\ldots,n.
		\]
		If one value is trivial, the corresponding opposite insertion sum
		vanishes by \Cref{cor:higher-trivial-value}.  If the inserted value is the
		arrow $\omega$, every occurrence contributes the same value by
		\Cref{cor:higher-arrow-reassembly}.  Summing the insertion signs gives the
		three displayed formulas.
	\end{proof}
	%=========================================
	%=============================
	\subsection{The case \texorpdfstring{$\gamma=\tau^w$}{gamma = tau power w}}
	\label{sec:higher-power-cycle}
	Assume that $\gamma=\tau^w$, where $w\geq2$ and
	$\tau=\tau_1\cdots\taue$ is a primitive relation cycle with
	$\ell(\tau)\geq2$.  Then $\ell(\tau)w=m+n+1$.
	
	The case $w=1$ belongs to the case where $\gamma$ has no repeated arrows, treated in
	\Cref{sec:higher-NRA}.
	
	For a nonzero contribution, \Cref{HOBF} and the corresponding local
	condition force both supports $x\in\Gamma_{m+1}$ and $y\in\Gamma_{n+1}$
	to occur as relation subpaths of $\gamma$.  These subpaths are proper since
	$m,n\geq1$.  All subscripts attached to arrows of $\tau$ are read cyclically
	modulo $\ell(\tau)$.
	%============================================================
	\begin{lemma}
		\label{lem:proper-subpaths-cycle-power}
		Let $y$ be a proper relation subpath of
		$\gamma=\tau^w$.
		Then $y$ has exactly one of the following forms.
		
		\begin{enumerate}
			\item
			$y=(\tau_u\cdots\tau_{u-1})^{w_2}$,
			where
			$1\leq u\leq\ell(\tau)$,
			$1\leq w_2<w$ and
			$w_2\ell(\tau)=n+1$.
			
			\item
			$y=(\tau_u\cdots\tau_{u-1})^{w_2}
			\tau_u\cdots\tau_{u+v}$,
			where
			$1\leq u\leq\ell(\tau)$,
			$0\leq w_2<w$,
			$0\leq v<\ell(\tau)-1$,
			and
			$w_2\ell(\tau)+v+1=n+1$.
		\end{enumerate}
	\end{lemma}
	
	\begin{proof}
		As observed in the proof of \Cref{gamma-cases}, the arrows of the
		primitive relation cycle $\tau$ are pairwise distinct. Hence the first
		arrow $\tau_u$ determines the cyclic starting position, modulo $\ell(\tau)$,
		of any relation concatenation whose arrows follow the cyclic order of $\tau$.
		Write $n+1=w_2\ell(\tau)+r_y$, where
		$0\leq r_y<\ell(\tau)$.
		Once the first arrow $\tau_u$ and the length $n+1$ are fixed, condition
		\textup{(G3)} uniquely determines the relation concatenation $y$.
		
		If $r_y=0$, then $y$ consists of $w_2$ complete rotated copies of
		$\tau$, giving the first form.  Since $y$ is nontrivial and proper, $1\leq w_2<w$.
		
		If $r_y>0$, put $v=r_y-1$.
		Then $0\leq v<\ell(\tau)-1$
		and $y$ has the second form.
	\end{proof}
	
	The same description applies to $x$, after replacing $n+1$ by $m+1$.
	For either support $z=x$ or $z=y$, its possible occurrence positions in
	$\gamma=\tau^w$ are determined by
	\Cref{lem:periodic-occurrences}, with $r=0$.  In particular, if the first
	arrow of $z$ is $\tau_u$, then
	\[
	\Occ_\gamma(z)
	=
	\left\{
	u-1+k\ell(\tau)
	\ \middle|\
	\begin{array}{l}
		k\in\mathbb Z_{\geq0},\\[1mm]
		u-1+k\ell(\tau)+\ell(z)
		\leq w\ell(\tau)
	\end{array}
	\right\}.
	\]
	
	Rather than distinguishing all possible cyclic forms of $x$ and $y$
	separately, we classify the bracket according to the lengths of the two
	values $a$ and $b$.  These alternatives are exhaustive.
	%========================================================
	\subsubsection{At least one value is trivial}
	Assume first that $a=e_{s(x)}$.
	Then $x$ is closed.  This possibility cannot be excluded merely
	from the arrow subscripts of $x$, since a primitive relation cycle may visit
	the same vertex more than once.
	
	By \Cref{cor:higher-trivial-value}, the second insertion sum vanishes
	identically, and
	\begin{align}
		l_2(
		\para{x}{e_{s(x)}}
		\otimes
		\para{y}{b}
		)(\gamma)
		=
		\sum_{i\in\Occ_\gamma(y)}
		(-1)^{i(n+2)}
		\mD_i^{(m,n)}
		(x,e_{s(x)};\gamma;b).
		\label{eq:cycle-power-trivial-value}
	\end{align}
	
	The only possible nonzero local terms in
	\eqref{eq:cycle-power-trivial-value} are the left boundary, interior
	single-arrow, and right boundary terms listed in
	\Cref{cor:higher-trivial-value}.  Their occurrence positions and signs are
	determined by \Cref{lem:periodic-occurrences}.
	
	The case in which $b$ is trivial is obtained by the skew-symmetry relation
	\eqref{sym}.  If both values are trivial, then $l_2(f\otimes g)=0$ by
	\Cref{cor:higher-value-length-reduction}.
	
	Thus, for $\gamma=\tau^w$, all trivial-value brackets are classified by
	\Cref{cor:higher-trivial-value,lem:periodic-occurrences}; we do not expand
	them into separate cyclic support cases.
	
	\begin{example}
		\label{ex:cycle-power-trivial-value}
		Consider the ordinary gentle algebra
		$A=\mK Q/\langle I\rangle$,
		where $Q$ is as follows and
		$I= \{ \tau_1\tau_2,\tau_2\tau_3, \tau_3\tau_4, \tau_4\tau_1 \}$.
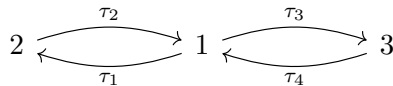
		\[
		\begin{tikzcd}[column sep=5em]
			2
			\arrow[r,bend left=18,"\tau_2"]
			&
			1
			\arrow[l,bend left=18,"\tau_1"]
			\arrow[r,bend left=18,"\tau_3"]
			&
			3
			\arrow[l,bend left=18,"\tau_4"]
		\end{tikzcd}
		\]
		The algebra $A$ is finite-dimensional and gentle, and
		$\tau=\tau_1\tau_2\tau_3\tau_4$
		is a primitive relation cycle.  The cycle $\tau$ passes through the
		vertex $1$ twice.
		
		Let
		$\gamma=\tau^2 = \tau_1\tau_2\tau_3\tau_4 \tau_1\tau_2\tau_3\tau_4 \in\Gamma_8$.
		Set
		$x=\tau_3\tau_4\tau_1\tau_2\tau_3\tau_4\in\Gamma_6$ and
		$y=\tau_1\tau_2\tau_3\in\Gamma_3$.
		Thus $m=5$ and $n=2$.
		The support $x$ starts and ends at the vertex $1$, so the trivial path
		$e_1$ is parallel to $x$.  Define
		$f=\para{x}{e_1}\in B^5$ and $g=\para{y}{\tau_3}\in B^2$.
		
		The support $y$ occurs in $\gamma$ at the positions
		$\Occ_\gamma(y)=\{0,4\}$.
		At the first occurrence, $y=\gamma_1\gamma_2\gamma_3$ and
		$x=\gamma_3\cdots\gamma_8$, while $b=\tau_3=e_1\tau_3$.
		Therefore the left boundary term is
		$\mD_0^{(5,2)}(x,e_1;\gamma;\tau_3)=\pi(e_1e_1)=e_1$.
		
		At the occurrence beginning at position $4$, replacing $y$ by
		$\tau_3$ does not produce the support $x$, so
		$\mD_4^{(5,2)}(x,e_1;\gamma;\tau_3)=0$.
		The second insertion sum vanishes because the value of $f$ is trivial.
		Hence
		\[
		l_2(f\otimes g)
		=
		\para{\gamma}{e_1}
		\neq0.
		\]
		
		This example shows that a proper relation subpath of a power of a primitive relation cycle may be closed and hence admit a trivial value.
	\end{example}
	%===============================
	\subsubsection{No value is trivial and at least one value is an arrow}
	
	Assume that $\ell(a)>0$ and $\ell(b)>0$, and, by skew-symmetry if
	necessary, suppose that $b$ is an arrow.
	
	By \Cref{cor:higher-arrow-reassembly}, the contribution of the first
	insertion direction is
	$(\sum_{i\in\mI_{y,b}^{x}(\gamma)}(-1)^{i(n+2)})a$.
	Hence
	\begin{align}
		l_2(f\otimes g)(\gamma)
		=
		\left(
		\sum_{i\in\mI_{y,b}^{x}(\gamma)}
		(-1)^{i(n+2)}
		\right)a
		-
		(-1)^{mn}
		\sum_{j\in\Occ_\gamma(x)}
		(-1)^{j(m+2)}
		\mD_j^{(n,m)}(y,b;\gamma;a).
		\label{eq:cycle-power-arrow-read-off}
	\end{align}
	
	The candidate positions in
	$\mI_{y,b}^{x}(\gamma)$ lie in $\Occ_\gamma(y)$, and are therefore
	computed by \Cref{lem:periodic-occurrences}.  The additional condition
	defining $\mI_{y,b}^{x}(\gamma)$ checks whether replacing the occurrence of
	$y$ by the arrow $b$ produces the complementary support $x$.
	
	There are two subcases.
	
	\begin{enumerate}
		\item If $a$ is also an arrow, then the second insertion direction is
		also given by \Cref{cor:higher-arrow-reassembly}, and
		\begin{align}
			l_2(f\otimes g)(\gamma)
			=
			\left(
			\sum_{i\in\mI_{y,b}^{x}(\gamma)}
			(-1)^{i(n+2)}
			\right)a
			-
			(-1)^{mn}
			\left(
			\sum_{j\in\mI_{x,a}^{y}(\gamma)}
			(-1)^{j(m+2)}
			\right)b.
			\label{eq:cycle-power-two-arrow-values}
		\end{align}
		
		\item If $\ell(a)>1$, then
		\Cref{cor:higher-value-length-reduction} shows that the second sum in
		\eqref{eq:cycle-power-arrow-read-off} can contribute only at $j=0$ or $j=n$.  Thus
		\begin{align}
			l_2(f\otimes g)(\gamma)
			=
			\left(
			\sum_{i\in\mI_{y,b}^{x}(\gamma)}
			(-1)^{i(n+2)}
			\right)a
			-
			(-1)^{mn}
			\sum_{j\in\Occ_\gamma(x)\cap\{0,n\}}
			(-1)^{j(m+2)}
			\mD_j^{(n,m)}(y,b;\gamma;a).
			\label{eq:cycle-power-arrow-long-value}
		\end{align}
		
		If a boundary term in the second line arises from the first alternative of
		\Cref{lem:higher-boundary-dichotomy}, then it is paired with the corresponding endpoint contribution from
		the first line and cancels by
		\Cref{rem:higher-boundary-cancellation}.  If it does not arise from the first alternative of
		\Cref{lem:higher-boundary-dichotomy}, its possible nonzero value is an
		idempotent-absorption term.
	\end{enumerate}
	
	Together with their skew-symmetric counterparts,
	\eqref{eq:cycle-power-two-arrow-values} and
	\eqref{eq:cycle-power-arrow-long-value} classify all brackets with
	$\gamma=\tau^w$ for which neither value is trivial and at least
	one value is an arrow. All surviving contributions are added in
	$A$, and the bracket is nonzero precisely when their sum is nonzero.
	%=========================
	\begin{example}
		\label{ex:cycle-power-arrow-reassembly}
		Let $Q^S$ be the skew-gentle quiver
		
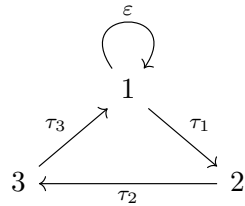
\[\begin{tikzcd}
			& 1 & \\
			3 && 2
			\arrow["\ep", from=1-2, to=1-2, loop, in=55, out=125, distance=10mm]
			\arrow["{\tau_1}", from=1-2, to=2-3]
			\arrow["{\tau_3}", from=2-1, to=1-2]
			\arrow["{\tau_2}", from=2-3, to=2-1]
		\end{tikzcd}\]
		where $\ep$ is the special loop at the vertex $1$.  Take
		$I=\{\tau_1\tau_2, \tau_2\tau_3, \tau_3\tau_1\}$ and
		$I^S=I\cup\{\ep^2\}$.
		Thus
		$A= \mK Q^S/ \langle I\cup\{\ep^2-\ep\} \rangle$
		is a finite-dimensional skew-gentle algebra.
		
		Put
		$\tau=\tau_1\tau_2\tau_3$
		and
		$\gamma=\tau^4\in\Gamma_{12}$.
		Let
		$y=\tau\tau_1 = \tau_1\tau_2\tau_3\tau_1 \in\Gamma_4$
		and
		$x=\tau^3\in\Gamma_9$.
		Hence $n=3$ and $m=8$.
		Define $f=\para{\tau^3}{\ep}\in B^8$ and
		$g=\para{\tau\tau_1}{\tau_1}\in B^3$.
		
		The support $y$ occurs in $\gamma$ at the positions
		$\Occ_\gamma(y)=\{0,3,6\}$.
		Replacing any one of these three occurrences of
		$y=\tau\tau_1$
		by the arrow $\tau_1$ removes one complete copy of $\tau$, and hence
		produces the same complementary support
		$x=\tau^3$.
		Therefore
		$\mI_{y,\tau_1}^{x}(\gamma) = \{0,3,6\}$.
		
		On the other hand,
		$\mI_{x,\ep}^{y}(\gamma)=\varnothing$,
		because replacing an occurrence of $\tau^3$ by $\ep$ does not produce
		the support $y=\tau\tau_1$.  Hence
		\begin{align*}
			l_2(f\otimes g)
			&=
			\left(
			\sum_{i\in\{0,3,6\}}
			(-1)^{i(n+2)}
			\right)
			\para{\gamma}{\ep}
			\\
			&=
			(
			1-1+1
			)
			\para{\gamma}{\ep}
			\\
			&=
			\para{\gamma}{\ep}.
		\end{align*}
		
		This example illustrates that the periodic occurrence positions determine
		the scalar coefficient, while
		\Cref{cor:higher-arrow-reassembly} determines the common local value.
	\end{example}
	%===============================
	\subsubsection{Both values have length greater than one}
	
	Finally, assume that $\ell(a)>1$ and $\ell(b)>1$.
	By \Cref{cor:higher-value-length-reduction}, no interior term can occur.  By
	\Cref{lem:higher-boundary-dichotomy,rem:higher-boundary-cancellation},
	every boundary term arising from the first alternative of
	\Cref{lem:higher-boundary-dichotomy} is paired with a term from the opposite
	insertion direction and cancels.
	
	Therefore, the only possible surviving terms are the four
	idempotent-absorption boundary terms listed in
	\Cref{cor:higher-long-value-reduction}.  No additional case is created by
	the periodic form $\gamma=\tau^w$; periodicity only determines whether the
	required boundary supports occur.
	
	In particular, once the boundary support conditions in
	\Cref{cor:higher-long-value-reduction} are satisfied, the corresponding
	projected values are $\pi(pa)$, $\pi(bq)$, $\pi(aq)$, and
	$\pi(pb)$, with the signs displayed there.  Terms with zero projected value are omitted.
	%======================================
	\begin{example}
		\label{ex:cycle-power-long-value-absorption}
		Let $Q^S$ be the quiver
		
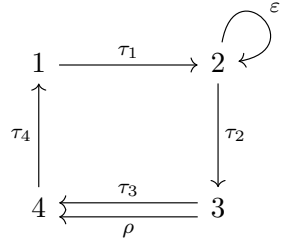
\[\begin{tikzcd}
			1 && 2 \\
			\\
			4 && 3
			\arrow["{\tau_1}", from=1-1, to=1-3]
			\arrow["\ep", from=1-3, to=1-3, loop, in=10, out=80, distance=10mm]
			\arrow["{\tau_2}", from=1-3, to=3-3]
			\arrow["{\tau_4}", from=3-1, to=1-1]
			\arrow["{\tau_3}"', shift right, from=3-3, to=3-1]
			\arrow["\rho", shift left, from=3-3, to=3-1]
		\end{tikzcd}\]
		where $\ep$ is the special loop at the vertex $2$, and let
		$I=\{\tau_1\tau_2, \tau_2\tau_3, \tau_3\tau_4, \tau_4\tau_1\}$
		and $I^S=I\cup\{\ep^2\}$.
		Then
		$A= \mK Q^S/ \langle I\cup\{\ep^2-\ep\} \rangle$
		is a finite-dimensional skew-gentle algebra.  In particular, the products
		$\tau_1\ep$, $\ep\tau_2$, $\tau_2\rho$, and $\rho\tau_4$ lie in $\mB$.
		
		Put
		$\tau=\tau_1\tau_2\tau_3\tau_4$
		and
		$\gamma=\tau^2\in\Gamma_8$.
		Set
		$y=\tau_1\tau_2\in\Gamma_2$
		and
		$x= \tau_2\tau_3\tau_4 \tau_1\tau_2\tau_3\tau_4 \in\Gamma_7$.
		Thus $n=1$ and $m=6$.
		
		Define $p=\tau_1\ep$, $b=p\tau_2=\tau_1\ep\tau_2$, and
		$a=\ep\tau_2\rho\tau_4$.
		Then $b\parallel y$, $a\parallel x$, $\ell(a)>1$, and
		$\ell(b)>1$.
		Let $f=\para{x}{a}\in B^6$ and $g=\para{y}{b}\in B^1$.
		
		The support $y$ occurs in $\gamma$ at the positions
		$\Occ_\gamma(y)=\{0,4\}$.
		At $i=0$, it is a left boundary occurrence, and
		$b=p\tau_2$.
		Moreover,
		$p_{\ell(p)}=\ep=\alpha_1$.
		Thus the interface between $p$ and $a$ is the special idempotent
		square $\ep^2$, and
		\begin{align*}
			\mD_0^{(6,1)}(x,a;\gamma;b)
			&=
			\pi(pa)
			\\
			&=
			\pi(
			\tau_1\ep\ep\tau_2\rho\tau_4
			)
			\\
			&=
			\tau_1\ep\tau_2\rho\tau_4.
		\end{align*}
		
		The occurrence beginning at position $4$ is internal.  Since
		$\ell(b)>1$, its local term is zero by
		\Cref{cor:higher-value-length-reduction}.
		
		The support $x$ occurs only at the right boundary in the second
		insertion direction.  A nonzero right boundary term there would
		require $a$ to start with the overlap arrow $\tau_2$, but
		$\alpha_1=\ep\neq\tau_2$.
		Hence the second insertion sum is zero.  Therefore
		\[
		l_2(f\otimes g)
		=
		\para{\gamma}
		{\tau_1\ep\tau_2\rho\tau_4}
		\neq0.
		\]
		
		This is a genuine idempotent-absorption term: it survives because
		$\ep^2=\ep$ in $A$,
		whereas a boundary term arising from the first alternative of
		\Cref{lem:higher-boundary-dichotomy} would be paired with a term from the
		opposite insertion direction and would cancel.
	\end{example}
	%================================================
	\begin{proposition}
		\label{prop:higher-primitive-cycle-power}
		Let $\gamma=\tau^w$, where $\tau=\tau_1\cdots\taue$ and $w\geq2$,
		and let $f=\para{x}{a}\in B^m$ and $g=\para{y}{b}\in B^n$, where
		$m,n\geq1$.  Then every nonzero value of
		$l_2(f\otimes g)(\gamma)$ belongs to exactly one of the following
		mutually exclusive cases:
		
		\begin{enumerate}
			\item at least one of $a,b$ is trivial, in which case the bracket is
			given by
			\Cref{cor:higher-trivial-value,lem:periodic-occurrences}
			and skew-symmetry;
			
			\item neither value is trivial and at least one of $a,b$ is an
			arrow, in which case the bracket is given by
			\eqref{eq:cycle-power-two-arrow-values} or
			\eqref{eq:cycle-power-arrow-long-value}, together with
			skew-symmetry;
			
			\item $\ell(a)>1$ and $\ell(b)>1$, in which case the only surviving terms are the idempotent-absorption
			boundary terms in
			\Cref{cor:higher-long-value-reduction}.
		\end{enumerate}
		
		Conversely, every term described in these three cases occurs whenever
		its support, occurrence, and nonvanishing conditions under $\pi$ are satisfied.  The
		resulting bracket is nonzero precisely when the sum of the surviving terms
		is nonzero in $A$.
	\end{proposition}
	
	\begin{proof}
		The three cases form a partition of the possible lengths of the value paths $a,b\in\mB$.  The first case follows from
		\Cref{cor:higher-trivial-value}; the second follows from
		\Cref{cor:higher-arrow-reassembly} and
		\Cref{cor:periodic-read-off}; and the third follows from
		\Cref{cor:higher-long-value-reduction}.  The occurrence positions in all
		three cases are determined by
		\Cref{lem:periodic-occurrences}.  These general results exhaust all local
		terms in \Cref{HOBF}.
	\end{proof}
	%===============================
	%================================================
	\subsection{The case \texorpdfstring{$\gamma=\tau^w\tau_1\cdots\tau_r$}{gamma = tau power w with suffix}}
	\label{sec:higher-power-suffix}
	
	Assume that $\gamma=\tau^w\tau_1\cdots\tau_r$, where
	$1\leq r<\ell(\tau)$, $w\geq1$, and
	$w\ell(\tau)+r=m+n+1$.  Here $\tau=\tau_1\cdots\taue$ is a primitive
	relation cycle, and all subscripts attached to the arrows of $\tau$ are read
	cyclically modulo $\ell(\tau)$.
	
	For nonzero contributions, the same local conditions
	show that $x\in\Gamma_{m+1}$ and $y\in\Gamma_{n+1}$ are relation subpaths
	of $\gamma$; since $m,n\geq1$, they are proper. In particular, the arrows of
	each support follow the cyclic order of $\tau$. More precisely, if a support
	$z=z_1\cdots z_s$ starts with $\tau_u$, then
	\[
	z_j=\tau_{1+((u+j-2)\bmod \ell(\tau))}
	\quad(1\leq j\leq s).
	\]
	Once the first arrow and the length of such a support are fixed, condition
	\textup{(G3)} uniquely determines the support itself.
	
	Suppose, for example, that the first arrow of a support $z$ is $\tau_u$.
	Then its occurrence positions in $\gamma$ are
	\[
	\Occ_\gamma(z)
	=
	\left\{
	u-1+k\ell(\tau)
	\ \middle|\
	\begin{array}{l}
		k\in\mathbb Z_{\geq0},\\[1mm]
		u-1+k\ell(\tau)+\ell(z)
		\leq w\ell(\tau)+r
	\end{array}
	\right\}
	\]
	by \Cref{lem:periodic-occurrences}.  Its signed occurrence multiplicities are
	given by \Cref{cor:signed-periodic-multiplicity}.
	
	The information specific to the suffix is therefore encoded entirely by the
	occurrence sets.  In the first insertion direction, $i=0$ and $i=m$ are the
	left and right boundary positions, respectively.  Hence they occur precisely
	when $0\in\Occ_\gamma(y)$ or $m\in\Occ_\gamma(y)$.  Similarly, the two
	boundary positions in the second insertion direction are present precisely
	when $0\in\Occ_\gamma(x)$ or $n\in\Occ_\gamma(x)$.  All other occurrence
	positions are interior positions.
	
	Thus the distinctions according to the cyclic shape of $x$ and $y$ are
	already encoded by $\Occ_\gamma(x)$ and $\Occ_\gamma(y)$.
	As in \Cref{sec:higher-power-cycle}, we classify the bracket according to the
	lengths of the two values $a$ and $b$.
	
	%================================================
	\subsubsection{At least one value is trivial}
	
	Assume first that $a=e_{s(x)}$.
	Then $x$ is closed.  As in the case $\gamma=\tau^w$, this possibility
	cannot be excluded from the arrow subscripts alone, since a primitive relation
	cycle may visit the same vertex more than once.
	
	By \Cref{cor:higher-trivial-value}, the second insertion sum vanishes
	identically, and
	\begin{align}
		l_2(
		\para{x}{e_{s(x)}}
		\otimes
		\para{y}{b}
		)(\gamma)
		=
		\sum_{i\in\Occ_\gamma(y)}
		(-1)^{i(n+2)}
		\mD_i^{(m,n)}
		(x,e_{s(x)};\gamma;b).
		\label{eq:suffix-trivial-value}
	\end{align}
	
	More explicitly, the possible nonzero local terms are exactly the following.
	
	\begin{enumerate}
		\item If $0\in\Occ_\gamma(y)$, $b=p\nu$, and
		$x=\nu\gamma_{n+2}\cdots\gamma_{m+n+1}$, then the left boundary
		contribution is $\para{\gamma}{p}$.
		
		\item If $i\in\Occ_\gamma(y)$ with $0<i<m$, and if $b=\nu$ is an
		arrow such that $x$ is the concatenation of
		$\gamma_1\cdots\gamma_i\nu$ and
		$\gamma_{i+n+2}\cdots\gamma_{m+n+1}$, then the interior contribution is
		$(-1)^{i(n+2)}\para{\gamma}{e_{s(x)}}$.
		
		\item If $m\in\Occ_\gamma(y)$, $b=\nu q$, and
		$x=\gamma_1\cdots\gamma_m\nu$, then the right boundary contribution is
		$(-1)^{m(n+2)}\para{\gamma}{q}$.
	\end{enumerate}
	
	The occurrence positions in these formulas are determined by
	\Cref{lem:periodic-occurrences}.  In particular, the last possible occurrence
	is automatically included only when the suffix $\tau_1\cdots\tau_r$ is
	long enough to contain it.
	
	The case in which $b$ is trivial is obtained from
	\eqref{eq:suffix-trivial-value} by the skew-symmetry relation \eqref{sym}.
	If both values are trivial, then $l_2(f\otimes g)=0$ by
	\Cref{cor:higher-value-length-reduction}.
	
	Thus every bracket with a trivial value is determined by the local formulas
	in \Cref{cor:higher-trivial-value}, the occurrence positions in
	\Cref{lem:periodic-occurrences}, and graded skew-symmetry.
	%================================================
	\begin{example}
		\label{ex:suffix-trivial-value}
		Consider the ordinary gentle algebra
		$A=\mK Q/\langle I\rangle$,
		where
		\[\begin{tikzcd}
			2 && 1 && 3
			\arrow["{\tau_2}", shift left, from=1-1, to=1-3]
			\arrow["{\tau_1}", shift left, from=1-3, to=1-1]
			\arrow["{\tau_3}", shift left, from=1-3, to=1-5]
			\arrow["{\tau_4}", shift left, from=1-5, to=1-3]
		\end{tikzcd}\]
		and
		$I= \{ \tau_1\tau_2,  \tau_2\tau_3,  \tau_3\tau_4,  \tau_4\tau_1 \}$.
		Then
		$\tau=\tau_1\tau_2\tau_3\tau_4$
		is a primitive relation cycle.  Put
		$\gamma=\tau\tau_1\tau_2 = \tau_1\tau_2\tau_3\tau_4\tau_1\tau_2 \in\Gamma_6$.
		Thus
		$w=1$ and $r=2$.
		
		Let
		$y=\tau_1\tau_2\tau_3\in\Gamma_3$ and $x=\tau_3\tau_4\tau_1\tau_2\in\Gamma_4$.
		Hence
		$n=2$ and $m=3$.
		The support $x$ is closed at the vertex $1$.  Define
		$f=\para{x}{e_1}\in B^3$ and $g=\para{y}{\tau_3}\in B^2$.
		
		By \Cref{lem:periodic-occurrences},
		$\Occ_\gamma(y)=\{0\}$.
		At this occurrence,
		$y=\gamma_1\gamma_2\gamma_3$ and $x=\gamma_3\gamma_4\gamma_5\gamma_6$.
		Writing
		$b=\tau_3=e_1\tau_3$,
		the left boundary term in
		\Cref{cor:higher-trivial-value} gives $\mD_0^{(3,2)}(x,e_1;\gamma;\tau_3)=\pi(e_1e_1)=e_1$.
		The second insertion sum vanishes because the value of $f$ is trivial.
		Therefore
		\[
		l_2(f\otimes g)
		=
		\para{\gamma}{e_1}
		\neq0.
		\]
		
		The next possible periodic starting position of $y$ would be $4$, but
		$4+\ell(y)>\ell(\gamma)$.
		Thus the final suffix is too short to contain another occurrence.  This
		illustrates how \Cref{lem:periodic-occurrences} automatically records the
		truncation caused by the suffix.
	\end{example}
	%================================================
	\subsubsection{No value is trivial and at least one value is an arrow}
	
	Assume that $\ell(a)>0$ and $\ell(b)>0$.
	By skew-symmetry, it is enough to begin with the case in which $b$ is an
	arrow.
	
	By \Cref{cor:higher-arrow-reassembly}, the contribution of the first
	insertion direction is
	$(\sum_{i\in\mI_{y,b}^{x}(\gamma)}(-1)^{i(n+2)})a$.
	Hence
	\begin{align}
		l_2(f\otimes g)(\gamma)
		=
		\left(
		\sum_{i\in\mI_{y,b}^{x}(\gamma)}
		(-1)^{i(n+2)}
		\right)a
		-
		(-1)^{mn}
		\sum_{j\in\Occ_\gamma(x)}
		(-1)^{j(m+2)}
		\mD_j^{(n,m)}(y,b;\gamma;a).
		\label{eq:suffix-arrow-read-off}
	\end{align}
	
	The candidate positions in $\mI_{y,b}^{x}(\gamma)$ belong to $\Occ_\gamma(y)$.  They are obtained by retaining exactly those
	occurrences for which replacing the support $y$ by the arrow $b$ produces
	the complementary support $x$.
	
	If the first arrow of $y$ is $\tau_u$, then
	\[
	\#\Occ_\gamma(y)
	=
	\max\left\{
	0,
	1+
	\left\lfloor
	\frac{
		w\ell(\tau)+r-(n+1)-(u-1)
	}{
		\ell(\tau)
	}
	\right\rfloor
	\right\}.
	\]
	This formula automatically distinguishes the cases in which the final suffix
	contains another occurrence from those in which it does not.  In particular,
	there is no need to treat $u<r$, $u=r$, and $u>r$ by separate
	occurrence counts.
	
	There are two subcases.
	
	\begin{enumerate}
		\item Suppose that $a$ is also an arrow.  Then the second insertion
		direction is also governed by
		\Cref{cor:higher-arrow-reassembly}, and
		\begin{align}
			l_2(f\otimes g)(\gamma)
			=
			\left(
			\sum_{i\in\mI_{y,b}^{x}(\gamma)}
			(-1)^{i(n+2)}
			\right)a
			-
			(-1)^{mn}
			\left(
			\sum_{j\in\mI_{x,a}^{y}(\gamma)}
			(-1)^{j(m+2)}
			\right)b.
			\label{eq:suffix-two-arrow-values}
		\end{align}
		
		Whenever one of the sets $\mI_{y,b}^{x}(\gamma)$ or $\mI_{x,a}^{y}(\gamma)$ coincides with the corresponding full occurrence set, its signed
		coefficient can be evaluated directly by
		\Cref{cor:signed-periodic-multiplicity}.
		
		\item Suppose that $\ell(a)>1$.  Then \Cref{cor:higher-value-length-reduction} shows that the second
		insertion direction can contribute only at the two boundary positions $j=0$ and $j=n$.
		Therefore
		\begin{align}
			l_2(f\otimes g)(\gamma)
			=
			\left(
			\sum_{i\in\mI_{y,b}^{x}(\gamma)}
			(-1)^{i(n+2)}
			\right)a
			-
			(-1)^{mn}
			\sum_{j\in\Occ_\gamma(x)\cap\{0,n\}}
			(-1)^{j(m+2)}
			\mD_j^{(n,m)}(y,b;\gamma;a).
			\label{eq:suffix-arrow-long-value}
		\end{align}
		
		If a boundary term in the second line arises from the first alternative of
		\Cref{lem:higher-boundary-dichotomy}, then the corresponding endpoint
		occurrence in the first line is present, and
		the two terms cancel by
		\Cref{rem:higher-boundary-cancellation}.  The remaining possible
		boundary terms are precisely the idempotent-absorption alternatives in
		\Cref{lem:higher-boundary-dichotomy}.
	\end{enumerate}
	
	The formulas
	\eqref{eq:suffix-two-arrow-values} and
	\eqref{eq:suffix-arrow-long-value}, together with their skew-symmetric
	counterparts, classify all suffix-power brackets in which neither value is
	trivial and at least one value is an arrow.
	
	All surviving contributions in these formulas are added in $A$; the bracket
	is nonzero precisely when their total sum is nonzero.
	%==================================
	\begin{example}
		\label{ex:suffix-arrow-reassembly}
		Consider the ordinary gentle algebra
		$A=\mK Q/\langle I\rangle$,
		where
		
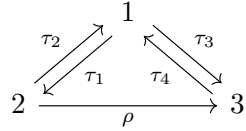
\[\begin{tikzcd}
			& 1 & \\
			2 && 3
			\arrow["{\tau_1}", shift left, from=1-2, to=2-1]
			\arrow["{\tau_3}", shift left, from=1-2, to=2-3]
			\arrow["{\tau_2}", shift left, from=2-1, to=1-2]
			\arrow["\rho"', from=2-1, to=2-3]
			\arrow["{\tau_4}", shift left, from=2-3, to=1-2]
		\end{tikzcd}\]
		and
		$I= \{ \tau_1\tau_2,  \tau_2\tau_3,  \tau_3\tau_4,  \tau_4\tau_1 \}$.
		The products
		$\tau_1\rho$ and $\rho\tau_4$
		lie in $\mB$.  The algebra is finite-dimensional and gentle, and
		$\tau=\tau_1\tau_2\tau_3\tau_4$
		is a primitive relation cycle.
		
		Put
		$\gamma=\tau^2\tau_1\tau_2\in\Gamma_{10}$.
		Thus
		$w=2$ and $r=2$.
		Let
		$y=\tau\tau_1 = \tau_1\tau_2\tau_3\tau_4\tau_1 \in\Gamma_5$,
		and
		$x=\tau\tau_1\tau_2 = \tau_1\tau_2\tau_3\tau_4\tau_1\tau_2 \in\Gamma_6$.
		Hence
		$n=4$ and $m=5$.
		
		Define
		$a=\tau_1\rho\tau_4$ and $b=\tau_1$,
		and
		$f=\para{x}{a}\in B^5$ and $g=\para{y}{b}\in B^4$.
		
		The support $y$ occurs in $\gamma$ at the positions
		$\Occ_\gamma(y)=\{0,4\}$.
		Replacing either occurrence of $y$ by the arrow $b=\tau_1$ produces
		the same complementary support $x$.  Hence
		$\mI_{y,\tau_1}^{x}(\gamma)=\{0,4\}$.
		Since $n+2=6$, the first insertion direction contributes
		\[
		(
		(-1)^{0\cdot6}+(-1)^{4\cdot6}
		)\para{\gamma}{a}
		=
		2\para{\gamma}{a}.
		\]
		
		The support $x$ occurs at the positions $0$ and $4$.  The occurrence
		at $0$ gives no local term in the second insertion direction.  At $j=4=n$, however, $y=\gamma_1\gamma_2\gamma_3\gamma_4\tau_1$ and $a=\tau_1(\rho\tau_4)$, so the right boundary term is $\mD_4^{(4,5)}(y,\tau_1;\gamma;a)=\pi(\tau_1\rho\tau_4)=a$.
		Moreover, $-(-1)^{mn}(-1)^{n(m+2)}=-1$.  Thus the second insertion direction contributes $-\para{\gamma}{a}$.
		Therefore
		\[
		l_2(f\otimes g)
		=
		(2-1)\para{\gamma}{a}
		=
		\para{\gamma}{a}
		\neq0.
		\]
		
		This example shows that an occurrence sum from the first insertion
		direction and a boundary term from the second insertion direction can
		contribute to the same output support. Their contributions are combined with the prescribed signs.
	\end{example}
	%================================================
	\subsubsection{Both values have length greater than one}
	
	Finally, assume that $\ell(a)>1$ and $\ell(b)>1$.
	By \Cref{cor:higher-value-length-reduction}, no interior term can occur.
	
	By
	\Cref{lem:higher-boundary-dichotomy,rem:higher-boundary-cancellation},
	every boundary term arising from the first alternative of
	\Cref{lem:higher-boundary-dichotomy} is paired with a term from the opposite
	insertion direction and cancels.  Consequently, every surviving contribution
	is one of the four idempotent-absorption boundary terms in
	\Cref{cor:higher-long-value-reduction}.
	
	The suffix creates no additional local term.  Its only role is
	to determine whether the required boundary occurrences are present.  Namely,
	the first and second alternatives of
	\Cref{cor:higher-long-value-reduction} can occur only when $0\in\Occ_\gamma(y)$ and $n\in\Occ_\gamma(x)$, respectively, while the third and fourth alternatives can occur only when $m\in\Occ_\gamma(y)$ and $0\in\Occ_\gamma(x)$, respectively.
	
	Thus the possible surviving projected values are $\pi(pa)$, $\pi(bq)$, $\pi(aq)$, and $\pi(pb)$,
	with the signs displayed in
	\Cref{cor:higher-long-value-reduction}.  Terms whose images under $\pi$ vanish
	are omitted.  If several boundary idempotent-absorption configurations occur
	simultaneously, their contributions are added.
	%===========================================
	\begin{example}
		\label{ex:suffix-long-value-absorption}
		Let $Q^S$ be the quiver
		
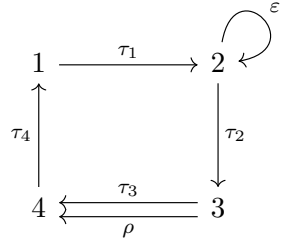
\[
		\begin{tikzcd}
			1 && 2 \\
			\\
			4 && 3
			\arrow["{\tau_1}", from=1-1, to=1-3]
			\arrow["\ep", from=1-3, to=1-3,
			loop, in=10, out=80, distance=10mm]
			\arrow["{\tau_2}", from=1-3, to=3-3]
			\arrow["{\tau_4}", from=3-1, to=1-1]
			\arrow["{\tau_3}"', shift right, from=3-3, to=3-1]
			\arrow["\rho", shift left, from=3-3, to=3-1]
		\end{tikzcd}
		\]
		where $\ep$ is the special loop at the vertex $2$.  Let
		$I=\{\tau_1\tau_2, \tau_2\tau_3, \tau_3\tau_4, \tau_4\tau_1\}$ and
		$I^S=I\cup\{\ep^2\}$.
		Then
		$A= \mK Q^S/ \langle I\cup\{\ep^2-\ep\}\rangle$
		is a finite-dimensional skew-gentle algebra.  In particular, the products
		$\tau_1\ep$, $\ep\tau_2$, $\tau_2\rho$, and $\rho\tau_4$ lie in $\mB$.
		
		Put
		$\tau=\tau_1\tau_2\tau_3\tau_4$
		and
		$\gamma=\tau^2\tau_1\tau_2\in\Gamma_{10}$.
		Thus
		$w=2$ and $r=2$.
		Set
		$y=\tau_1\tau_2\in\Gamma_2$
		and
		$x= \tau_2\tau_3\tau_4\tau_1 \tau_2\tau_3\tau_4\tau_1\tau_2 \in\Gamma_9$.
		Hence
		$n=1$ and $m=8$.
		
		Define
		$p=\tau_1\ep$, $b=p\tau_2=\tau_1\ep\tau_2$, and $a=\ep\tau_2$,
		and
		$f=\para{x}{a}\in B^8$ and $g=\para{y}{b}\in B^1$.
		Both values have length greater than one.
		
		The support $y$ occurs at the positions
		$\Occ_\gamma(y)=\{0,4,8\}$.
		At $i=0$, the occurrence is a left boundary occurrence and
		$b=p\tau_2$, where $p_{\ell(p)}=\ep=\alpha_1$.
		Therefore
		\[\mD_0^{(8,1)}(x,a;\gamma;b)=\pi(pa)=\pi(\tau_1\ep\ep\tau_2)=\tau_1\ep\tau_2.\]
		This is a special-loop idempotent absorption.
		
		The occurrence at $i=4$ is interior, and its local term is zero because
		$\ell(b)>1$.
		Although $i=8=m$ is a right boundary occurrence, its complementary
		support does not equal $x$, so its local term is also zero.
		
		The support $x$ occurs only at $j=1=n$ in the second insertion
		direction.  A nonzero right boundary term there would require $a$ to
		start with the overlap arrow $\tau_2$, but
		$\alpha_1=\ep\neq\tau_2$.
		Hence the second insertion sum vanishes.  Consequently,
		\[
		l_2(f\otimes g)
		=
		\para{\gamma}{\tau_1\ep\tau_2}
		\neq0.
		\]
		
		This example illustrates that, for two values of length greater than one,
		a boundary term arising from the first alternative of
		\Cref{lem:higher-boundary-dichotomy} cancels with the opposite insertion
		direction, whereas a special idempotent square can leave an uncancelled
		contribution.
	\end{example}
	%================================================
	\begin{proposition}
		\label{prop:higher-primitive-cycle-suffix}
		Let $\gamma=\tau^w\tau_1\cdots\tau_r$, where
		$1\leq r<\ell(\tau)$ and $w\geq1$, and let
		$f=\para{x}{a}\in B^m$ and $g=\para{y}{b}\in B^n$, where
		$m,n\geq1$.  Then every nonzero value of $l_2(f\otimes g)(\gamma)$ belongs to exactly one of the following mutually exclusive cases.
		
		\begin{enumerate}
			\item At least one of $a,b$ is trivial.  The bracket is given by
			\Cref{cor:higher-trivial-value,lem:periodic-occurrences}
			and skew-symmetry.
			
			\item Neither value is trivial and at least one of $a,b$ is an
			arrow.  The bracket is given by
			\eqref{eq:suffix-two-arrow-values} or
			\eqref{eq:suffix-arrow-long-value}, together with skew-symmetry.
			
			\item $\ell(a)>1$ and $\ell(b)>1$.  The only possible surviving terms are the idempotent-absorption
			boundary terms in
			\Cref{cor:higher-long-value-reduction}.
		\end{enumerate}
		
		The occurrence positions in all three cases are determined by
		\Cref{lem:periodic-occurrences}.  The corresponding signed multiplicities
		are evaluated by \Cref{cor:signed-periodic-multiplicity} whenever the
		relevant set is the full occurrence set.
		
		Conversely, every term described in these three cases occurs whenever
		its support, occurrence, and nonvanishing conditions under $\pi$ are satisfied.  The
		resulting bracket is nonzero precisely when the sum of all surviving terms
		is nonzero in $A$.
	\end{proposition}
	
	\begin{proof}
		The three cases form a partition of the possible lengths of the two
		value paths $a,b\in\mB$.
		
		The first case follows from
		\Cref{cor:higher-trivial-value}.  The second case follows from
		\Cref{cor:higher-arrow-reassembly,cor:periodic-read-off}.  The third case
		follows from \Cref{cor:higher-long-value-reduction}.
		
		By \Cref{lem:periodic-occurrences}, every occurrence of either support is
		included in the displayed sums, including occurrences which fit into the
		final suffix and excluding those which do not.  Hence the distinction
		between the cyclic starting positions and the relative position of the
		suffix is already encoded by $\Occ_\gamma(x)$ and $\Occ_\gamma(y)$.
		The local terms in all three cases exhaust the alternatives
		in \Cref{HOBF}.  No further contribution can occur.
	\end{proof}
	
	\begin{remark}
		The possible cyclic relative positions of the supports $x$ and $y$ are
		recovered by fixing their first arrows and lengths.  They introduce no
		additional local term beyond those in \Cref{HOBF}; their
		occurrence numbers and boundary contacts are already encoded by $\Occ_\gamma(x)$ and $\Occ_\gamma(y)$.
	\end{remark}
	%============================
	%================================================

	\begin{proof}[Proof of \Cref{prop:positive-positive-classification}]
		By \Cref{HOBF}, every contribution to $l_2(f\otimes g)(\gamma)$ arises
		from an occurrence of $x$ or $y$ in $\gamma$, and its local value is one of
		the terms in the definition of $\mD_i^{(m,n)}$.  By \Cref{gamma-cases},
		$\gamma$ either has no repeated arrows or has exactly one of the three
		periodic forms listed in the proposition.  These four cases are treated,
		respectively, by
		\Cref{NRA-read-off,prop:higher-pure-loop-read-off,prop:higher-primitive-cycle-power,prop:higher-primitive-cycle-suffix}.
		The periodic occurrence positions and their signed multiplicities are
		accounted for by
		\Cref{lem:periodic-occurrences,cor:signed-periodic-multiplicity}.  Thus the
		four cases provide a finite procedure for evaluating
		$l_2(f\otimes g)(\gamma)$ for every $\gamma\in\Gamma_{m+n+1}$.
	\end{proof}
	
	A compact tabular summary for brackets of type $(+,+)$ is given in
	\Cref{app:positive-positive-lookup}.
	
	Together with the classifications in Sections~4--6,
	\Cref{prop:positive-positive-classification} completes the proof of
	Theorem~B.
	
	%=============================================
	\section{Higher operations and the formality classification problem}
	\label{sec:higher-operations-formality}
	
	This section concerns higher transferred operations and Hochschild
	$L_\infty$-formality. The Hochschild dg Lie algebras of $A_n$-type
	skew-gentle algebras are shown to be homotopy abelian, while a
	special vertex that is not an endpoint gives rise to a nonzero
	ternary operation on the transferred CS model. Together with
	Wong's non-formal gentle example, these results motivate a
	combinatorial classification of Hochschild $L_\infty$-formality.
	
	We call a skew-gentle algebra $A$ \emph{$A_n$-type} if deleting the special
	loops leaves an arbitrary orientation of the Dynkin graph $A_n$.  When $n=1$,
	the algebra is either $\mK$ or
	$\mK[\ep]/\langle \ep^2-\ep \rangle \cong\mK\times\mK$; this case will be treated separately
	in the proof of \Cref{thm:An-formality}.
	
\begin{lemma}
	\label{lem:An-Hochschild-cohomology}
	Let $A$ be an $A_n$-type skew-gentle algebra with $n\geq2$.
	Suppose that the ordinary quiver is linearly oriented and every vertex is
	special. After relabelling the vertices if necessary, write
	\[
	1\xrightarrow{\alpha_1}2\xrightarrow{\alpha_2}\cdots
	\xrightarrow{\alpha_{n-1}}n
	\]
	and let $\ep_i$ be the special loop at the vertex $i$. Then
	\[
	\HH^q(A)\cong
	\begin{cases}
		\mK, & \text{if }q=0,\\
		\mK, &\text{if }q=n-1,\\
		0, & \text{otherwise}.
	\end{cases}
	\]
	A basis element of $\HH^{n-1}(A)$ can be chosen as
	\[
	\zeta
	=
	\left[
	\para{\alpha_1\cdots\alpha_{n-1}}
	{\ep_1\alpha_1\ep_2\alpha_2\cdots
		\ep_{n-1}\alpha_{n-1}\ep_n}
	\right].
	\]
	
	For every other $A_n$-type skew-gentle algebra,
	$\HH^*(A)$ is one-dimensional and concentrated in degree zero; namely,
	\[
	\HH^0(A)\cong\mK,
	\qquad
	\HH^q(A)=0
	\quad\text{for every }q\geq1.
	\]
\end{lemma}
	
	\begin{proof}
		This is a direct specialization of the Hochschild basis theorem
		\cite[Theorem~3.15]{BSSWW26}. Since the ordinary underlying graph is the
		tree $A_n$, all families in that theorem arising from cyclic configurations
		are empty. Specializing the remaining families to a path graph shows that a
		positive-degree class can occur precisely when the ordinary quiver is
		linearly oriented and every vertex is special. 
		In that case, $\HH^{n-1}(A)$ is the only nonzero positive-degree
		cohomology group; it is one-dimensional and is spanned by the
		class $\zeta$ displayed above.
		In all other cases no positive-degree basis element survives, while in
		degree zero only the unit class remains. Hence the stated description of
		$\HH^*(A)$ follows.
	\end{proof}

	\begin{theorem}
		\label{thm:An-formality}
		Let $A$ be an $A_n$-type skew-gentle algebra. Then the Hochschild dg Lie
		algebra $C^*(A)[1]$ is homotopy abelian. In particular, every $A_n$-type
		skew-gentle algebra is Hochschild $L_\infty$-formal.
	\end{theorem}
	
	\begin{proof}
		Suppose first that $n=1$.  Then $A\cong\mK$ or
		$A\cong\mK\times\mK$, so $A$ is separable.  Hochschild's vanishing theorem
		gives $\HH^q(A)=0$ for every $q>0$; see
		\cite[Theorem~4.1]{Hoc45}.  Hence $H^*(C^*(A)[1])$ is concentrated in shifted
		degree $-1$, and $C^*(A)[1]$ is formal by
		\cite[Example~6.2.3]{Man22}.  Its cohomology graded Lie algebra is abelian,
		since the bracket of two degree-$-1$ classes would lie in degree $-2$, where
		the cohomology vanishes.  It follows that $C^*(A)[1]$ is homotopy abelian;
		see \cite[Definitions~6.2.1--6.2.2]{Man22}.
		
		Assume now that $n\geq2$ and put $\mathcal H=\HH^*(A)[1]$.  By
		\Cref{lem:An-Hochschild-cohomology}, $\mathcal H$ is either generated by
		the unit class $[1_A]$ alone, or, in the exceptional case, by
		$[1_A]$ together with the class $\zeta$.
		
		Choose graded decompositions
		\[
		\ker d_C=\operatorname{im}d_C\oplus H',
		\qquad
		C^*(A)[1]=\ker d_C\oplus W,
		\]
		with $1_A\in H'$, and identify $H'$ with $\mathcal H$.  The restriction
		$d_C|_W:W\to\operatorname{im}d_C$ is an isomorphism.  The associated
		inclusion and projection, together with the negative inverse of $d_C|_W$ on
		$\operatorname{im}d_C$ and zero on $H'\oplus W$, give contraction data in
		our sign convention.  In particular, the inclusion sends $[1_A]$ to
		$1_A$, and the contracting homotopy vanishes on $1_A$. Let
		$\{l_r^{\min}\}_{r\geq1}$ denote the transferred minimal
		$L_\infty$-structure on $\mathcal H$ associated with this contraction.
		
		Because normalized insertions first project the inserted value by $\varpi$
		and $\varpi(1_A)=0$, the unit $0$-cochain is central for the normalized
		Gerstenhaber bracket: $[1_A,\varphi]=0$ for every normalized Hochschild
		cochain $\varphi$. 
		A simultaneous induction on $r\geq2$ using
		\eqref{eq:vn}--\eqref{eq:phin} shows that both $l_r^{\min}$ and
		$\phi_r$ vanish whenever one input is $[1_A]$.
		Indeed, the induction hypothesis and the centrality of $1_A$
		first give $v_r=0$, and hence $l_r^{\min}=0$, on such inputs.
		In $u_r$, the term with $k=r$ is $-\phi_1l_r^{\min}$ and
		therefore vanishes by this conclusion; the terms with $k<r$
		vanish by induction. Thus $u_r=0$, and consequently $\phi_r=0$.
		
		If the exceptional class $\zeta$ does not occur, this already shows that all
		minimal operations of arity at least two vanish.  Suppose now that $\zeta$
		occurs.  In the shifted grading, $|[1_A]|=-1$ and $|\zeta|=n-2$.
		It remains only to consider
		$l_r^{\min}(\zeta^{\otimes r})$ with $r\geq2$.
		If $n$ is even, then $|\zeta|$ is even, and graded skew-symmetry over the
		characteristic-zero field $\mK$ gives $l_r^{\min}(\zeta^{\otimes r})=0$.
		If $n$ is odd, then an $r$-ary operation has output degree $r(n-2)+(2-r)=r(n-3)+2$.
		By \Cref{lem:An-Hochschild-cohomology}, the graded vector space $\mathcal H$
		is supported only in degrees $-1$ and $n-2$, whereas, for every $r\geq2$,
		$r(n-3)+2\notin\{-1,n-2\}$.
		Thus $l_r^{\min}(\zeta^{\otimes r})=0$ also in this case.
		
		Therefore every transferred minimal operation of arity at least two
		vanishes. Hence $C^*(A)[1]$ is $L_\infty$ quasi-isomorphic to
		$\mathcal H=\HH^*(A)[1]$ with all operations zero, so it is homotopy
		abelian. In particular, $A$ is Hochschild $L_\infty$-formal.
	\end{proof}
	
	The preceding theorem concerns the Hochschild minimal model.  It does not
	imply that the transferred CS model has vanishing higher operations.
	
	\begin{theorem}
		\label{thm:An-interior-special}
		Let $A$ be an $A_n$-type skew-gentle algebra. If a special vertex $j$
		is not an endpoint, then the transferred operation $l_3$ is nonzero,
		and the graded Jacobi relation for $l_2$ is governed by a nontrivial
		ternary homotopy.
		In particular, if
		$(B_{\mathrm{CS}}^*(A)[1],l_1,l_2)$
		is a dg Lie algebra, then every special vertex is an endpoint.
	\end{theorem}
	
	\begin{proof}
		Since $j$ is not an endpoint, precisely two ordinary arrows are incident
		with it.  Together with the special loop $\ep_j:j\to j$, condition
		\textup{(G2)} rules out both ordinary arrows entering $j$ and also rules out
		both leaving $j$.  Hence, after relabelling the two neighbours, the local
		orientation is $i\xrightarrow{\alpha}j\xrightarrow{\beta}k$.  Since
		$\ep_j^2\in I^S$, condition \textup{(G3)} gives
		$\alpha\ep_j\notin I^S$ and $\ep_j\beta\notin I^S$. Since
		$\alpha\ep_j\notin I^S$ and $\beta\neq\ep_j$, condition
		\textup{(G4)} forces $\alpha\beta\in I^S$, while
		$\alpha\ep_j\beta\in\mB$.
		Set $f_j=\para{\alpha}{\alpha\ep_j}$,
		$g_j=\para{\ep_j}{\ep_j}$, and
		$h_j=\para{\alpha\beta}{\alpha\ep_j\beta}$.
		The formulas in \Cref{DZSF,MSBF} give $l_2(f_j\otimes g_j)=-f_j$,
		$l_2(g_j\otimes h_j)=h_j$, and $l_2(f_j\otimes h_j)=h_j$.
		Indeed, the last equality uses
		$\pi(\alpha\ep_j\ep_j\beta)=\alpha\ep_j\beta$; moreover,
		$L_\alpha(\alpha\ep_j)=0$, so
		$\mC_\alpha^{\alpha\ep_j}(\alpha\beta;\alpha\ep_j\beta)=0$.  Since $|f_j|=|g_j|=0$ and $|h_j|=1$, the graded Jacobiator of $l_2$ evaluated on
		$f_j\otimes g_j\otimes h_j$ is
		\[
		\begin{aligned}
			l_2(l_2(f_j\otimes g_j)\otimes h_j)
			+l_2(l_2(g_j\otimes h_j)\otimes f_j)
			+l_2(l_2(h_j\otimes f_j)\otimes g_j)=-h_j-h_j+h_j
			=-h_j\neq0.
		\end{aligned}
		\]
		Thus $(B_{\mathrm{CS}}^*(A)[1],l_1,l_2)$ is not a dg Lie algebra.  If the transferred
		ternary operation were identically zero, the arity-three $L_\infty$ identity
		would reduce to the strict graded Jacobi identity for $l_2$.  Hence
		$l_3\neq0$.
	\end{proof}
	
\begin{remark}
	\Cref{thm:An-formality,thm:An-interior-special} show that homotopy
	abelianity of the Hochschild dg Lie algebra is compatible with
	nonzero higher operations on the transferred CS complex. When a
	special vertex is not an endpoint, the ternary operation provides
	a nontrivial homotopy for the graded Jacobi relation.
\end{remark}
%==========================
	
	The formality result above does not extend to all finite-dimensional gentle
	algebras.  Let $D=\mK[\theta]/\langle \theta^2 \rangle$, equivalently the path algebra of the
	one-loop quiver modulo the relation $\theta^2$.  
	
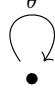
\[\begin{tikzcd}
		\bullet
		\arrow["\theta", from=1-1, to=1-1, loop, in=55, out=125, distance=10mm]
	\end{tikzcd}\]
	This is a finite-dimensional
	gentle algebra, hence also a skew-gentle algebra with empty special set, and as
	an ungraded associative algebra it is the exterior algebra on a
	one-dimensional vector space.  Wong proved that its Hochschild dg Lie algebra
	is non-formal \cite[Theorem~6.3]{Won17}.  Hence $D$ is not Hochschild $L_\infty$-formal.
	Here ``formal'' refers to the Hochschild dg Lie algebra, not to $D$ viewed as a
	dg associative algebra with zero differential.
	
	Thus the finite-dimensional gentle/skew-gentle setting exhibits both formal
	and non-formal behaviour. This leads to the following classification problem.

\medskip
\noindent
\textbf{Question.}
Are there necessary and sufficient combinatorial conditions on a triple
$(Q,I,S)$ ensuring that the skew-gentle algebra $A$ is Hochschild
$L_\infty$-formal?
	
	%=======================================
	\clearpage
	
	\appendix
	%=========================================================
	\providecommand{\FigCell}[1]{%
		\begin{minipage}[t]{\linewidth}
			\vspace{0pt}
			\centering
			#1
	\end{minipage}}
	
	\providecommand{\CondCell}[1]{%
		\begin{minipage}[t]{\linewidth}
			\vspace{0pt}
			\raggedright
			#1
	\end{minipage}}
	
	\providecommand{\ValCell}[1]{%
		\begin{minipage}[t]{\linewidth}
			\vspace{0pt}
			\raggedright
			#1
	\end{minipage}}
	
	\providecommand{\Stack}[1]{%
		$\begin{gathered}#1\end{gathered}$}
	
	\providecommand{\FitMath}[1]{%
		\resizebox{\linewidth}{!}{$\displaystyle #1$}}
	
	%=========================================================
	The following appendices summarize the bracket classifications obtained in
	\Cref{sec:degree-minus-one-classification,sec:degree-zero-classification,sec:mixed-classification,sec:positive-positive-classification}.
	Throughout, brackets are recorded in the order $l_2(f\otimes g)$; brackets
	with the two inputs exchanged are recovered from graded skew-symmetry
	\eqref{sym} and are not listed separately.
	
	\section{Brackets of type \texorpdfstring{$(-1,n)$}{(-1,n)}}
	\label[appendix]{app:degree-minus-one-lookup}
	%=========================================================
	
	Let $f=\para{e_i}{a}\in B^{-1}$ and $g=\para{y}{b}\in B^n$, where $n\geq-1$.
	For the rows with $\ell(a)>1$, write
	$a=\alpha_1\cdots\alpha_{\ell(a)}$ for the cycle in $\mB$ based at $i$.
	
	\begin{center}
		\footnotesize
		\setlength{\tabcolsep}{4pt}
		\renewcommand{\arraystretch}{1.35}
		\begin{longtable}{
				@{}
				>{\centering\arraybackslash}m{0.19\textwidth}
				>{\raggedright\arraybackslash}m{0.40\textwidth}
				>{\centering\arraybackslash}m{0.35\textwidth}
				@{}
			}
			\caption{Summary of brackets of type $(-1,n)$.}
			\label{tab:degree-minus-one-lookup}
			\\
			\hline
			\textbf{Local configuration}
			&
			\textbf{Assumptions / nonzero condition}
			&
			$\boldsymbol{l_2(f\otimes g)}$
			\\
			\hline
			\endfirsthead
			
			\hline
			\textbf{Local configuration}
			&
			\textbf{Assumptions / nonzero condition}
			&
			$\boldsymbol{l_2(f\otimes g)}$
			\\
			\hline
			\endhead
			
			\hline
			\endfoot

			%========================================================
			&
			Either $n=-1$, or $n\geq0$ and $a=e_i$.\newline
			No additional condition.
			&
			$\displaystyle 0$
			\\[2mm]
			\hline
			
			%========================================================
			\begin{minipage}[c]{\linewidth}
				\centering
				\begin{tikzcd}
					i
					\arrow["\omega", from=1-1, to=1-1, loop, in=55, out=125, distance=10mm]
				\end{tikzcd}
			\end{minipage}
			&
			$a=\omega$, where $\omega$ is a loop at $i$.\newline
			$n\geq0$ and $y=\omega^{n+1}$.\newline
			Nonzero precisely when $n$ is even.
			&
			$\displaystyle
			\begin{cases}
				-\para{\omega^n}{b}, & \text{if } n\text{ is even},\\
				0, & \text{if } n\text{ is odd}.
			\end{cases}$
			\\[2mm]
			\hline
			
			%========================================================
			\begin{minipage}[c]{\linewidth}
				\centering
				\begin{tikzcd}
					i
					\arrow["a", dotted, from=1-1, to=1-1, loop, in=55, out=125, distance=10mm]
				\end{tikzcd}
			\end{minipage}
			&
			$\ell(a)>1$, $n=0$, and $g=\para{y}{b}\in B^0$.\newline
			Nonzero precisely when $\Sub_y^b(a)\neq0$.\newline
			See \Cref{lem:replacement-criterion} for the possible nonzero substitutions.
			&
			$\displaystyle -\para{e_i}{\Sub_y^b(a)}$
			\\[2mm]
			\hline
			
			%========================================================
			\begin{minipage}[c]{\linewidth}
				\centering
				\begin{tikzcd}
					i
					\arrow["a", dotted, from=1-1, to=1-1, loop, in=55, out=125, distance=10mm]
				\end{tikzcd}
			\end{minipage}
			&
			$\ell(a)>1$, $n\geq1$,\newline
			$y=y_1\cdots y_{n+1}$ and $y_1=\alpha_{\ell(a)}$.\newline
			Retained precisely when $\pi(L_{\alpha_{\ell(a)}}(a)b)\neq0$.
			&
			$\displaystyle
			(-1)^{n+1}
			\para{y_2\cdots y_{n+1}}
			{\pi(L_{\alpha_{\ell(a)}}(a)b)}$
			\\[2mm]
			\hline
			
			%========================================================
			\begin{minipage}[c]{\linewidth}
				\centering
				\begin{tikzcd}
					i
					\arrow["a", dotted, from=1-1, to=1-1, loop, in=55, out=125, distance=10mm]
				\end{tikzcd}
			\end{minipage}
			&
			$\ell(a)>1$, $n\geq1$,\newline
			$y=y_1\cdots y_{n+1}$ and $y_{n+1}=\alpha_1$.\newline
			Retained precisely when $\pi(bR_{\alpha_1}(a))\neq0$.
			&
			$\displaystyle -\para{y_1\cdots y_n}{\pi(bR_{\alpha_1}(a))}$
			\\[2mm]
			\hline
		\end{longtable}
	\end{center}
	
	If both conditions in the last two rows hold, then
	\[
	\begin{aligned}
		l_2(f\otimes g)
		=
		(-1)^{n+1}
		\para{y_2\cdots y_{n+1}}
		{\pi(L_{\alpha_{\ell(a)}}(a)b)}
		-
		\para{y_1\cdots y_n}
		{\pi(bR_{\alpha_1}(a))}.
	\end{aligned}
	\]
	The two output supports are distinct, so the two terms cannot cancel.
	\Cref{tab:degree-minus-one-lookup} exhausts all possible nonzero brackets
	involving $B^{-1}$ under the convention stated above.
	%===========================
	\section{Brackets of type \texorpdfstring{$(0,0)$}{(0,0)}}
	\label[appendix]{app:degree-zero-lookup}
	%=========================================================
	Let $f=\para{x}{a}\in B^0$ and $g=\para{y}{b}\in B^0$ be basis elements.
	The table below summarizes the nonzero brackets of type $(0,0)$ classified in
	\Cref{sec:degree-zero-classification}.
	
	\begin{center}
		\footnotesize
		\setlength{\tabcolsep}{3pt}
		\renewcommand{\arraystretch}{1.45}
		\begin{longtable}{
				@{}
				>{\centering\arraybackslash}m{0.31\textwidth}
				>{\raggedright\arraybackslash}m{0.40\textwidth}
				>{\centering\arraybackslash}m{0.23\textwidth}
				@{}
			}
			\caption{Summary of brackets of type $(0,0)$.}
			\label{tab:degree-zero-nonzero-lookup}
			\\
			\hline
			\textbf{Local configuration}
			&
			\textbf{Assumptions / nonzero condition}
			&
			$\boldsymbol{l_2(f\otimes g)}$
			\\
			\hline
			\endfirsthead
			
			\hline
			\textbf{Local configuration}
			&
			\textbf{Assumptions / nonzero condition}
			&
			$\boldsymbol{l_2(f\otimes g)}$
			\\
			\hline
			\endhead
			
			\hline
			\endfoot
			
			%========================================================
			\begin{minipage}[c]{\linewidth}
				\centering
				\begin{tikzcd}[ampersand replacement=\&, column sep=5em]
					\bullet
					\arrow[r,dotted,shift left=0.5ex,"y"]
					\arrow[r,dotted,shift right=0.5ex,swap,"b"]
					\&
					\bullet
				\end{tikzcd}
			\end{minipage}
			&
			$a=x$.\newline
			Nonzero precisely in one of the following two cases:\newline
			\textup{(i)} $x=y$ and $\num{x}{b}=0$;\newline
			\textup{(ii)} $x\neq y$ and $\num{x}{b}=1$.\newline
			See \Cref{00-x-equal-to-a}.
			&
			$\displaystyle
			\begin{cases}
				-\para{y}{b}, & \text{if (i) holds},\\
				\para{y}{b}, & \text{if (ii) holds}.
			\end{cases}$
			\\[3mm]
			\hline
			
			%========================================================
			\begin{minipage}[c]{\linewidth}
				\centering
				\begin{tikzcd}[ampersand replacement=\&, column sep=small]
					\bullet \& \& \bullet
					\arrow["{x=\beta_1}", from=1-1, to=1-1, loop, in=145, out=215, distance=10mm]
					\arrow["{\beta_2}"', from=1-1, to=1-3]
					\arrow["q", dotted, from=1-3, to=1-3, loop, in=325, out=35, distance=10mm]
				\end{tikzcd}
			\end{minipage}
			&
			$x=\beta_1$ is a loop and $a=e_{s(x)}$.\newline
			$y=\beta_2$ and $b=x\beta_2q$ with $\ell(b)>1$.\newline
			$q$ is a path based at $t(\beta_2)$, possibly trivial.
			&
			$\displaystyle \para{\beta_2}{\beta_2q}$
			\\[3mm]
			\hline
			
			%========================================================
			\begin{minipage}[c]{\linewidth}
				\centering
				\begin{tikzcd}[ampersand replacement=\&, column sep=small]
					\bullet \& \& \bullet
					\arrow["p", dotted, from=1-1, to=1-1, loop, in=145, out=215, distance=10mm]
					\arrow["{\beta_{\ell(b)-1}}"', from=1-1, to=1-3]
					\arrow["{x=\bet}", from=1-3, to=1-3, loop, in=325, out=35, distance=10mm]
				\end{tikzcd}
			\end{minipage}
			&
			$x=\bet$ is a loop and $a=e_{s(x)}$.\newline
			$y=\beta_{\ell(b)-1}$ and $b=p\beta_{\ell(b)-1}x$ with $\ell(b)>1$.\newline
			$p$ is a path based at $s(\beta_{\ell(b)-1})$, possibly trivial.
			&
			$\displaystyle \para{\beta_{\ell(b)-1}}{p\beta_{\ell(b)-1}}$
			\\[3mm]
			\hline
			
			%========================================================
			\begin{minipage}[c]{\linewidth}
				\centering
				\begin{tikzcd}[ampersand replacement=\&, column sep=5em]
					\bullet
					\arrow[r,shift left=0.5ex,"x"]
					\arrow[r,shift right=0.5ex,swap,"y"]
					\&
					\bullet
				\end{tikzcd}
			\end{minipage}
			&
			$x\neq y$ are parallel arrows.\newline
			$a=y$ and $b=x$.
			&
			$\displaystyle \para{y}{y}-\para{x}{x}$
			\\[3mm]
			\hline
			
			%========================================================
			\begin{minipage}[c]{\linewidth}
				\centering
				\begin{tikzcd}[ampersand replacement=\&, column sep=5em]
					\bullet
					\arrow[r,shift left=0.5ex,"x"]
					\arrow[r,shift right=0.5ex,swap,"y"]
					\arrow[r,bend left=55,dotted,"a"]
					\&
					\bullet
				\end{tikzcd}
			\end{minipage}
			&
			$x\neq y$ and $\ell(a)>1$.\newline
			$b=x$.
			&
			$\displaystyle \para{y}{a}$
			\\[3mm]
			\hline
			
			%========================================================
			\begin{minipage}[c]{\linewidth}
				\centering
				\begin{tikzcd}[ampersand replacement=\&, column sep=3em]
					\bullet \& \bullet \& \cdots
					\arrow["p", dotted, from=1-1, to=1-1, loop, in=145, out=215, distance=10mm]
					\arrow["{\beta_1}"', from=1-1, to=1-2]
					\arrow["\beta_2", from=1-2, to=1-2, loop, in=55, out=125, distance=10mm]
					\arrow[from=1-2, to=1-3]
				\end{tikzcd}
			\end{minipage}
			&
			$x=\beta_1\neq y$ and $a=p\beta_1\mu$.\newline
			$\ell(b)>1$.\newline
			$p$ is a nontrivial path based at $s(\beta_1)$, with $pb\neq0$.\newline
			$\mu=e_{t(\beta_1)}$, or $\mu=\beta_2$ with $\beta_2\in\Sp$.
			&
			$\displaystyle \para{y}{pb}$
			\\[3mm]
			\hline
			
			%========================================================
			\begin{minipage}[c]{\linewidth}
				\centering
				\begin{tikzcd}[ampersand replacement=\&, column sep=3em]
					\cdots \& \bullet \& \bullet
					\arrow[from=1-1, to=1-2]
					\arrow["{\beta_{\ell(b)-1}}", from=1-2, to=1-2, loop, in=55, out=125, distance=10mm]
					\arrow["\bet"', from=1-2, to=1-3]
					\arrow["q", dotted, from=1-3, to=1-3, loop, in=325, out=35, distance=10mm]
				\end{tikzcd}
			\end{minipage}
			&
			$x=\bet\neq y$ and $a=\lambda\bet q$.\newline
			$\ell(b)>1$.\newline
			$q$ is a nontrivial path based at $t(\bet)$, with $bq\neq0$.\newline
			$\lambda=e_{s(\bet)}$, or $\lambda=\beta_{\ell(b)-1}$ with
			$\beta_{\ell(b)-1}\in\Sp$.
			&
			$\displaystyle \para{y}{bq}$
			\\[3mm]
			\hline
			
			%========================================================
			\begin{minipage}[c]{\linewidth}
				\centering
				\begin{tikzcd}[ampersand replacement=\&, column sep=2em]
					\bullet \&\& \bullet \\
					\\
					\bullet \&\& \bullet
					\arrow["{\beta_{i_0-1}}", from=1-1, to=1-1, loop, in=100, out=170, distance=10mm]
					\arrow["{x=\bi}", from=1-1, to=1-3]
					\arrow["{\beta_{i_0+1}}", from=1-3, to=1-3, loop, in=10, out=80, distance=10mm]
					\arrow[dotted, from=1-3, to=3-3]
					\arrow[dotted, from=3-1, to=1-1]
					\arrow["y"', from=3-1, to=3-3]
				\end{tikzcd}
			\end{minipage}
			&
			$x=\bi\neq y$.  Nonzero precisely in one of the following cases:\newline
			\textup{(i)} $a=\bi\beta_{i_0+1}$, with $1\leq i_0<\ell(b)-1$;\newline
			\textup{(ii)} $a=\beta_{i_0-1}\bi$, with $2<i_0\leq\ell(b)$;\newline
			\textup{(iii)} $a=\beta_{i_0-1}\bi\beta_{i_0+1}$, with
			$2<i_0<\ell(b)-1$.\newline
			Here $\beta_{i_0-1}$ and/or $\beta_{i_0+1}$ are the adjacent special loops.
			&
			$\displaystyle \para{y}{b}$
			\\[3mm]
			\hline
		\end{longtable}
	\end{center}
	
	\Cref{tab:degree-zero-nonzero-lookup} exhausts all possible nonzero brackets
	of type $(0,0)$ under the convention stated at the beginning of the appendices.
	%=========================================================
	%=========================================================
	\section{Brackets of type \texorpdfstring{$(0,+)$}{(0,+)}}
	\label[appendix]{app:mixed-lookup}
	%=========================================================
	Let $f=\para{x}{a}\in B^0$ and $g=\para{y}{b}\in B^n$, where $n\geq1$, be
	basis elements.  When needed, write $y=\gamma_1\cdots\gamma_{n+1}$.
	The table below summarizes the nonzero brackets classified in
	\Cref{sec:mixed-classification}.  The four possible forms of $y$ in
	\Cref{gamma-cases} produce no additional representatives beyond the rows
	listed here.  All displayed products are assumed to be nonzero paths in
	$\mB$ unless stated otherwise.
	
	\begin{center}
		\footnotesize
		\setlength{\tabcolsep}{3pt}
		\renewcommand{\arraystretch}{1.45}
		\begin{longtable}{
				@{}
				>{\centering\arraybackslash}m{0.30\textwidth}
				>{\raggedright\arraybackslash}m{0.38\textwidth}
				>{\centering\arraybackslash}m{0.26\textwidth}
				@{}
			}
			\caption{Summary of brackets of type $(0,+)$.}
			\label{tab:mixed-nonzero-lookup}
			\\
			\hline
			\textbf{Local configuration}
			&
			\textbf{Assumptions / nonzero condition}
			&
			$\boldsymbol{l_2(f\otimes g)}$
			\\
			\hline
			\endfirsthead
			
			\hline
			\textbf{Local configuration}
			&
			\textbf{Assumptions / nonzero condition}
			&
			$\boldsymbol{l_2(f\otimes g)}$
			\\
			\hline
			\endhead
			
			\hline
			\endfoot
			
			%========================================================
			\begin{minipage}[c]{\linewidth}
				\centering
				\begin{tikzcd}[ampersand replacement=\&, column sep=5em]
					\bullet
					\arrow[r,dotted,shift left=0.5ex,"y"]
					\arrow[r,dotted,shift right=0.5ex,swap,"b"]
					\&
					\bullet
				\end{tikzcd}
			\end{minipage}
			&
			$a=x$.\newline
			Nonzero precisely when $\num{x}{b}-\num{x}{y}\neq0$.\newline
			See \Cref{01-x-equal-to-a}.
			&
			$\displaystyle
			(\num{x}{b}-\num{x}{y})\para{y}{b}$
			\\[3mm]
			\hline
			
			%========================================================
			\begin{minipage}[c]{\linewidth}
				\centering
				\begin{tikzcd}[ampersand replacement=\&, column sep=small]
					\bullet
					\arrow["{b=x}", from=1-1, to=1-1, loop, in=145, out=215, distance=10mm]
					\arrow["y", dotted, from=1-1, to=1-1, loop, in=325, out=35, distance=10mm]
				\end{tikzcd}
			\end{minipage}
			&
			$a=e_{s(x)}$ and $b=x$; hence $x$ is a loop.\newline
			See \Cref{cor:mixed-loop-deletion}.
			&
			$\displaystyle \para{y}{e_{s(x)}}$
			\\[3mm]
			\hline
			
			%========================================================
			\begin{minipage}[c]{\linewidth}
				\centering
				\begin{tikzcd}[ampersand replacement=\&, column sep=small]
					\bullet \& \& \cdots
					\arrow["{x=\beta_1}", from=1-1, to=1-1, loop, in=145, out=215, distance=10mm]
					\arrow["{\beta_2}", from=1-1, to=1-3]
				\end{tikzcd}
			\end{minipage}
			&
			$a=e_{s(x)}$, $x=\beta_1$, and $\ell(b)>1$.\newline
			See \Cref{cor:mixed-loop-deletion}.
			&
			$\displaystyle \para{y}{\beta_2\cdots\beta_{\ell(b)}}$
			\\[3mm]
			\hline
			
			%========================================================
			\begin{minipage}[c]{\linewidth}
				\centering
				\begin{tikzcd}[ampersand replacement=\&, column sep=small]
					\cdots \&\& \bullet
					\arrow["\beta_{\ell(b)-1}", from=1-1, to=1-3]
					\arrow["{\bet=x}", from=1-3, to=1-3, loop, in=325, out=35, distance=10mm]
				\end{tikzcd}
			\end{minipage}
			&
			$a=e_{s(x)}$, $x=\beta_{\ell(b)}$, and $\ell(b)>1$.\newline
			See \Cref{cor:mixed-loop-deletion}.
			&
			$\displaystyle \para{y}{\beta_1\cdots\beta_{\ell(b)-1}}$
			\\[3mm]
			\hline
			
			%========================================================
			\begin{minipage}[c]{\linewidth}
				\centering
				\begin{tikzcd}[ampersand replacement=\&, column sep=5em]
					\bullet
					\arrow[r,shift left=0.5ex,"x=b"]
					\arrow[r,dotted,shift right=0.5ex,swap,"y"]
					\arrow[r,bend left=55,dotted,"a"]
					\&
					\bullet
				\end{tikzcd}
			\end{minipage}
			&
			$b=x$, $x$ is not a loop, $a\neq x$, and $\ell(a)>0$.\newline
			$\mC_x^a(y;x)=0$.\newline
			See \Cref{01-sub-only,rem:substitution-only-conditions}.
			&
			$\displaystyle \para{y}{a}$
			\\[3mm]
			\hline
			
			%========================================================
			\begin{minipage}[c]{\linewidth}
				\centering
				\begin{tikzcd}[ampersand replacement=\&, column sep=large]
					\& \bullet \\
					\bullet \\
					\& \bullet
					\arrow[dotted, from=1-2, to=3-2]
					\arrow["{\gamma_1}", from=2-1, to=1-2]
					\arrow["p", dotted, from=2-1, to=2-1, loop, in=145, out=215, distance=10mm]
					\arrow["{\gamma_{n+1}}", from=3-2, to=2-1]
				\end{tikzcd}
			\end{minipage}
			&
			$x=\gamma_1$ and $a=p\gamma_1$.\newline
			$p$ is a nontrivial left path factor.\newline
			$b=e_{s(y)}$, or $b=p\in\Sp$.
			&
			$\displaystyle -\para{y}{p}$
			\\[3mm]
			\hline
			
			%========================================================
			\begin{minipage}[c]{\linewidth}
				\centering
				\begin{tikzcd}[ampersand replacement=\&, column sep=large]
					\& \bullet \\
					\bullet \\
					\& \bullet
					\arrow[dotted, from=1-2, to=3-2]
					\arrow["{\gamma_1}", from=2-1, to=1-2]
					\arrow["q", dotted, from=2-1, to=2-1, loop, in=145, out=215, distance=10mm]
					\arrow["{\gamma_{n+1}}", from=3-2, to=2-1]
				\end{tikzcd}
			\end{minipage}
			&
			$x=\gamma_{n+1}$ and $a=\gamma_{n+1}q$.\newline
			$q$ is a nontrivial right path factor.\newline
			$b=e_{t(y)}$, or $b=q\in\Sp$.
			&
			$\displaystyle -\para{y}{q}$
			\\[3mm]
			\hline
			
			%========================================================
			\begin{minipage}[c]{\linewidth}
				\centering
				\begin{tikzcd}[ampersand replacement=\&, column sep=2em]
					\bullet \&\& \bullet \\
					\\
					\&\& \bullet
					\arrow["p", dotted, from=1-1, to=1-1, loop, in=100, out=170, distance=10mm]
					\arrow["{\beta_1}", from=1-1, to=1-3]
					\arrow["y"', dotted, from=1-1, to=3-3]
					\arrow["\mu", from=1-3, to=1-3, loop, in=10, out=80, distance=10mm]
					\arrow[dotted, from=1-3, to=3-3]
				\end{tikzcd}
			\end{minipage}
			&
			$x=\beta_1$, $a=p\beta_1\mu$, and $\ell(b)>1$.\newline
			$p$ is nontrivial and $pb\neq0$.\newline
			$\mu=e_{t(\beta_1)}$, or $\mu=\beta_2$ with $\beta_2\in\Sp$.\newline
			$\mC_{\beta_1}^{p\beta_1\mu}(y;b)=0$.\newline
			See \Cref{cor:mixed-endpoint-extension-general}.
			&
			$\displaystyle \para{y}{pb}$
			\\[3mm]
			\hline
			
			%========================================================
			\begin{minipage}[c]{\linewidth}
				\centering
				\begin{tikzcd}[ampersand replacement=\&, column sep=2em]
					\bullet \&\& \bullet \\
					\\
					\bullet
					\arrow["\lambda", dotted, from=1-1, to=1-1, loop, in=100, out=170, distance=10mm]
					\arrow["\bet", from=1-1, to=1-3]
					\arrow["q", from=1-3, to=1-3, loop, in=10, out=80, distance=10mm]
					\arrow[dotted, from=3-1, to=1-1]
					\arrow["y"', dotted, from=3-1, to=1-3]
				\end{tikzcd}
			\end{minipage}
			&
			$x=\beta_{\ell(b)}$ and $a=\lambda\beta_{\ell(b)}q$.\newline
			$\ell(b)>1$, $q$ is nontrivial, and $bq\neq0$.\newline
			$\lambda=e_{s(\beta_{\ell(b)})}$, or
			$\lambda=\beta_{\ell(b)-1}$ with $\beta_{\ell(b)-1}\in\Sp$.\newline
			$\mC_{\beta_{\ell(b)}}^{\lambda\beta_{\ell(b)}q}(y;b)=0$.\newline
			See \Cref{cor:mixed-endpoint-extension-general}.
			&
			$\displaystyle \para{y}{bq}$
			\\[3mm]
			\hline
			
			%========================================================
			\begin{minipage}[c]{\linewidth}
				\centering
				\begin{tikzcd}[ampersand replacement=\&, column sep=2em]
					\bullet \&\& \bullet \\
					\\
					\bullet \&\& \bullet
					\arrow["{\beta_{i_0-1}}", from=1-1, to=1-1, loop, in=100, out=170, distance=10mm]
					\arrow["{x=\bi}", from=1-1, to=1-3]
					\arrow["{\beta_{i_0+1}}", from=1-3, to=1-3, loop, in=10, out=80, distance=10mm]
					\arrow[dotted, from=1-3, to=3-3]
					\arrow[dotted, from=3-1, to=1-1]
					\arrow["y"', dotted, from=3-1, to=3-3]
				\end{tikzcd}
			\end{minipage}
			&
			$x=\bi$, $\ell(b)>1$, and $\mC_x^a(y;b)=0$.\newline
			Nonzero precisely in one of the following cases:\newline
			\textup{(i)} $a=\bi\beta_{i_0+1}$, with $1\leq i_0<\ell(b)-1$;\newline
			\textup{(ii)} $a=\beta_{i_0-1}\bi$, with $2<i_0\leq\ell(b)$;\newline
			\textup{(iii)} $a=\beta_{i_0-1}\bi\beta_{i_0+1}$, with
			$2\leq i_0\leq\ell(b)-1$.\newline
			Here $\beta_{i_0-1}$ and/or $\beta_{i_0+1}$ are the adjacent special loops.
			&
			$\displaystyle \para{y}{b}$
			\\[3mm]
			\hline
		\end{longtable}
	\end{center}
	
	\Cref{tab:mixed-nonzero-lookup} exhausts all possible nonzero brackets of type
	$(0,+)$ under the convention stated at the beginning of the appendices.
	
	%=========================================================
	\section{Brackets of type \texorpdfstring{$(+,+)$}{(+,+)}}
	\label[appendix]{app:positive-positive-lookup}
	
	The formulas for brackets of type $(+,+)$ are most conveniently evaluated
	using the lengths of the two values, the form of the output support, and
	the corresponding local configurations.
	
	Let $f=\para{x}{a}\in B^m$ and $g=\para{y}{b}\in B^n$, where $m,n\geq1$,
	and fix $\gamma=\gamma_1\cdots\gamma_{m+n+1}\in\Gamma_{m+n+1}$.
	The lookup procedure uses the three tables in the following order:
	\[
	\Cref{tab:positive-positive-value-length-lookup}
	\longrightarrow
	\Cref{tab:positive-positive-occurrence-lookup}
	\longrightarrow
	\Cref{tab:positive-positive-local-lookup}
	\longrightarrow
	\Cref{tab:positive-positive-value-length-lookup}.
	\]
	The first table provides both the preliminary length-based filter and the final reduction rules.
	
	\begin{enumerate}[label=\textup{Step\arabic*:},leftmargin=*]
		\item \textbf{Value lengths.}
		Determine whether each of $a$ and $b$ is trivial, an arrow, or has length
		greater than one.  Use
		\Cref{tab:positive-positive-value-length-lookup} to determine which local
		rows and insertion directions can contribute.
		
		\item \textbf{Occurrence positions.}
		Determine the form of $\gamma$ and use
		\Cref{tab:positive-positive-occurrence-lookup} to compute the required
		occurrence sets $\Occ_\gamma(y)$ and/or $\Occ_\gamma(x)$.  For a
		$y$-occurrence, $i=0$, $0<i<m$, and $i=m$ correspond respectively to the
		first-left, first-interior, and first-right rows.  For an $x$-occurrence,
		$j=0$, $0<j<n$, and $j=n$ correspond respectively to the second-left,
		second-interior, and second-right rows.
		
		\item \textbf{Local contributions.}
		For each relevant occurrence, use
		\Cref{tab:positive-positive-local-lookup}.  Record a local contribution
		only when all support and value conditions in the corresponding row are
		satisfied.
		
		\item \textbf{Reduction and summation.}
		Return to \Cref{tab:positive-positive-value-length-lookup}, apply the
		indicated vanishing and cancellation rules, combine equal output pairs,
		and sum the surviving cochain contributions.  The resulting component
		supported on $\gamma$ is then evaluated at $\gamma$ to obtain
		$l_2(f\otimes g)(\gamma)$.  Repeating this for every
		$\gamma\in\Gamma_{m+n+1}$ determines the cochain
		$l_2(f\otimes g)$.
	\end{enumerate}
	
	%=========================================================
	\begin{center}
		\footnotesize
		\setlength{\tabcolsep}{2pt}
		\setlength{\LTcapwidth}{\textwidth}
		\renewcommand{\arraystretch}{1.45}
		\begin{longtable}{
				@{}
				>{\raggedright\arraybackslash}p{0.22\textwidth}
				>{\raggedright\arraybackslash}p{0.37\textwidth}
				>{\raggedright\arraybackslash}p{0.35\textwidth}
				@{}
			}
			\caption{Value-length cases and reduction rules for brackets of type $(+,+)$.}
			\label{tab:positive-positive-value-length-lookup}
			\\
			\hline
			\textbf{Assumptions on $a,b$}
			&
			\textbf{Local rows to use}
			&
			\textbf{Reduction / cancellation rule}
			\\
			\hline
			\endfirsthead
			
			\hline
			\textbf{Assumptions on $a,b$}
			&
			\textbf{Local rows to use}
			&
			\textbf{Reduction / cancellation rule}
			\\
			\hline
			\endhead
			
			\hline
			\endfoot
			
			%========================================================
			\CondCell{Both $a$ and $b$ are trivial.}
			&
			\CondCell{No local row contributes.}
			&
			\ValCell{$l_2(f\otimes g)(\gamma)=0$.}
			\\[4mm]
			\hline
			
			%========================================================
			\CondCell{$a$ is trivial and $b$ is nontrivial.\par\vspace{1mm}} 
			&
			\CondCell{Use only the first three rows of
				Table~\ref{tab:positive-positive-local-lookup}.}
			&
			\ValCell{The entire second insertion direction vanishes.}
			\\[4mm]
			\hline
			
			%========================================================
			\CondCell{$b$ is trivial and $a$ is nontrivial.\par\vspace{1mm}}
			&
			\CondCell{Use only the last three rows of
				Table~\ref{tab:positive-positive-local-lookup}.}
			&
			\ValCell{The entire first insertion direction vanishes.}
			\\[4mm]
			\hline
			
			%========================================================
			\CondCell{Both $a$ and $b$ are arrows.}
			&
			\CondCell{
				Use every applicable row in both insertion directions.
			}
			&
			\ValCell{
				The single-arrow contributions are summed over
				$\mI_{y,b}^{x}(\gamma)$ and
				$\mI_{x,a}^{y}(\gamma)$.  Ordinary matching boundary terms
				cancel by \Cref{rem:higher-boundary-cancellation}; special-loop
				absorption terms may survive.\par\vspace{1mm}
			}
			\\[4mm]
			\hline
			
			%========================================================
			\CondCell{$\ell(a)>1$ and $b$ is an arrow.}
			&
			\CondCell{
				Use all three rows in the first insertion direction and only the
				second-left and second-right boundary rows in the second insertion
				direction.
			}
			&
			\ValCell{
				The first-direction single-arrow contributions are summed over
				$\mI_{y,b}^{x}(\gamma)$.  Ordinary matching boundary terms cancel
				by \Cref{rem:higher-boundary-cancellation}; special-loop absorption
				terms may survive.\par\vspace{1mm}
			}
			\\[4mm]
			\hline
			
			%========================================================
			\CondCell{$a$ is an arrow and $\ell(b)>1$.}
			&
			\CondCell{
				Use only the first-left and first-right boundary rows in the first
				insertion direction and all three rows in the second insertion
				direction.
			}
			&
			\ValCell{
				The second-direction single-arrow contributions are summed over
				$\mI_{x,a}^{y}(\gamma)$.  Ordinary matching boundary terms cancel
				by \Cref{rem:higher-boundary-cancellation}; special-loop absorption
				terms may survive.\par\vspace{1mm}
			}
			\\[4mm]
			\hline
			
			%========================================================
			\CondCell{
				\[
				\ell(a)>1,
				\
				\ell(b)>1.
				\]
			}
			&
			\CondCell{Use only the four boundary rows of
				Table~\ref{tab:positive-positive-local-lookup}.}
			&
			\ValCell{
				Interior terms vanish.\newline
				Ordinary boundary terms cancel in the pairs
				$(\text{first-left},\text{second-right})$ and
				$(\text{first-right},\text{second-left})$ by
				\Cref{rem:higher-boundary-cancellation}.\newline
				The surviving terms are exactly the four special-loop alternatives in
				\Cref{cor:higher-long-value-reduction}.\par\vspace{1mm}
			}
			\\[4mm]
			\hline
		\end{longtable}
	\end{center}
	
	%=========================================================
	\begin{center}
		\footnotesize
		\setlength{\tabcolsep}{2pt}
		\renewcommand{\arraystretch}{1.45}
		\begin{longtable}{
				@{}
				>{\raggedright\arraybackslash}p{0.18\textwidth}
				>{\raggedright\arraybackslash}p{0.48\textwidth}
				>{\raggedright\arraybackslash}p{0.28\textwidth}
				@{}
			}
			\caption{Occurrence patterns for brackets of type $(+,+)$.}
			\label{tab:positive-positive-occurrence-lookup}
			\\
			\hline
			\textbf{Form of $\gamma$}
			&
			\textbf{Occurrence positions}
			&
			\textbf{Applicable rule}
			\\
			\hline
			\endfirsthead
			
			\hline
			\textbf{Form of $\gamma$}
			&
			\textbf{Occurrence positions}
			&
			\textbf{Applicable rule}
			\\
			\hline
			\endhead
			
			\hline
			\endfoot
			
			%========================================================
			\CondCell{No repeated arrows.}
			&
			\CondCell{
				Each support occurs at most once. Interior terms are zero; only the two boundary junctions can contribute.
			}
			&
			\ValCell{
				Use \Cref{NRA-read-off}.  If the two terms at the same junction are
				both present, they cancel by
				\Cref{rem:higher-boundary-cancellation}.\par\vspace{1mm}
			}
			\\[4mm]
			\hline
			
			%========================================================
			\CondCell{
				Pure loop power
				\[
				\gamma=\omega^{m+n+1}.
				\]
			}
			&
			\CondCell{For the only possible supports
				\[
				y=\omega^{n+1},
				\qquad
				x=\omega^{m+1},
				\]
				one has
				\[
				\Occ_\gamma(\omega^{n+1})=\{0,\ldots,m\},
				\]
				\[
				\Occ_\gamma(\omega^{m+1})=\{0,\ldots,n\}.
				\]
			}
			&
			\ValCell{
				Use
				\eqref{eq:higher-loop-trivial-arrow},
				\eqref{eq:higher-loop-arrow-trivial}, and
				\eqref{eq:higher-loop-arrow-arrow}.
			}
			\\[4mm]
			\hline
			
			%========================================================
			\CondCell{
				Primitive-cycle power
				\[
				\gamma=\tau^w,
				\  w\geq2.
				\]
				\par\vspace{0.8mm}
			}
			&
			\CondCell{
				If the first arrow of $z$ is $\tau_u$, then
				\[
				\Occ_\gamma(z)
				=
				\left\{
				u-1+k\ell(\tau)
				\ \middle|\
				\substack{
					k\in\mathbb Z_{\geq0},\\
					u-1+k\ell(\tau)+\ell(z)\leq w\ell(\tau)
				}
				\right\}.
				\]
			}
			&
			\ValCell{
				Use \Cref{prop:higher-primitive-cycle-power}.  Evaluate a full signed
				occurrence sum by \Cref{cor:signed-periodic-multiplicity}.
			}
			\\[4mm]
			\hline
			
			%========================================================
			\CondCell{
				Primitive-cycle power with suffix
				\[
				\gamma=\tau^w\tau_1\cdots\tau_r.
				\]
			}
			&
			\CondCell{
				If the first arrow of $z$ is $\tau_u$, then
				\[
				\Occ_\gamma(z)
				=
				\left\{
				u-1+k\ell(\tau)
				\ \middle|\
				\substack{
					k\in\mathbb Z_{\geq0},\\
					u-1+k\ell(\tau)+\ell(z)\leq w\ell(\tau)+r
				}
				\right\}.
				\]
			}
			&
			\ValCell{
				Use \Cref{prop:higher-primitive-cycle-suffix}.  The suffix changes
				only which occurrences and boundary positions are present.\par\vspace{1mm}
			}
			\\[4mm]
			\hline
		\end{longtable}
	\end{center}
	
	%=========================================================
	The following table is the local lookup form of \Cref{HOBF}.  The first three
	rows belong to the first insertion direction and the last three rows to the
	second insertion direction; all applicable local contributions are added.
	
	\begin{center}
		\footnotesize
		\setlength{\tabcolsep}{2pt}
		\renewcommand{\arraystretch}{1.45}
		\begin{longtable}{
				@{}
				>{\centering\arraybackslash}p{0.27\textwidth}
				>{\raggedright\arraybackslash}p{0.38\textwidth}
				>{\raggedright\arraybackslash}p{0.29\textwidth}
				@{}
			}
			\caption{Local contributions for brackets of type $(+,+)$.}
			\label{tab:positive-positive-local-lookup}
			\\
			\hline
			\textbf{Local configuration}
			&
			\textbf{Assumptions / local conditions}
			&
			\textbf{Local contribution}
			\\
			\hline
			\endfirsthead
			
			\hline
			\textbf{Local configuration}
			&
			\textbf{Assumptions / local conditions}
			&
			\textbf{Local contribution}
			\\
			\hline
			\endhead
			
			\hline
			\endfoot
			
			%========================================================
			\FigCell{
				\begin{tikzcd}[ampersand replacement=\&, column sep=3em]
					\bullet \&\& \bullet \\
					\\
					\bullet \&\& \bullet
					\arrow["{\gamma_1\cdots\gamma_n}", shift left, dotted, from=1-1, to=1-3]
					\arrow["p"', shift right, dotted, from=1-1, to=1-3]
					\arrow["a", dotted, from=1-3, to=3-1]
					\arrow["{\gamma_{n+1}=\nu}", from=1-3, to=3-3]
					\arrow["{\gamma_{n+2}\cdots\gamma_{m+n+1}}", dotted, from=3-3, to=3-1]
				\end{tikzcd}
			}
			&
			\CondCell{
				First direction, left boundary:
				\[
				i=0,\
				y=\gamma_1\cdots\gamma_{n+1},
				\]
				\[
				b=p\nu,\
				x=\nu\gamma_{n+2}\cdots\gamma_{m+n+1}.
				\]
			}
			&
			\ValCell{
				\[
				\para{\gamma}{\pi(pa)}.
				\]
			}
			\\[4mm]
			\hline
			
			%========================================================
			\FigCell{
				\begin{tikzcd}[ampersand replacement=\&, column sep=3em]
					\bullet \&\& \bullet \\
					\\
					\bullet \&\& \bullet
					\arrow["{\gamma_1\cdots\gamma_i}", dotted, from=1-1, to=1-3]
					\arrow["a"', dotted, from=1-1, to=3-1]
					\arrow["b=\nu"', shift right, from=1-3, to=3-3]
					\arrow["y", shift left, dotted, from=1-3, to=3-3]
					\arrow["{\gamma_{i+n+2}\cdots\gamma_{m+n+1}}", dotted, from=3-3, to=3-1]
				\end{tikzcd}
			}
			&
			\CondCell{
				First direction, interior:
				\[
				0<i<m,\
				y=\gamma_{i+1}\cdots\gamma_{i+n+1},
				\]
				\[
				b=\nu,\
				x=\gamma_1\cdots\gamma_i\nu\gamma_{i+n+2}\cdots\gamma_{m+n+1}.
				\]
			}
			&
			\ValCell{
				\[
				(-1)^{i(n+2)}\para{\gamma}{a}.
				\]
			}
			\\[4mm]
			\hline
			
			%========================================================
			\FigCell{
				\begin{tikzcd}[ampersand replacement=\&, column sep=3em]
					\bullet \&\& \bullet \\
					\\
					\bullet \&\& \bullet
					\arrow["{\gamma_1\cdots\gamma_m}", dotted, from=1-1, to=1-3]
					\arrow["a", dotted, from=1-1, to=3-3]
					\arrow["{\nu=\gamma_{m+1}}", from=1-3, to=3-3]
					\arrow["{\gamma_{m+2}\cdots\gamma_{m+n+1}}", shift left, dotted, from=3-3, to=3-1]
					\arrow["q"', shift right, from=3-3, to=3-1]
				\end{tikzcd}
			}
			&
			\CondCell{
				First direction, right boundary:
				\[
				i=m,\
				y=\gamma_{m+1}\cdots\gamma_{m+n+1},
				\]
				\[
				b=\nu q,\
				x=\gamma_1\cdots\gamma_m\nu.
				\]
			}
			&
			\ValCell{
				\[
				(-1)^{mn}\para{\gamma}{\pi(aq)}.
				\]
			}
			\\[4mm]
			\hline
			
			%========================================================
			\FigCell{
				\begin{tikzcd}[ampersand replacement=\&, column sep=3em]
					\bullet \&\& \bullet \\
					\\
					\bullet \&\& \bullet
					\arrow["{\gamma_1\cdots\gamma_m}", shift left, dotted, from=1-1, to=1-3]
					\arrow["p"', shift right, dotted, from=1-1, to=1-3]
					\arrow["b", dotted, from=1-3, to=3-1]
					\arrow["{\gamma_{m+1}=\nu}", from=1-3, to=3-3]
					\arrow["{\gamma_{m+2}\cdots\gamma_{m+n+1}}", dotted, from=3-3, to=3-1]
				\end{tikzcd}
			}
			&
			\CondCell{
				Second direction, left boundary:
				\[
				j=0,\
				x=\gamma_1\cdots\gamma_{m+1},
				\]
				\[
				a=p\nu,\
				y=\nu\gamma_{m+2}\cdots\gamma_{m+n+1}.
				\]
			}
			&
			\ValCell{
				\[
				-(-1)^{mn}\para{\gamma}{\pi(pb)}.
				\]
			}
			\\[4mm]
			\hline
			
			%========================================================
			\FigCell{
				\begin{tikzcd}[ampersand replacement=\&, column sep=3em]
					\bullet \&\& \bullet \\
					\\
					\bullet \&\& \bullet
					\arrow["{\gamma_1\cdots\gamma_j}", dotted, from=1-1, to=1-3]
					\arrow["b"', dotted, from=1-1, to=3-1]
					\arrow["a=\nu"', shift right, from=1-3, to=3-3]
					\arrow["x", shift left, dotted, from=1-3, to=3-3]
					\arrow["{\gamma_{j+m+2}\cdots\gamma_{m+n+1}}", dotted, from=3-3, to=3-1]
				\end{tikzcd}
			}
			&
			\CondCell{
				Second direction, interior:
				\[
				0<j<n,\
				x=\gamma_{j+1}\cdots\gamma_{j+m+1},
				\]
				\[
				a=\nu,\
				y=\gamma_1\cdots\gamma_j\nu\gamma_{j+m+2}\cdots\gamma_{m+n+1}.
				\]
			}
			&
			\ValCell{
				\[
				-(-1)^{mn+j(m+2)}\para{\gamma}{b}.
				\]
			}
			\\[4mm]
			\hline
			
			%========================================================
			\FigCell{
				\begin{tikzcd}[ampersand replacement=\&, column sep=3em]
					\bullet \&\& \bullet \\
					\\
					\bullet \&\& \bullet
					\arrow["{\gamma_1\cdots\gamma_n}", dotted, from=1-1, to=1-3]
					\arrow["b", dotted, from=1-1, to=3-3]
					\arrow["{\nu=\gamma_{n+1}}", from=1-3, to=3-3]
					\arrow["{\gamma_{n+2}\cdots\gamma_{m+n+1}}", shift left, dotted, from=3-3, to=3-1]
					\arrow["q"', shift right, from=3-3, to=3-1]
				\end{tikzcd}
			}
			&
			\CondCell{
				Second direction, right boundary:
				\[
				j=n,\
				x=\gamma_{n+1}\cdots\gamma_{m+n+1},
				\]
				\[
				a=\nu q,\
				y=\gamma_1\cdots\gamma_n\nu.
				\]
			}
			&
			\ValCell{
				\[
				-\para{\gamma}{\pi(bq)}.
				\]
			}
			\\[4mm]
			\hline
		\end{longtable}
	\end{center}
	
	Together, \Cref{tab:positive-positive-value-length-lookup,tab:positive-positive-occurrence-lookup,tab:positive-positive-local-lookup}
	provide the tabular procedure for brackets of type $(+,+)$ in
	\Cref{prop:positive-positive-classification}.
	%========================================
	\begin{example}
		\label{ex:appendix-positive-positive-lookup}
		We revisit the algebra and the data from
		\Cref{ex:suffix-arrow-reassembly}.  Recall that
		$\tau=\tau_1\tau_2\tau_3\tau_4$ and
		$\gamma=\tau^2\tau_1\tau_2\in\Gamma_{10}$.  Let
		$y=\tau\tau_1=\tau_1\tau_2\tau_3\tau_4\tau_1\in\Gamma_5$ and
		$x=\tau\tau_1\tau_2=\tau_1\tau_2\tau_3\tau_4\tau_1\tau_2\in\Gamma_6$.
		Thus $m=5$ and $n=4$.  The values are
		$a=\tau_1\rho\tau_4$ and $b=\tau_1$, so
		$f=\para{x}{a}\in B^5$ and $g=\para{y}{b}\in B^4$.
		
		We compute $l_2(f\otimes g)(\gamma)$ using the three tables above.
		
		\medskip
		\noindent
		\textup{(S1) Preliminary value-length reduction.}
		We have $\ell(a)=3$ and $\ell(b)=1$.  Thus neither value is trivial and
		$b$ is an arrow.  By
		Table~\ref{tab:positive-positive-value-length-lookup}, we must compute the
		single-arrow contributions in the first insertion direction.  Since
		$\ell(a)>1$, the second insertion direction can contribute only at its
		two boundary positions.
		
		\medskip
		\noindent
		\textup{(S2) Occurrence lookup.}
		The output support is
		$\gamma=\tau^2\tau_1\tau_2$.  By
		Table~\ref{tab:positive-positive-occurrence-lookup},
		$\Occ_\gamma(y)=\{0,4\}$ and $\Occ_\gamma(x)=\{0,4\}$.
		Replacing either occurrence of $y$ by the arrow $b=\tau_1$ produces the
		complementary support $x$, so
		$\mI_{y,\tau_1}^{x}(\gamma)=\{0,4\}$.
		
		\medskip
		\noindent
		\textup{(S3) Local-term lookup.}
		At $i=0$, the support $y$ is at the first-left boundary.  In the first
		row of Table~\ref{tab:positive-positive-local-lookup}, take $p=e_1$ and
		$\nu=\tau_1$.  The corresponding contribution is
		$\para{\gamma}{\pi(e_1a)}=\para{\gamma}{a}$.
		
		At $i=4$, we have $0<i<m$, so this is a first-interior single-arrow
		term.  Since
		$(-1)^{i(n+2)}=(-1)^{4\cdot6}=1$, the second row of
		Table~\ref{tab:positive-positive-local-lookup} contributes
		$\para{\gamma}{a}$.  Thus the first insertion direction contributes
		$2\para{\gamma}{a}$.
		
		For the second insertion direction, the occurrence $j=0$ satisfies none
		of the required local conditions and contributes zero.  At
		$j=4=n$, we have a second-right boundary configuration with
		$a=\tau_1(\rho\tau_4)=\nu q$ and $b=\tau_1$.  The last row of
		Table~\ref{tab:positive-positive-local-lookup} therefore contributes
		$-\para{\gamma}{\pi(bq)}=-\para{\gamma}{\tau_1\rho\tau_4}
		=-\para{\gamma}{a}$.
		
		\medskip
		\noindent
		\textup{(S4) Final value-length reduction.}
		The first-left contribution at $i=0$ and the second-right contribution
		at $j=n$ cancel:
		$\para{\gamma}{a}-\para{\gamma}{a}=0$, by
		\Cref{rem:higher-boundary-cancellation}.  The interior single-arrow
		contribution at $i=4$ has no matching opposite boundary term and
		therefore survives.  Consequently,
		$l_2(f\otimes g)(\gamma)=a=\tau_1\rho\tau_4\neq0$.
		
		This calculation illustrates the order of use of the three
		tables: Table~\ref{tab:positive-positive-value-length-lookup} first determines which
		local rows need to be considered,
		Table~\ref{tab:positive-positive-occurrence-lookup} gives the admissible positions,
		Table~\ref{tab:positive-positive-local-lookup} gives the corresponding local
		contributions, and
		Table~\ref{tab:positive-positive-value-length-lookup} is then used again to perform the
		final cancellation and reduction.
	\end{example}
	%================================

\end{document}